\documentclass[11pt]{amsart}

\usepackage[english]{babel}

\usepackage[letterpaper,top=2cm,bottom=2cm,left=3cm,right=3cm,marginparwidth=1.75cm]{geometry}

\usepackage{subcaption}
\usepackage{mathtools}
\usepackage{amsmath}
\usepackage{graphicx}
\usepackage{hyperref}

\usepackage{extarrows}
\usepackage{multirow}
\usepackage[alphabetic]{amsrefs}
\usepackage{arydshln}
\usepackage{cases}

\usepackage{amsfonts}
\usepackage{bm}

\usepackage{amssymb,amscd,bbm,amsthm,dsfont}
\usepackage{mathrsfs}
\usepackage{pb-diagram}
\usepackage{amssymb}
\usepackage{tikz}

\usetikzlibrary{positioning}

\usepackage[color,matrix,arrow]{xy}
\newtheorem{theorem}{Theorem}[section]

\newtheorem{proposition}[theorem]{Proposition}
\newtheorem{lemma}[theorem]{Lemma}

\theoremstyle{definition}
\newtheorem{definition}[theorem]{Definition}
\newtheorem{example}[theorem]{Example}
\newtheorem{proposition-definition}[theorem]{Proposition-Definition}
\newtheorem{corollary}[theorem]{Corollary}

\newtheorem{assumption}[theorem]{Assumption}
\newtheorem{setting}[theorem]{Setting}
\theoremstyle{remark}
\newtheorem{remark}[theorem]{Remark}

\numberwithin{equation}{section}

\def\TT{\mathbb{T}}

\def\QQ{\mathbb{Q}}
\def\Acal{\mathcal{A}}

\def\mod{\opname{mod}\nolimits}

\newcommand{\opname}[1]{\operatorname{\mathsf{#1}}}
\newcommand{\Gr}{\opname{Gr}\nolimits}
\newcommand{\Hom}{\opname{Hom}}

\newcommand{\de}{\opname{deg}}
\newcommand{\gd}{\opname{good}}

\newcommand{\Ext}{\opname{Ext}}

\newcommand{\Fac}{\opname{Fac}}
\newcommand{\Irr}{\opname{Irr}}

\newcommand{\add}{\opname{add}\nolimits}
\newcommand{\ie}{\emph{i.e.,}\ }
\newcommand{\confer}{\emph{cf.}\ }
\newcommand{\supp}{\opname{supp}}
\newcommand{\dom}{\opname{dom}}

\newcommand{\hd}{\opname{hd}}

\begin{document}

\title{F-invariant in cluster algebras}


\author{Peigen Cao}
\address{School of Mathematical Sciences, University of Science and Technology of China, Hefei, 230026, People's Republic of China}
\email{peigencao@126.com}

\subjclass[2020]{13F60, 16G20, 18M05}

\date{}

\dedicatory{Dedicated to Professor Fang Li on the occasion of his 60th birthday}


\begin{abstract}
We consider skew-symmetrizable (upper) cluster algebras with a compatible Poisson structure, called $\mathsf{\Lambda}$-(upper) cluster algebras. For any two good elements (e.g., cluster monomials) in a $\mathsf{\Lambda}$-upper cluster algebra, we introduce two invariants, called tropical invariant and $F$-invariant. We prove that (i) the product of two cluster monomials is still a cluster monomial if and only if their $F$-invariant is zero; (ii) if two cluster variables are log-canonical, then they are contained in the same cluster; and (iii) the notion of $F$-invariant for a pair of cluster monomials can be defined for any (upper) cluster algebra, 
regardless of whether it is a $\mathsf{\Lambda}$-(upper) cluster algebra.

When restricting to cluster monomials, we prove that the tropical invariant and $F$-invariant respectively coincide with the $\Lambda$-invariant and twice $\mathfrak{d}$-invariant in the monoidal cluster categorification using various monoidal subcategories of finite-dimensional modules over quantum affine algebras and quiver Hecke algebras; and we prove that the $F$-invariant coincides with the $E$-invariant in the additive cluster categorification using the theory of quivers with potentials.

Inspired by $F$-invariant, we introduce the dominant sets for seeds of cluster algebras as a replacement of torsion classes for $\tau$-tilting pairs in $\tau$-tilting theory. With the help of the dominant sets, we prove that the oriented exchange graphs of cluster algebras are acyclic. In particular, this implies that green mutations induce a partial order on the set of seeds (up to seed equivalence) of cluster algebras. We prove that the oriented exchange graphs of cluster algebras coincide with the Hasse quivers of the above posets of seeds.
\end{abstract}

\maketitle

\setcounter{tocdepth}{1}
\tableofcontents

\section{Introduction}
\subsection{Backgrounds} Cluster algebras were introduced by Fomin and Zelevinsky \cite{fz_2002}  as a combinatorial approach to the dual canonical bases (or upper global bases) of quantum groups and to the theory of total positivity in algebraic groups. Such algebras often arise as the coordinate rings of spaces, such as double Bruhat cells \cite{bfz_2005},  unipotent cells \cite{gls_2011}, double Bott-Samelson cells \cite{Shen-Weng-2021} and arise as the Grothendieck rings of certain monoidal subcategories of the representations of quantum affine algebras \cite{HL_2010,HL_2013,kkop-2024} and quiver Hecke algebras \cite{kkko-2018}.

A cluster algebra is a $\mathbb Z$-subalgebra of a rational function field generated by a special set of
generators called {\em cluster variables}, which are grouped into overlapping subsets,
called {\em clusters}. A {\em seed} is a pair consisting of a cluster and a rectangular
integer matrix with skew-symmetrizable principal part. One can obtain new seeds from a given one by a procedure called {\em mutation}.
The sets of cluster variables and clusters of a cluster algebra are determined by an initial seed and the iterative mutations.
 A {\em cluster momomial} is a monomial in variables from the same cluster. In the main text, we actually use a broader definition of cluster monomials, which allows the presence of the inverse of {\em frozen variables}.

A fundamental problem in cluster algebras is to characterize when two cluster variables are contained in the same cluster. A more general problem is to characterize when the product of two cluster monomials is still a cluster monomial. Many progresses have been made by introducing different type of compatibility degrees or invariants.

\subsubsection*{Compatibility degrees in cluster algebras} In the study of Y-systems and generalized associahedra, Fomin and Zelevinsky \cite{fz-2003y} introduce a function 
$$(-\mid\mid-) \colon\Phi_{\geq -1}\times \Phi_{\geq -1}\rightarrow\mathbb Z_{\geq 0}$$ defined on the set $\Phi_{\geq -1}$ of almost positive roots associated to a Cartan matrix $C$ of finite type.  Let $\mathcal A$ be a cluster algebra of finite type with the same Cartan-Killing type as the Cartan matrix $C$. Fomin and Zelevinsky \cite{fz_2003f} proved that there is a bijection 
from $\Phi_{\geq -1}$ to the set  $\mathcal X$ of (unfrozen) cluster variables of $\mathcal A$ and they proved that two cluster variables $x_{\alpha}$ and $x_{\beta}$ are contained in the same cluster  if and only if their corresponding roots $\alpha,\beta\in \Phi_{\geq -1}$ are {\em compatible},  \ie $(\alpha\mid\mid \beta)=0$.

Now let $\mathcal A$ be any skew-symmetrizable cluster algebra (not necessarily finite type). In \cite{cao-li-2020}, Li and the author of the present paper proved that there exists a function $$(-\mid\mid  -)_d \colon \mathcal X\times \mathcal X \rightarrow\mathbb Z_{\geq -1}$$ defined on the set of  (unfrozen) cluster variables $\mathcal X$ of $\mathcal A$, which is called the {\em $d$-compatibility degree}. The values of the $d$-compatibility degree $(-\mid\mid  -)_d$ are given by the components of the $d$-vectors \cite{fomin_zelevinsky_2007} in cluster algebras. It is proved that two cluster variables $x,z\in\mathcal X$ are contained in the same cluster if and only if $$(x\mid\mid z)_d=0\;\;\;\text{or}\;\;\;(x\mid\mid z)_d=-1,$$
moreover, $(x\mid\mid z)_d=-1$ if and only if $x=z$. This confirms a conjecture on $d$-vectors proposed by Fomin and Zelevinsky \cite[Conjecture 7.4]{fomin_zelevinsky_2007}.

Recently, Fu and Gyoda \cite{Fu-Gyoda} introduced the {\em $f$-compatibility degree} in cluster algebras using the components of $f$-vectors, which is a function
$$(-\mid\mid-)_f \colon \mathcal X\times \mathcal X\rightarrow \mathbb Z_{\geq 0}$$
defined on the set $\mathcal X$ of  (unfrozen) cluster variables of $\mathcal A$.
Based on the results in \cite{cao-li-2020}, Fu and Gyoda  proved that two cluster variables $x,z\in\mathcal X$ are contained in the same cluster if and only if $(x\mid\mid z)_f=0$. When restricted to cluster algebras of finite type, the $f$-compatibility degree coincides with Fomin-Zelevinsky's compatibility degree.

\subsubsection*{$\Lambda$-invariant and $\mathfrak{d}$-invariant} The 
$\Lambda$-invariant and $\mathfrak{d}$-invariant are introduced by Kang, Kashiwara, Kim and Oh \cite{kkko-2018} in the representation theory of quiver Hecke algebras and by Kashiwara, Kim, Oh and Park \cite{kkop-2020} in the representation theory of quantum affine algebras in their study of the monoidal categorification of (quantum) cluster algebras.

Here we briefly recall these notions in the quiver Hecke algebras case and we mainly follow \cite{kkko-2018,kk_2019}. Let $$R\text{-gmod}=\bigoplus_{\beta\in Q^+}R(\beta)\text{-gmod}$$
 be the category of finite-dimensional graded modules with
homomorphisms of degree $0$ given in \cite[Page 2270]{kk_2019}, which is a monoidal abelian category with the tensor product ${\circ}$ given there. Denote by
 $q$ the grading shift functor in $R\text{-gmod}$: for $M\in R\text{-gmod}$, we have $(qM)_i=M_{i-1}$ for $i\in\mathbb Z$. For each pair $(M,N)$ with $0\neq M\in R(\beta)\text{-gmod}$ and $0\neq N\in R(\gamma)\text{-gmod}$, 
 Kang, Kashiwara, Kim and Oh \cite[Definition 2.2.1]{kkko-2018} constructed a renormalized $r$-matrix ${\bf r}_{M,N}$, which is a morphism from $M{\circ}N$ to $q^{c}N{\circ} M$ in $R\text{-gmod}$ for some integer $c$. The {\em $\Lambda$-invariant} $\Lambda(M,N)$ of the pair $(M,N)$ is defined as the homogeneous degree of ${\bf r}_{M,N}$,  \ie
  $$\Lambda(M,N)=-c\;\in\mathbb Z.$$
The {\em $\mathfrak{d}$-invariant} $\mathfrak{d}(M,N)$ of the pair $(M,N)$ is defined as half of the symmetrized sum:
  $$\mathfrak{d}(M,N)=\frac{1}{2}(\Lambda(M,N)+\Lambda(N,M)).$$
It turns out that $\mathfrak{d}(M,N)\in\mathbb Z_{\geq 0}$, \confer \cite[Lemma 3.2.1]{kkko-2018}.

\begin{remark}
Denote by $R\text{-mod}^0$ the category obtained from $R\text{-gmod}$ by forgetting the gradings.  By construction of the renormalized $r$-matrices, we have \[
\Lambda(q^aM,q^bN)=\Lambda(M,N)\;\;\;\text{and}\;\;\;\mathfrak{d}(q^aM,q^bN)=\mathfrak{d}(M,N),\;\;\forall a,b\in\mathbb Z.
\]
Thus the $\Lambda$-invariant and $\mathfrak{d}$-invariant descend to the ungraded category $R\text{-mod}^0$.
\end{remark}

A simple object $S$ in a monoidal abelian category is said to be {\em real}, if the tensor product of $S$ with itself is still simple. Kang, Kashiwara, Kim and Oh proved that if $M$ is a real simple object in $R\text{-gmod}$, then $\mathfrak{d}(M,M)=0$, \confer \cite[Lemma 3.2.3]{kkko-2018}. In monoidal categorification of cluster algebras, each cluster monomial corresponds to a real simple object. In particular, the $\mathfrak{d}$-invariant associated to a  ``cluster monomial" is zero.

Now let us turn to the $\mathfrak{d}$-invariant in the representation theory of  quantum affine algebras. In this setting,  Kashiwara, Kim, Oh and Park \cite{kkop-2020} proved that (i) $\mathfrak{d}(M,N)$ takes non-negative values; (ii) if $M$ is a real simple object, then $\mathfrak{d}(M,M)=0$, \confer \cite[Proposition 3.16, Corollary 3.17]{kkop-2020}.

\subsubsection*{$E$-invariant}
The $E$-invariant is introduced by Derksen, Weyman and Zelevinsky \cite{DWZ10} in the additive categorification of cluster algebras using decorated representations of quivers with potentials. 
For any two decorated representations $\mathcal M$ and $\mathcal N$ of a quiver with potential $(Q,W)$,
 Derksen, Weyman and Zelevinsky \cite[Section 7]{DWZ10} defined an integer $E^{\rm inj}(\mathcal M,\mathcal N)$, which is called the {\em partial $E$-invariant} of the pair $(\mathcal M,\mathcal N)$ in this paper (see Definition \ref{def:g-F-E}).  The {\em $E$-invariant} $E^{\rm sym}(\mathcal M,\mathcal N)$ of the pair $(\mathcal M,\mathcal N)$ is  defined to be the symmetrized sum: 
 $$E^{\rm sym}(\mathcal M,\mathcal N)=E^{\rm inj}(\mathcal M,\mathcal N)+E^{\rm inj}(\mathcal N,\mathcal M)\;\;\in\mathbb Z.$$
It is known from \cite[Section 10]{DWZ10} (or \cite[Section 3]{CLS-2015}) that
\[E^{\rm inj}(\mathcal M,\mathcal N)\geq 0\;\;\;\text{and}\;\;\;
E^{\rm sym}(\mathcal M,\mathcal N)\geq 0.
\]
  Derksen, Weyman and Zelevinsky  \cite[Corollary 7.2]{DWZ10} proved that  $E^{\rm sym}(\mathcal M,\mathcal M)=0$ for any negative-reachable decorated representations,  \ie those decorated representations  corresponding to cluster monomials.

Inspired by Derksen-Weyman-Zelevinsky's $E$-invariant for quivers with potentials, \ie for Jacobian algebras, the $E$-invariant for any finite-dimensional algebra is introduced by  Derksen and Fei \cite{DK-2015} in the study of general presentations and by Adachi, Iyama and Reiten  \cite{air_2014} in the study of $\tau$-tilting theory.

\subsection{Results (I): Tropical invariant and $F$-invariant}
We consider skew-symmetrizable (upper) cluster algebras with a compatible Poisson structure, called {\em $\mathsf{\Lambda}$-(upper) cluster algebras}. In a $\mathsf{\Lambda}$-upper cluster algebra, each seed $({\bf x}_t,(\widetilde B_t)_{m\times n})$ is endowed with a skew-symmetric integer matrix $(\Lambda_t)_{m\times m}$, called {\em Poisson coefficient matrix} such that 
\begin{eqnarray}\label{eqn:S-0}
   \widetilde B_t^T\Lambda_t=(S\mid {\bf 0}) 
\end{eqnarray} for some fixed diagonal matrix $S=diag(s_1,\ldots,s_n)$ with $s_i>0$. We call the triple $({\bf x}_t,\widetilde B_t,\Lambda_t)$ a {\em $\mathsf{\Lambda}$-seed} of $\mathcal U$.

For any two good elements (e.g., cluster monomials) in a $\mathsf{\Lambda}$-upper cluster algebra, we introduce two invariants, called tropical invariant and $F$-invariant, which generalize the  $\Lambda$-invariant and $\mathfrak{d}$-invariant in the monoidal cluster categorification. The $F$-invariant also generalizes the $E$-invariant in the additive cluster categorification.

A {\em good element} is a compatibly pointed element in the sense of \cite{qin_2019} satisfying the universally positive property (see Definition \ref{def:pointed}). We view such elements as a replacement of decorated representations of quivers with potentials and finite-dimensional graded modules of quiver Hecke algebras. 

The Laurent expansion of a good element $u\in\mathcal U$ with respect to any seed $({\bf x}_w,\widetilde B_w)$ of $\mathcal U$ has a canonical expression:
\[ u={\bf x}_w^{{\bf g}_u^w}F_u^w(\hat y_{1;w},\ldots,\hat y_{n;w})
\]
where ${\bf g}_u^w\in\mathbb Z^m$ and $F_u^w\in\mathbb Z_{\geq 0}[y_1,\ldots,y_n]$ is a polynomial with constant term $1$. They are respectively called the {\em extended $g$-vector} and {\em $F$-polynomial} of $u$ with respect to seed $({\bf x}_w,\widetilde B_w)$.

For each good element $u$ in a $\mathsf{\Lambda}$-upper cluster algebra $\mathcal U$, we introduce a semifield homomorphism 
 \[\beta_{u} \colon \mathbb F_{>0} \coloneqq   \mathbb Q_{\rm sf}(x_{1;t_0},\ldots,x_{m;t_0})\rightarrow \mathbb Z^{\rm max},\]
 where $\mathbb F_{>0}$ is the universal semifield generated by the initial cluster variables of $\mathcal U$ and $\mathbb Z^{\max}=(\mathbb Z,+,\max)$ is the tropical semifield (see Proposition \ref{pro:beta-map}). Using such semifield homomorphisms, we introduce the tropical invariant and $F$-invariant.

\begin{definition}[Tropical invariant and $F$-invariant]
Let $u$ and $u'$ be two good elements in a $\mathsf{\Lambda}$-upper cluster algebra $\mathcal U$. The {\em tropical invariant} $\langle u,u'\rangle_{\rm trop}$ and {\em $F$-invariant} $(u\mid\mid u')_F$ of the pair $(u,u')$ are defined as follows:
\[\langle u,u'\rangle_{\rm trop} \coloneqq   \beta_{u'}(u)\;\;\;\text{and}\;\;\;(u\mid\mid u')_F=\beta_{u'}(u)+\beta_u(u')=\langle u,u'\rangle_{\rm trop}+\langle u',u\rangle_{\rm trop}.
\]
\end{definition}

 Given a non-zero polynomial $$F=\sum_{{\bf v}\in\mathbb N^n}c_{{\bf v}}{\bf y}^{{\bf v}}\in\mathbb Z[y_1,\ldots,y_n]$$ and a vector ${\bf r}\in \mathbb Z^n$, we denote by
$$F[{\bf r}] \coloneqq   \max\{{\bf v}^T{\bf r}\mid c_{{\bf v}}\neq 0\}\;\in\mathbb Z.$$
It is easy to see if $F$ has the constant term $1$, then $F[{\bf r}]\in\mathbb Z_{\geq 0}$ for any ${\bf r}\in\mathbb Z^n$.

The first main result in this paper is that for each $\mathsf{\Lambda}$-seed $({\bf x}_w,\widetilde B_w, \Lambda_w)$ of $\mathcal U$, we have an explicit formula to calculate the tropical invariant and $F$-invariant.  
\begin{theorem} [{Theorem \ref{thm:mutation-inv}}]
\label{intro:thm 1.2}
Let $u$ and $u'$ be two good elements in a $\mathsf{\Lambda}$-upper cluster algebra $\mathcal U$. Let
\[ u={\bf x}_w^{{\bf g}_u^w}F_u^w(\hat y_{1;w},\ldots,\hat y_{n;w}),\;\;\;u'={\bf x}_w^{{\bf g}_{u'}^w}F_{u'}^w(\hat y_{1;w},\ldots,\hat y_{n;w})
\]
be the canonical expressions of $u$ and $u'$ with respect to any vertex $w\in \mathbb T_n$. Then we have
\vspace{1mm}
\begin{itemize}
    \item [(i)] $\langle u,u'\rangle_{\rm trop}= ({\bf g}_u^w)^T\Lambda_w{\bf g}_{u'}^w+F_u^w[(S\mid {\bf 0}){\bf g}_{u'}^w]$, where $(S\mid {\bf 0})$ is as in \eqref{eqn:S-0}.
    \vspace{2mm}
    \item[(ii)] $(u\mid\mid u')_F= F_u^w[(S\mid {\bf 0}){\bf g}_{u'}^w]+F_{u'}^w[(S\mid {\bf 0}){\bf g}_{u}^w]$. In particular,  $(u\mid\mid u')_F\in\mathbb Z_{\geq 0}$. \vspace{2mm}
    \item[(iii)] $(u\mid\mid u)_F=0$ whenever $u$ is a cluster monomial.
\end{itemize}
\end{theorem}

\begin{remark}
    From the formula in (i), we see that the tropical invariant heavily depends on the Poisson coefficient matrix $\Lambda_w$. Since we define $F$-invariant by taking the symmetrized sum of tropical invariants, a prior, the $F$-invariant should also heavily depend on the Poisson coefficient matrix $\Lambda_w$. However, if we look at the formula in (ii), this formula eliminates the dependence of $(u\mid\mid u')_F$ on the Poisson coefficient matrix and it only depends on the notion of $g$-vectors and $F$-polynomials.  From this viewpoint,   {\em $F$-invariant can be defined for any skew-symmetrizable (upper) cluster algebra with trivial coefficients} if we restrict to cluster monomials (see Proposition \ref{pro:trivial-coef}). 
    The advantange to work on $\mathsf{\Lambda}$-(upper) cluster algebras is that we can define $F$-invariant for any pair of good elements and go beyond the class of cluster monomials.
\end{remark}

Since the $F$-invariant $(u\mid\mid u')_F$ can be written as 
\[(u\mid\mid u')_F= F_u^w[(S\mid {\bf 0}){\bf g}_{u'}^w]+F_{u'}^w[(S\mid {\bf 0}){\bf g}_{u}^w],\] 
we call the non-negative integer $F_u^w[(S\mid {\bf 0}){\bf g}_{u'}^w]$  the {\em partial $F$-invariant} of the pair $(u,u')$ at vertex $w\in\mathbb T_n$.

As the $F$-invariant  $(u\mid\mid u')_F$ always takes non-negative values, it is natural to ask when $(u\mid\mid u')_F=0$ holds. 

\begin{theorem}[Theorem \ref{pro:uu'}] Let $u$ and $u'$ be two cluster monomials in an upper cluster algebra. Then  $(u\mid\mid u')_F=0$ if and only if the product $u u'$ is still a cluster monomial.
\end{theorem}

 For arbitrary two good elements $u$ and $u'$ (not necessarily cluster monomials),  we give a necessary condition for $(u\mid\mid u')_F=0$ in Theorem \ref{thm:sign-coherent}. Before stating the result, we need some preparations.

Recall that  $m$ is the number of cluster variables (frozen and unfrozen) in a cluster and  $n$ is the number of unfrozen cluster variables in a cluster. By Definition \ref{def:pointed}, each good element $u$ corresponds to a tropical point $[{\bf g}]=\{{\bf g}^t\in\mathbb Z^m\mid t\in\mathbb T_n\}$ (associated to a $Y$-pattern). 
Two tropical points $[{\bf g}]=\{{\bf g}^t\in\mathbb Z^m\mid t\in\mathbb T_n\}$ and $[{\bf g}']=\{({\bf g}')^{t}\in\mathbb Z^m\mid t\in\mathbb T_n\}$ are said to be {\em sign-coherent} if for any vertex $t\in\mathbb T_n$ and any $k\in[1,n]$, we have $g_k^t\cdot (g_k^\prime)^{t}\geq 0$, where $g_k^t$ and $(g_k^\prime)^{t}$ are the $k$th components of  ${\bf g}^t$ and $({\bf g}')^{t}$. 

\begin{theorem}[Theorem \ref{thm:sign-coherent}] Let $u$ and $u'$ be two good elements in a $\mathsf{\Lambda}$-upper cluster algebra $\mathcal U$. If $(u\mid\mid u')_F=0$, then 
the tropical points correspond to $u$ and $u'$ are sign-coherent and the product $uu'$ is still a good element in $\mathcal U$.
\end{theorem}

Since we work on a $\mathsf{\Lambda}$-upper cluster algebra $\mathcal U$, it has a natural compatible Poisson structure. By definition, if two cluster variables are in the same cluster, then they are log-canonical with respect to the compatible Poisson structure. We prove that the converse statement is also true.

 \begin{theorem}[Theorem \ref{thm:log}] Let $\mathcal U$ be a $\mathsf{\Lambda}$-upper cluster algebra. Then the following statements hold.
 \begin{itemize}
     \item [(i)] Let $u\in\mathcal U$ be a good element and $x\in\mathcal U$ an unfrozen cluster variable. If $u$ and  $x$ are log-canonical, then $(x\mid\mid u)_F=0$.
     \item[(ii)] If two cluster variables are log-canonical, then they are contained in the same cluster.
 \end{itemize}
 \end{theorem}

\subsection{Results (II): Comparison with some categorical invariants}
As we have seen, the $F$-invariant shares the similar properties with  the $E$-invariant in the additive cluster categorification and the $\mathfrak{d}$-invariant in the monoidal cluster categorification. It is natural to compare them.

We first focus on the  monoidal cluster categorification \cite{HL_2010,HL_2013}.  A finite length, $\mathbf{k}$-linear, monoidal abelian category $\mathcal C$ is a {\em monoidal categorification} of a cluster algebra $\mathcal A^+$ (with non invertible frozen variables), if there exists a $\mathbb Z$-algebra isomorphism
  $$\varphi\colon\mathcal K_0(\mathcal C) \rightarrow \mathcal A^+$$ 
 from the Grothendieck ring $\mathcal K_0(\mathcal C)$ to the cluster algebra $\mathcal A^+$
  such that the cluster monomials of $\mathcal A^+$ correspond to a subset of the simple objects (up to isomorphism) in $\mathcal C$. Such simple objects are said to be {\em reachable}. We denote by $(\mathcal A^+, \mathcal C, \varphi)$ for this monoidal categorification.

In order to unify the definition of $\Lambda$-invariant and $\mathfrak{d}$-invariant in different settings and unify the proofs for our results, we introduce the notion of $\mathsf{\Lambda}$-structure on a monoidal (abelian) category, which extracts the key properties of $\Lambda$-invariant appearing in the representation theory of quiver
Hecke algebras \cite{kkko-2018} and that of quantum affine algebras \cite{kkop-2020}.

Roughly speaking, a {\em $\mathsf{\Lambda}$-structure} on a  monoidal category $\mathcal C$ is a pair $\mathsf{\Lambda}=(\Omega,\Lambda)$, where 
\begin{itemize}
    \item $\Omega$ is a class of non-zero objects in $\mathcal C$ containing all the simple objects and closed under
 isomorphism,  taking tensor products and non-zero subquotients;
 \item ${\Lambda}:\Omega\times \Omega\rightarrow\mathbb Z$ is a map satisfying certain axiomatic conditions (see Definition \ref{def:L-structure}).
\end{itemize}
If  $\mathsf{\Lambda}=(\Omega,\Lambda)$ is a $\mathsf{\Lambda}$-structure on a monoidal category $\mathcal C$, we simply say that $\Lambda\colon\Omega\times \Omega\rightarrow\mathbb Z$ is a $\mathsf{\Lambda}$-structure on  $\mathcal C$. For $M,N\in\Omega$, the numbers  $\Lambda(M,N)$ and 
     \[\mathfrak{d}(M,N) \coloneqq   \frac{\Lambda(M,N)+\Lambda(N,M)}{2}\]
     are respectively called  the {\em ${\Lambda}$-invariant} and
   {\em $\mathfrak{d}$-invariant} of the pair $(M,N)$.

\begin{definition}[$\mathsf{\Lambda}$-monoidal categorification\footnote{We borrow this teminology from \cite[Definition 6.7]{kkop-2020}.}] 
Let $(\mathcal A^+,\mathcal C,\varphi)$ be a monoidal categorification and  $\Lambda\colon\Omega\times \Omega\rightarrow\mathbb Z$ a $\mathsf{\Lambda}$-structure on $\mathcal C$. We say that  $(\mathcal A^+,\mathcal C,\varphi, \Omega, \Lambda)$ is a {\em ${\mathsf{\Lambda}}$-monoidal categorification} if the following conditions hold. 
\begin{itemize}
    \item [(c1)] Let $M, N\in\mathcal C$ be simple objects such that at least one of them is real. Then the head $M\nabla N$ of $M\otimes N$ is simple. 
    \item[(c2)] The $\mathsf{\Lambda}$-structure and cluster structure on $\mathcal C$ are compatible in the sense that there exists a monoidal cluster $\mathcal M_w=(M_{1;w},\ldots, M_{m;w})$ in $\mathcal C$ such that  $(\widetilde B_w, \Lambda_w)$ forms a compatible pair, where $\widetilde B_w$ comes from the cluster structure and $\Lambda_{w}=(\lambda_{ij;w})$ with $\lambda_{ij;w}=\Lambda(M_{i;w}, M_{j;w})$ comes from the $\mathsf{\Lambda}$-structure.
\end{itemize}
\end{definition}
The examples of $\mathsf{\Lambda}$-monoidal categorifications contain those monoidal categorifications constructed by  Kang,  Kashiwara, Kim, and  Oh  \cite{kkko-2018} 
using monoidal subcategories of representations of (symmetric) quiver Hecke algebras
    and  by Kashiwara, Kim, Oh, and  Park \cite{kkop-2024} using monoidal subcategories of representations of quantum affine algebras. 

Thanks to Proposition \ref{prop:B-L-t}, the cluster algebra $\mathcal A^+$ above, together with the family of skew-symmetric matrices $\{\Lambda_t\mid t\in\mathbb T_n\}$  forms a $\mathsf{\Lambda}$-cluster algebra. Thus the tropical invariant $\langle u,u'\rangle_{\rm trop}$ is defined for any pair of good elements in $\mathcal A^+$.

\begin{theorem}[Theorem \ref{thm:trop-Lambda}] \label{thm:introd-trop-Lambda}
Let  $(\mathcal A^+,\mathcal C,\varphi, \Omega, \Lambda)$ be a $\mathsf{\Lambda}$-monoidal categorification. Let $M,N$ be two simple objects in $\mathcal C$ such that $\varphi(M)$ and $\varphi(N)$ are good elements in $\mathcal A^+$. If $M$ or $N$ is reachable,
then we have
    \[\Lambda(M,N)=\langle \varphi(M), \varphi(N)\rangle_{\rm trop}\;\;\;\text{and}\;\;\;2\mathfrak{d}(M,N) = (\varphi(M)\mid\mid \varphi(N))_F.\] 
 \end{theorem}
Since the reachable simple objects correspond to cluster monomials, which are always good elements in $\mathcal A^+$, we have the following consequence.
\begin{corollary}
    Let  $(\mathcal A^+,\mathcal C,\varphi, \Omega, \Lambda)$ be a $\mathsf{\Lambda}$-monoidal categorification. If $M$ and $N$ are reachable simple objects, then we have 
    \[\Lambda(M,N)=\langle \varphi(M), \varphi(N)\rangle_{\rm trop}\;\;\;\text{and}\;\;\;2\mathfrak{d}(M,N) = (\varphi(M)\mid\mid \varphi(N))_F.\] 
\end{corollary}

As mentioned in \cite[Pages 1040--1041]{kkop-2020}, the integer-valued invariants $\Lambda(M,N)$ and  $\mathfrak{d}(M,N)$ provide important information in the representation theory of quiver Hecke algebra. However, in general, computing those values is quite difficult. Thanks to  Theorem \ref{thm:introd-trop-Lambda} and Theorem \ref{intro:thm 1.2}, we obtain explicit formulas expressing the $\Lambda$-invariant and $\mathfrak{d}$-invariant via $g$-vectors and $F$-polynomials, which can be easily implemented by a computer
program.

Now, let us turn to the additive cluster categorification \cite{DWZ08, DWZ10}.
\begin{theorem}[Theorem \ref{thm:F=E}]
Let $(Q,W)$ be a quiver with non-degenerate potential and $\mathcal A_Q$ the corresponding cluster algebra with trivial coefficients. 
Let $u, u'$ be two cluster monomials of $\mathcal A_Q$ and let $\mathcal M_u, \mathcal M_{u'}$ be the negative-reachable decorated representations of $(Q,W)$ corresponding to $u$ and $u'$. Then $$(u\mid\mid u')_F=E^{\rm sym}(\mathcal M_u,\mathcal M_{u'}).$$
\end{theorem}
Notice that in the above theorem we take $S=I_n$ to be the fixed skew-symmetrizer for the exchange matrices of $\mathcal A_Q$, when we consider the $F$-invariant for cluster monomials in $\mathcal A_Q$.
\subsection{Results (III): Dominant sets and oriented exchange graphs}  Two seeds of a cluster algebra are {\em equivalent}, if they are the same up to relabeling. The {\em exchange graph} $\mathcal H=\mathcal H(\mathcal A)$ of a cluster algebra $\mathcal A$ is the graph whose vertices correspond to the seeds (up to seed equivalence) of $\mathcal A$ and whose edges correspond to seed mutations.
Thanks to the sign-coherence of $C$-matrices \cite{GHKK18}, each mutation has two states: green mutation or red mutation. Fix an initial seed at vertex $t_0\in\mathbb T_n$,  
 the {\em oriented exchange graph} $\overrightarrow{\mathcal H}^{t_0}$ of $\mathcal A$ is the quiver whose underlying graph is the exchange graph $\mathcal H$ of $\mathcal A$ and whose orientation is induced by the green mutations.

$\tau$-tilting theory is introduced by Adachi, Iyama and Reiten \cite{air_2014}, which completes (classical) tilting theory from the viewpoint of mutation. Various of the fundamental concepts in cluster algebras, such as clusters/seeds, mutations, $C$-matrices, $G$-matrices, (oriented) exchange graphs,
etc., were generalized to $\tau$-tilting theory.

The seeds in cluster algebras correspond to (basic) $\tau$-tilting pairs in $\tau$-tilting theory, which are in bijection with the functorially finite torsion classes in the module categories.  The inclusion of subcategories induces a natural partial order on the set of functorially finite torsion classes and thus a partial order on the set of $\tau$-tilting pairs (up to isomorphism).

Thanks to the existence of the partial order,  the oriented exchange graphs in  $\tau$-tilting theory are acyclic. Actually,  Adachi, Iyama and Reiten \cite[Theorem 0.6]{air_2014}  proved that the oriented exchange graphs in $\tau$-tilting theory coincide with the Hasse quivers of the posets of functorially finite torsion classes, equivalently, the posets of $\tau$-tilting pairs.

It is natural to ask whether the oriented exchange graphs of (skew-symmetrizable) cluster algebras are acyclic and whether these oriented exchange graphs coincide with
Hasse quivers of some posets.

Inspired by $F$-invariant, we introduce the dominant sets for seeds of cluster algebras as a replacement of torsion classes for $\tau$-tilting pairs in  $\tau$-tilting theory. With the help of the dominant sets, we prove the following results on oriented exchange graphs.

\begin{theorem}[Theorems \ref{thm:graph}, \ref{thm:hasse}, Corollary \ref{cor:hasse}]
The following statements hold.
\begin{itemize}
    \item [(i)] The oriented exchange graphs of cluster algebras are acyclic. In particular, green mutations induce a partial order on the set of seeds (up to seed equivalence) of cluster algebras.
    \item[(ii)]  The oriented exchange graphs of cluster algebras coincide with the Hasse quivers of the  posets of seeds of cluster algebras.
\end{itemize}
\end{theorem}

We would like to point out that by ignoring initial cluster variables, the dominant sets in cluster algebras correspond to the sets of indecomposable $\tau$-rigid modules contained in the functorially finite torsion classes (see Proposition \ref{pro:dom-tau}). In this sense, we view the dominant sets as a replacement of torsion classes in $\tau$-titling theory.

\subsection{Conclusion and future work} 
In this paper, we introduce the tropical invariant and $F$-invariant for ($\mathsf{\Lambda}$)-cluster algebras with a focus on $F$-invariant. Up to a constant rescaling factor, we prove that the following three invariants coincide on cluster monomials: (i) the $\mathfrak{d}$-invariant in  $\mathsf{\Lambda}$-monoidal cluster categorification; (ii) the $F$-invariant in  cluster algebras; (iii) the $E$-invariant in  additive cluster categorification. Notice that these three invariants are symmetric, e.g., $(u\mid\mid u')_F=(u'\mid\mid u)_F$.
We prove the oriented exchange graphs of cluster algebras are acyclic by introducing the dominant sets for seeds of cluster algebras as a replacement of torsion classes for $\tau$-tilting pairs in $\tau$-tilting theory.

Now let us consider those $\mathsf{\Lambda}$-monoidal categorifications $(\mathcal A^+,\mathcal C,\varphi, \Omega, \Lambda)$ constructed by Kashiwara, Kim, Oh and Park \cite{kkop-2024} using the representations of quantum affine algebras. For a pair $(M,N)$ of simple objects in $\mathcal C$, it is known \cite[Proposition 3.16]{kkop-2020} that the $\mathfrak{d}$-invariant can be written as 
\[\mathfrak{d}(M,N)  = \mathfrak{o}(M,N)+\mathfrak{o}(N,M),\]
where $\mathfrak{o}(M,N)\in\mathbb Z_{\geq 0}$ is the pole order of the normalized $R$-matrix $R^{\mathrm{norm}}_{M,N_z}$ at $z=1$. We may call $\mathfrak{o}(M,N)$ the {\em partial $\mathfrak{d}$-invariant} of the  pair $(M,N)$.
It is an important quantity in itself, particularly as a key ingredient in the construction of the generalized quantum affine Schur--Weyl duality functors due to \cite{KKK-2018R}.

In an ongoing work\footnote{This work will be an extended version of the note \cite{cao_2025}.} with Ryo Fujita and Kota Murakami,  we show that the following three (non-symmetric) partial invariants coincide on cluster monomials with a good choice of the reference seed: (i) the partial $\mathfrak{d}$-invariant in  $\mathsf{\Lambda}$-monoidal cluster categorification using representations of quantum affine algebras; (ii) the partial $F$-invariant in cluster algebras; (iii) the partial $E$-invariant in additive cluster categorification.  Theorem \ref{thm:introd-trop-Lambda} will play an important role in the proof of the coincidence of the partial $F$-invariant and partial $\mathfrak{d}$-invariant. Such results will be used to  prove the conjectural unified denominator formula \cite[Conjecture 6.7]{Fujita-Oh-2021} (equivalently, \cite[Conjecture 5.17]{Fujita-Murakami-2023}) for the normalized $R$-matrices between the KR-modules.

\subsection{Organization} In Section \ref{sec2}, we recall some well-known concepts in cluster algebras, including $Y$-patterns, cluster patterns, compatible pairs, and $\mathsf{\Lambda}$-(upper) cluster algebras.

In Section \ref{sec3}, we recall the definition of tropical points for $Y$-patterns and cluster patterns, and compatibly pointed elements in upper cluster algebras. We give the definition of good elements for upper cluster algebras of full rank. 

In Section \ref{sec4}, for each good element in a $\mathsf{\Lambda}$-upper cluster algebra, we introduce  a semifield homomorphism (see Proposition \ref{pro:beta-map}) and we define the tropical invariant and $F$-invariant using such semifield homomorphisms.  The main results in this section are Theorem \ref{thm:mutation-inv}, Theorem \ref{pro:uu'} , Theorem \ref{thm:sign-coherent}, Theorem \ref{thm:log}. Notice that $F$-invariant can be defined for any skew-symmetrizable (upper) cluster algebra with trivial coefficients, when we only consider cluster monomials (see Proposition \ref{pro:trivial-coef}).

In Section \ref{sec:monoidal} and Section \ref{sec:E-invariant}, we compare the tropical invariant and $F$-invariant with some known invariants in additive/monoidal cluster categorification. 
 When restricting to cluster monomials, we prove that the tropical invariant and $F$-invariant respectively coincide with the $\Lambda$-invariant and twice  $\mathfrak{d}$-invariant in $\mathsf{\Lambda}$-monoidal categorification (see Theorem \ref{thm:trop-Lambda}) and we prove that the $F$-invariant coincides with the $E$-invariant in additive cluster categorification (see Theorem \ref{thm:F=E}).

In Section \ref{sec-dom-sets}, we introduce the dominant sets for the seeds of cluster algebras as a replacement of torsion classes for $\tau$-tilting pairs in $\tau$-tilting theory. Then we prove the results on oriented exchange graphs (see Theorem \ref{thm:graph} and Theorem \ref{thm:hasse}).

\section{Preliminaries on cluster algebras}\label{sec2}

Throughout this paper, we fix a pair $(n,m)$ of integers with $m\geq n > 0$. For any positive integer $r$, we denote $[1,r] \coloneqq   \{1,\ldots,r\}$.
The matrices in this paper are always integer matrices.
\subsection{Mutation matrices and their mutations}

Recall that an $n\times n$ integer matrix $B$ is said to be {\em skew-symmetrizable}, if there exists a diagonal integer matrix $S=diag(s_1,\ldots,s_n)$ with $s_i>0$ ($i\in[1,n]$) such that $SB$ is skew-symmetric. Such a diagonal matrix $S$ is called a {\em skew-symmetrizer} of $B$.

An $m\times n$ integer matrix $\widetilde B=\begin{bmatrix}
       B\\ P
   \end{bmatrix}=(b_{ij})$ is called a {\em mutation matrix}, if its top $n\times n$ submatrix $B$ is skew-symmetrizable. The submatrix $B$ is called the {\em principal part} of $\widetilde B$.

Let $\widetilde B=\begin{bmatrix}
       B\\ P
   \end{bmatrix}=(b_{ij})$ be an $m\times n$ mutation matrix. The {\em mutation} of $\widetilde B$
   in direction $k\in[1,n]$ is
defined to be the new integer matrix $\mu_k(\widetilde B)=\widetilde B'=\begin{bmatrix}
       B'\\ P'
   \end{bmatrix}=(b_{ij}')$ given by
\begin{equation}\label{eqn:b-mutation}
b_{ij}^\prime=\begin{cases}-b_{ij}, & \text{if}\;i=k\;\text{or}\;j=k;\\
 b_{ij}+[b_{ik}]_+[b_{kj}]_+-[-b_{ik}]_+[-b_{kj}]_+,&\text{otherwise},\end{cases}
\end{equation}
where $[a]_+ \coloneqq   \max\{a,0\}$ for any $a\in\mathbb R$.

The following statements are well-known,  \confer\cite[Proposition 4.5]{fz_2002}, \cite[Lemma 3.2]{bfz_2005}.
\begin{itemize}
    \item $\mu_k(\widetilde B)=\widetilde B'$ is still a mutation matrix, namely, the submatrix $B'$ of $\widetilde B'$ is skew-symmetrizable;
    \item $B$ and $B'$ share the same skew-symmetrizer;
    \item $\mu_k$ is an involution,  \ie $\mu_k^2=id$;
    \item The rank of $ \widetilde B'=\mu_k(\widetilde B)$ is equal to the rank of $\widetilde B$.
\end{itemize}

As shown in \cite[(3.2)]{bfz_2005}, the  matrix $\widetilde B'=\mu_k(\widetilde B)$ can be written as
\begin{eqnarray}
    \label{eqn:ebf}
    \widetilde B'=\mu_k(\widetilde B)=E_{k,\varepsilon}^{\widetilde B}\widetilde BF_{k,\varepsilon}^{\widetilde B}, 
\end{eqnarray}
where $\varepsilon\in\{\pm \}$ and
\begin{itemize}
    \item $E_{k,\varepsilon}^{\widetilde B}$ is the $m\times m$ matrix obtained by replacing the $k$th column of $I_m$ by the column vector $(e_1,\ldots,e_m)^T$, where
    \[e_i=\begin{cases}
        -1,&\text{if }i=k;\\
        [-\varepsilon b_{ik}]_+,&\text{if }i\neq k.
    \end{cases}\]
    \item $F_{k,\varepsilon}^{\widetilde B}$ is the $n\times n$ matrix obtained by replacing the $k$th row of $I_n$ by the row vector $(f_1,\ldots,f_n)$, where
    \[f_i=\begin{cases}
        -1,&\text{if }i=k;\\
        [\varepsilon b_{ki}]_+,&\text{if }i\neq k.
    \end{cases}\]
\end{itemize}
 It turns out that $\widetilde B'$ is independent of the choice of a sign $\varepsilon\in\{\pm\}$ in \eqref{eqn:ebf}. The matrices $E_{k,\varepsilon}^{\widetilde B}$ and $F_{k,\varepsilon}^{\widetilde B}$ defined above coincide with the ones used in \cite[(3.2), (3.3)]{bz-2005}. It is easy to check
\begin{eqnarray}\label{eqn:ef2}
    (E_{k,\varepsilon}^{\widetilde B})^2=I_m\;\;\;\text{and}\;\;\; (F_{k,\varepsilon}^{\widetilde B})^2=I_n.
\end{eqnarray}

\subsection{Compatible pairs and their mutations}

 Let $\Lambda=(\lambda_{ij})$ be an $m\times m$ skew-symmetric integer matrix, and let
$\widetilde B=\begin{bmatrix}
       B\\ P
   \end{bmatrix}=(b_{ij})$ be
 an $m\times n$ integer matrix, where $B$ is the top $n\times n$ submatrix of $\widetilde B$.  Here we do not assume that $B$ is skew-symmetrizable.

\begin{definition}[Compatible pair, \cite{bz-2005}]
   Keep $\Lambda$ and $\widetilde B$ as above. Let $S=diag(s_1,\ldots,s_n)$ be a diagonal integer matrix with $s_j>0,\;j\in[1,n]$. The pair $(\widetilde B,\Lambda)$ is called a {\em compatible pair} of type $S$, if for any $j\in[1,n]$ and $i\in[1,m]$, we have
   \begin{eqnarray}\label{eqn:lpair}
    \sum_{k=1}^mb_{kj}\lambda_{ki}=\delta_{i,j}s_j.
   \end{eqnarray}
 In other words,  $\widetilde B^T\Lambda=(S\mid {\bf 0})$, where  ${\bf 0}$ is the $n\times (m-n)$ zero matrix.
\end{definition}

Notice that the diagonal matrix $S$ is uniquely determined by the compatible pair $(\widetilde B,\Lambda)$. The matrix $\Lambda$ in a compatible pair $(\widetilde B,\Lambda)$ is called a {\em Poisson coefficient matrix}.

\begin{proposition}[{\cite[Proposition 3.3]{bz-2005}}]
\label{pro:fullrank}
Let $(\widetilde B,\Lambda)$ be a compatible pair of type $S$ and $B$ the top $n\times n$ submatrix of $\widetilde B$. Then the matrix $\widetilde B$ has the full rank $n$ and $SB=\widetilde B^T\Lambda \widetilde B$. In particular,  $B$  is skew-symmetrizable.
\end{proposition}

Conversely, by \cite[Theorem 1.4]{gsv-2003}, whenever $\widetilde B$ is a mutation matrix of full rank, we can find a skew-symmetric integer matrix $\Lambda$ such that $(\widetilde B,\Lambda)$ is a compatible pair.

Let $(\widetilde B,\Lambda)$ be a compatible pair of type $S$ and $k\in [1,n]$. We set $\widetilde B' \coloneqq   \mu_k(\widetilde B)=E_{k,\varepsilon}^{\widetilde B}\widetilde BF_{k,\varepsilon}^{\widetilde B}$ and
\[\Lambda' \coloneqq   (E_{k,\varepsilon}^{\widetilde B})^T\Lambda E_{k,\varepsilon}^{\widetilde B}=(\lambda_{ij}').\]
We have
\begin{eqnarray}\label{eqn:L-mutation}
    \lambda_{ij}'=\begin{cases}-\lambda_{ik}+\sum_{l=1}^m[-\varepsilon b_{lk}]_+\lambda_{il},&\text{if }j=k;\\
-\lambda_{kj}+\sum_{l=1}^m[-\varepsilon b_{lk}]_+\lambda_{lj},&\text{if }i=k;\\
\lambda_{ij},&\text{otherwise}.
\end{cases}
\end{eqnarray}

\begin{proposition}[{\cite[Proposition 3.4]{bz-2005}}] \label{pro:pm}
\hfill
\begin{itemize}
\item[(i)] The matrix  $\Lambda'$ is independent of the choice of  $\varepsilon\in\{\pm\}$.
\item[(ii)] The pair $(\widetilde B',\Lambda')$ is still a compatible pair of type $S$.
\end{itemize}
\end{proposition}
Keep the above setting. We call $(\widetilde B',\Lambda')$ the {\em mutation of the compatible pair}  $(\widetilde B,\Lambda)$ in direction $k$ and denote $(\widetilde B',\Lambda')=\mu_k(\widetilde B,\Lambda)$. It is known from \cite[Proposition 3.6]{bz-2005} that  $\mu_k(\widetilde B',\Lambda')=(\widetilde B,\Lambda)$,  \ie $\mu_k$ is an involution.

The following proposition  is very useful to construct many examples of compatible pairs from a given skew-symmetrizable matrix.

\begin{proposition}
[{\cite[Example 0.5]{zelevinsky-2005}}]
\label{pro:ze-2005}
Let $B$ be an $n\times n$ skew-symmetrizable integer matrix and set $\widetilde B=\begin{bmatrix}
       B\\ I_n
   \end{bmatrix}$. Then the skew-symmetric
matrices $\Lambda$ compatible with $\widetilde B$ in the sense of \eqref{eqn:lpair} are those of the form
\[\Lambda=\begin{bmatrix}
    \Lambda_0&-S-\Lambda_0B\\
    S-B^T\Lambda_0&-SB+B^T\Lambda_0B
\end{bmatrix},
\]
where $\Lambda_0$ is an arbitrary skew-symmetric integer $n\times n$ matrix and $S$ is an arbitrary skew-symmetrizer of $B$.
\end{proposition}

\subsection{$Y$-pattern and cluster pattern}
Recall that we fixed a pair $(n,m)$ of integers with $m\geq n >0$. Let $\mathbb F$ be the field of rational functions over $\mathbb Q$ in $m$ variables.

A {\em $Y$-seed} of rank $n$ in $\mathbb F$ is a pair $({\bf y}, \widehat B)$, where
\begin{itemize}
	\item ${\bf y} = (y_1, \ldots, y_{m})$ satisfies that $\{y_1, \ldots, y_{m}\}$ is a free generating set of $\mathbb F$ over $\mathbb Q$;
	\item  $\widehat B=(B\mid Q)=(\hat b_{ij})$ is an $n\times m$ integer matrix such that its leftmost $n\times n$ submatrix $B$ is skew-symmetrizable.
\end{itemize}
The variables $y_1,\ldots,y_m$ are called the {\em $y$-variables} of $({\bf y}, \widehat B)$.

Let  $({\bf y}, \widehat B)$ be a $Y$-seed of rank $n$ in $\mathbb F$. The {\em mutation} of  $({\bf y}, \widehat B)$ in direction $k\in[1,n]$ is the pair  $({\bf y}', \widehat B') \coloneqq   \mu_k({\bf y}, \widehat B)$ given as follows:

\begin{align}\label{eqn:y-mutation}
 y_i^\prime&=\begin{cases}y_k^{-1}, &\text{if}\;i=k; \\
y_iy_k^{[\hat b_{ki}]_+}(1+y_k)^{-\hat b_{ki}},&\text{otherwise}.\end{cases}\\
\hat b_{ij}^\prime&=\begin{cases}-\hat b_{ij}, & \text{if}\;i=k\;\text{or}\;j=k;\\
\hat b_{ij}+[\hat b_{ik}]_+[\hat b_{kj}]_+-[-\hat b_{ik}]_+[-\hat b_{kj}]_+,&\text{otherwise}.\end{cases}\nonumber
\end{align}
One can check that the new pair  $({\bf y}', \widehat B')=\mu_k({\bf y}, \widehat B)$ is still a $Y$-seed of rank $n$ and $({\bf y}, \widehat B)=\mu_k({\bf y}', \widehat B')$.

A {\em cluster seed} or simply a {\em seed} of rank $n$ in $\mathbb F$ is a pair
$({\bf x}, \widetilde B)$, where
\begin{itemize}
	\item ${\bf x} = (x_1, \ldots, x_{m})$ satisfies that $\{x_1, \ldots, x_{m}\}$ is a free generating set of $\mathbb F$ over $\mathbb Q$;
	\item  $\widetilde B=\begin{bmatrix}
       B\\ P
   \end{bmatrix}=(b_{ij})$ is an $m\times n$ mutation matrix.
\end{itemize}
In this case, the tuple ${\bf x}$ is called the {\it cluster} of $({\bf x}, \widetilde B)$. Elements in ${\bf x}$ are called
{\it cluster variables}. More precisely, we call $x_1,\ldots,x_n$ {\em unfrozen cluster variables} and $x_{n+1},\ldots,x_{m}$ {\em frozen (cluster) variables}.
The matrices $B$ and $P$ are respectively called the {\it  exchange matrix}, {\it coefficient matrix} of $({\bf x}, \widetilde B)$. We denote by $\widehat {\bf y} \coloneqq   (\hat y_1,\ldots,\hat y_n)$, where $\hat y_k \coloneqq   \prod_{j=1}^mx_j^{b_{jk}}$. The variables $\hat y_1,\ldots,\hat y_n$ are called {\em $\hat y$-variables} associated to $({\bf x}, \widetilde B)$.

Let  $({\bf x}, \widetilde B)$ be a seed of rank $n$ in $\mathbb F$. The {\em mutation} of $({\bf x}, \widetilde B)$ in direction $k\in[1,n]$ is the pair  $({\bf x}', \widetilde B')=\mu_k({\bf x}, \widetilde B)$ given by $\widetilde B'=\mu_k(\widetilde B)$ and
\begin{eqnarray}
\label{eqn:x-mutation}
 x_i^\prime=\begin{cases}x_i,&
 \text{if}\;i\neq k;\\
 x_k^{-1}\cdot (\prod_{j=1}^mx_j^{[b_{jk}]_+}+\prod_{j=1}^mx_j^{[-b_{jk}]_+}),&\text{if}\;i= k.\end{cases}
\end{eqnarray}
One can check that the new pair  $({\bf x}', \widetilde B')=\mu_k({\bf x}, \widetilde B)$ is still a seed of rank $n$ and $({\bf x}, \widetilde B)=\mu_k({\bf x}', \widetilde B')$.

Let $\mathbb T_n$ denote the $n$-regular tree. We
 label the edges of $\mathbb T_n$ by $1,\ldots, n$ such that the $n$ different edges adjacent to the same vertex of $\mathbb T_n$ receive different labels.

\begin{definition}
(i) A {\em $Y$-pattern} $\mathcal S_Y=\{({\bf y}_t, \widehat B_t)\mid t\in \mathbb T_n\}$ of rank $n$
	is an assignment of a $Y$-seed $({\bf y}_t, \widehat B_t)$ of rank $n$ to
 	every vertex $t$ of $\mathbb T_n$ such that $({\bf y}_{t'}, \widehat B_{t'})=\mu_k({\bf y}_t, \widehat B_t)$ whenever
	\begin{xy}(0,1)*+{t}="A",(10,1)*+{t'}="B",\ar@{-}^k"A";"B" \end{xy} in $\TT_n$.

 (ii) A {\em cluster pattern} $\mathcal S_X=\{({\bf x}_t, \widetilde B_t)\mid t\in \mathbb T_n\}$ of rank $n$
	is an assignment of a cluster seed $({\bf x}_t, \widetilde B_t)$ of rank $n$ to
 	every vertex $t$ of $\mathbb T_n$ such that $({\bf x}_{t'}, \widetilde B_{t'})=\mu_k({\bf x}_t, \widetilde B_t)$ whenever
	\begin{xy}(0,1)*+{t}="A",(10,1)*+{t'}="B",\ar@{-}^k"A";"B" \end{xy} in $\TT_n$.
\end{definition}

We usually write ${\bf y}_t=(y_{1;t},\ldots,y_{m;t})$, $\widehat B_t=(B_t\mid Q_t)=(\hat b_{ij;t})$, ${\bf x}_t=(x_{1;t},\ldots,x_{m;t})$ and \[\widetilde B_t = \begin{pmatrix}
	    B_t\\ P_t
	\end{pmatrix}=(b_{ij;t}).\]
We write $\widehat{\bf y}_t=(\hat y_{1;t},\ldots,\hat y_{n;t})$ for the collection of $\hat y$-variables associated to seed $({\bf x}_t,\widetilde B_t)$, where $$\hat y_{k;t}=\prod_{j=1}^m x_{j;t}^{b_{jk;t}}.$$

In this paper, we usually fix a vertex $t_0$ as the {\em rooted vertex} of the $n$-regular tree $\mathbb T_n$. Clearly, both $Y$-pattern and cluster pattern are uniquely determined by the data at the rooted vertex $t_0$.

\begin{definition}\label{def:dual-pair}
Let  $\mathcal S_X=\{({\bf x}_t, \widetilde B_t)\mid t\in \mathbb T_n\}$  be a cluster pattern, $\mathcal S_Y=\{({\bf y}_t, \widehat B_t)\mid t\in \mathbb T_n\}$ a $Y$-pattern, and let $\mathsf{\Lambda}=\{\Lambda_t\mid t\in\mathbb T_n\}$ be a collection of $m\times m$ skew-symmetric integer matrices indexed by the vertices in $\mathbb T_n$.

\begin{itemize}
    \item [(i)] The pair $(\mathcal S_X,\mathcal S_Y)$ is called a {\em Langlands dual pair}, if $\widehat B_{t_0}=-\widetilde B_{t_0}^T$ holds for the rooted vertex $t_0$.
    \item[(ii)] The pair $(\mathcal S_X, \mathsf{\Lambda})$ is called a {\em $\mathsf{\Lambda}$-cluster pattern}, if $\{(\widetilde B_t,\Lambda_t)\mid t\in\mathbb T_n\}$ forms a collection of compatible pairs and $(\widetilde B_{t'},\Lambda_{t'})=\mu_k(\widetilde B_t,\Lambda_t)$  whenever \begin{xy}(0,1)*+{t}="A",(10,1)*+{t'}="B",\ar@{-}^k"A";"B" \end{xy} in $\TT_n$. In this case, we call the triple $({\bf x}_t,\widetilde B_t,\Lambda_t)$ a {\em $\mathsf{\Lambda}$-seed}.
    
    \item[(iii)]  The triple $(\mathcal S_X,\mathcal S_Y,\mathsf{\Lambda})$ is called a {\em Langlands-Poisson triple}, if   $(\mathcal S_X, \mathsf{\Lambda})$ is a $\mathsf{\Lambda}$-cluster pattern and
    $(\mathcal S_X,\mathcal S_Y)$ is a Langlands dual pair.
\end{itemize}

\end{definition}
The above definition is purely for convenience of this paper. Notice that each cluster pattern has a unique Langlands dual $Y$-pattern up to isomorphism.  The following result can be checked easily.

\begin{corollary}
Let $(\mathcal S_X,\mathcal S_Y)$ be a Langlands dual pair. Then $\widehat B_{t}=-\widetilde B_{t}^T$ holds for any vertex $t$ of $\mathbb T_n$.
\end{corollary}

Given a $\mathsf{\Lambda}$-cluster pattern $(\mathcal S_X, \mathsf{\Lambda})$, we can define a Poisson bracket $\{-,-\}$ on the ambient field $\mathbb F$ using the $m\times m$ skew-symmetric matrix $\Lambda_{t_0}=(\lambda_{ij;t_0})$ at the rooted vertex $t_0$:
$$\{x_{i;t_0},x_{j;t_0}\} \coloneqq   \lambda_{ij;t_0}\cdot x_{i;t_0}x_{j;t_0}.$$
It turns out this Poisson bracket is compatible with the cluster pattern $\mathcal S_X$, that is, for any cluster ${\bf x}_t$ of $\mathcal S_X$, we have $$\{x_{i;t},x_{j;t}\}=\lambda_{ij;t}\cdot x_{i;t}x_{j;t},$$
where $\lambda_{ij;t}$ is the $(i,j)$-entry of $\Lambda_t$, \confer \cite{gsv-2003}, \cite[Remark 4.6]{bz-2005}.

Finally, let us give the definition of (upper) cluster algebras.
\begin{definition}
 Let $\mathcal S_X=\{({\bf x}_t, \widetilde B_t)\mid t\in \mathbb T_n\}$ be a cluster pattern of rank $n$ in $\mathbb F$.
\begin{itemize}
    \item [(i)] The {\em cluster algebra} $\mathcal A$ associated to $\mathcal S_X$ is the $\mathbb Z$-subalgebra of $\mathbb F$ given by   $$\mathcal A \coloneqq   \mathbb Z[x_{1;t},\ldots,x_{n;t},x_{n+1;t}^{\pm1},\ldots,x_{m;t}^{\pm1}\mid t\in\mathbb T_n].$$
    Denote by $\mathcal A^+ \coloneqq   \mathbb Z[x_{1;t},\ldots,x_{m;t}\mid t\in\mathbb T_n]\subseteq\mathcal A$, the version of cluster algebra with frozen variables non inverted.
    \item[(ii)] The {\em upper cluster algebra} $\mathcal U$ associated to $\mathcal S_X$  is the $\mathbb Z$-subalgebra of  $\mathbb F$
    defined by   $$\mathcal U \coloneqq   \bigcap\limits_{t\in \mathbb T_n}\mathcal L(t),$$ where $\mathcal L(t) \coloneqq   \mathbb Z[x_{1;t}^{\pm 1},\ldots,x_{m;t}^{\pm 1}]$.
    \item[(iii)]  If $(\mathcal S_X, \mathsf{\Lambda})$ is a  $\mathsf{\Lambda}$-cluster pattern, then the corresponding (upper) cluster algebra $\mathcal A$, $\mathcal A^+$ and  $\mathcal U$
    are endowed with an extra data $\mathsf{\Lambda}=\{\Lambda_t\mid t\in\mathbb T_n\}$. In this case, we call $\mathcal A$ or $\mathcal A^+$ a {\em $\mathsf{\Lambda}$-cluster algebra} and call $\mathcal U$ a {\em $\mathsf{\Lambda}$-upper cluster algebra}.
    \item[(iv)] If $m=n$,  \ie there are no frozen variables, the corresponding (upper) cluster algebra is said to be {\em with trivial coefficients}.  
\end{itemize}
\end{definition}

Notice that the mutation matrices of $\mathsf{\Lambda}$-(upper) cluster algebras always have full rank, by Proposition \ref{pro:fullrank}.

\begin{theorem}[{\cite{fz_2002}*{Laurent phenomenon}}] For any cluster variable $z$ and any seed $({\bf x}_t,\widetilde B_t)$ of a cluster algebra $\mathcal A$, we have  \[z\in \mathbb Z[x_{1;t}^{\pm 1}, \ldots, x_{n;t}^{\pm 1},x_{n+1;t},\ldots,x_{m;t}].\] In particular,  $\mathcal A$ is contained in $\mathcal U$.
\end{theorem}

\begin{example}\label{ex:A2}
   Take $\widetilde B=\begin{bmatrix}
    0&1\\-1&0
\end{bmatrix}$ and ${\bf x}=(x_1,x_2)$. One can check that the  cluster algebra $\mathcal A$ defined by the initial seed $({\bf x},\widetilde B)$  has only five cluster variables:
\[x_1, \;x_2, \;x_3 \coloneqq   \frac{x_2+1}{x_1}, \;x_4 \coloneqq   \frac{x_1+x_2+1}{x_1x_2}, \;x_5 \coloneqq   \frac{x_1+1}{x_2}.\]
It is clear that all the cluster variables are contained in $\mathbb Z[x_1^{\pm 1},x_2^{\pm 1}]$.
\end{example}

\section{Tropical points, pointed elements and good elements}\label{sec3}

\subsection{Universal semifield and  tropical points}

Recall that $(\mathbb P, \cdot, \oplus)$ is called a {\em semifield} if $(\mathbb P,  \cdot)$ is an abelian multiplicative group endowed with a binary operation of auxiliary addition $\oplus$ which is commutative, associative and satisfies that the multiplication  distributes over the auxiliary addition. For example, $$\mathbb Z^{\rm min} \coloneqq   (\mathbb Z,+,\min\{-,-\})\;\;\;\text{and}\;\;\; \mathbb Z^{\rm max} \coloneqq   (\mathbb Z,+, \max\{-,-\})$$ 
are semifield and they are called {\em tropical semifield}.

Let $\QQ_{\rm sf}(u_1, \ldots, u_m)$ be the
set of all non-zero rational functions in $u_1, \ldots, u_m$ that have subtraction free expressions. The set $\QQ_{\rm sf}(u_1, \ldots, u_m)$  is a semifield
with respect to the usual operations of multiplication and addition. It is called an {\em universal semifield}.

From now on, we fix $\mathbb F \coloneqq   \mathbb Q(u_1,\ldots,u_m)$ and $\mathbb F_{>0} \coloneqq   \QQ_{\rm sf}(u_1, \ldots, u_m)$.
For any semifield $\mathbb P$ and  ${\bf p}=(p_1,\ldots,p_m)^T\in\mathbb P^m$, there exists a unique semifield homomorphism $${\pi}_{\bf p}:\mathbb F_{>0}\rightarrow \mathbb P$$ induced by $u_i\mapsto p_i$ for any $i$.
The map $\pi\colon{\bf p}\mapsto \pi_{\bf p}$ induces a bijection from $\mathbb P^m$ to the set $\Hom_{\rm sf}(\mathbb  F_{>0},\mathbb P)$ of semifield homomorphisms from $\mathbb F_{>0}$ to $\mathbb P$.

In this paper, we often take $\mathbb P=\mathbb Z^{\rm max}=(\mathbb Z,+,\max)$. The elements in $\Hom_{\rm sf}(\mathbb  F_{>0},\mathbb Z^{\rm max})$ are called {\em tropical $\mathbb Z^{\rm max}$-points} or simply {\em tropical points}, which are in bijection with the points in $\mathbb Z^m$.

\begin{definition}
 A {\em chart} in $\mathbb F_{>0}=\QQ_{\rm sf}(u_1, \ldots, u_m)$  is an $m$-tuple ${\bf z}=(z_1,\ldots,z_m)$
 of elements in $\mathbb  F_{>0}$ that generate $\mathbb F_{>0}$ as a semifield.
\end{definition}

The statements in the following proposition are clear.
\begin{proposition}\label{pro:chart}
    Let  $\mathcal C=\{{\bf u}_t\mid t\in\mathbb T\}$ be a collection of charts in $\mathbb F_{>0}=\QQ_{\rm sf}(u_1, \ldots, u_m)$ indexed by a set $\mathbb T$. Then the following statements hold.
    \begin{itemize}
    \item [(i)] For any chart ${\bf u}_t=(u_{1;t},\ldots,u_{m;t})$, denote by $u_i({\bf u}_t)$ the expression of $u_i$ as an element in $\mathbb  F(t)_{>0} \coloneqq   \mathbb Q_{\rm sf}(u_{1;t},\ldots,u_{m;t})$. The map
     $\varphi_t\colon u_i\mapsto u_i({\bf u}_t)$  induces an isomorphism of semifields from $\mathbb  F_{>0}$ to $\mathbb  F(t)_{>0}$.

        \item [(ii)] For any chart ${\bf u}_t=(u_{1;t},\ldots,u_{m;t})$ and any $\beta\in\Hom_{\rm sf}(\mathbb  F_{>0},\mathbb Z^{\rm max})$, denote by
        $${\bf q}^t(\beta) \coloneqq   (\beta(u_{1;t}),\ldots,\beta(u_{m;t}))^T\in\mathbb Z^m.$$ Then the following diagram commutes.
\[
\xymatrix{\mathbb F_{>0}\ar[r]^{\varphi_t\;\;\;}\ar[d]_{\beta}&\mathbb  F(t)_{>0}\ar[ld]^{\pi_{{\bf q}^t(\beta);t}}\\
\mathbb Z^{\rm max}}
\]
where $\pi_{{\bf q}^t(\beta);t}$ is the unique semifield homomorphism from $\mathbb  F(t)_{>0}$ to $\mathbb Z^{\rm max}$ induced by $u_{i;t}\mapsto \beta(u_{i;t})$ for any $i$.

\item[(iii)] The map $\beta\mapsto {\bf q}^t(\beta)$ induces a bijection from the tropical points in $\Hom_{\rm sf}(\mathbb F_{>0},\mathbb Z^{\rm max})$ and the points in $\mathbb Z^m$.
    \end{itemize}
\end{proposition}

Keep the setting in the above proposition. We call the column vector $${\bf q}^t(\beta)=(\beta(u_{1;t}),\ldots,\beta(u_{m;t}))^T\in\mathbb Z^m$$ the {\em coordinate vector} of the tropical point $\beta$ under the chart ${\bf u}_t$. When a collection $\mathcal C=\{{\bf u}_t\mid t\in\mathbb T\}$ of charts in $\mathbb  F_{>0}$ is given, we often identify a tropical point $\beta\in\Hom_{\rm sf}(\mathbb  F_{>0},\mathbb Z^{\rm max})$ with the collection $\{{\bf q}^t(\beta)\in\mathbb Z^m\mid t\in\mathbb T\}$ of coordinate vectors of $\beta$ under the charts.

Now let us go back to the situations that we are interested in.

\begin{itemize}
\item If we work on a $Y$-pattern $\mathcal S_Y=\{({\bf y}_t, \widehat B_t)\mid t\in \mathbb T_n\}$, we take $\mathbb F_{>0}$ to be the universal semifield generated by the initial $y$-variables $y_{1;t_0},\ldots, y_{m;t_0}$. Then
$\mathcal C \coloneqq   \{{\bf y}_t\mid t\in\mathbb T_n\}$ is a collection of charts in $\mathbb F_{>0}$.  We usually denote by 
$$\{{\bf g}^t(\beta)\in\mathbb Z^m\mid t\in\mathbb T_n\}$$ the collection of coordinate vectors of a tropical point $\beta\in\Hom_{\rm sf}(\mathbb  F_{>0},\mathbb Z^{\rm max})$ for the $Y$-pattern case.

\item If we work on a cluster pattern  $\mathcal S_X=\{({\bf x}_t, \widetilde B_t)\mid t\in \mathbb T_n\}$, we take $\mathbb F_{>0}$ to be the universal semifield generated by the initial
 cluster variables $x_{1;t_0},\ldots, x_{m;t_0}$. Then $\mathcal C \coloneqq   \{{\bf x}_t\mid t\in\mathbb T_n\}$ is a collection of charts  in $\mathbb F_{>0}$.
 We usually denote by 
$$\{{\bf a}^t(\beta)\in\mathbb Z^m\mid t\in\mathbb T_n\}$$ the collection of coordinate vectors of a tropical point $\beta\in\Hom_{\rm sf}(\mathbb  F_{>0},\mathbb Z^{\rm max})$ for the cluster pattern case.
\end{itemize}

Since the transition maps between any two adjacent charts are clear in both $Y$-pattern and cluster pattern, the corresponding tropical points (after identifying with the corresponding coordinate vectors) can be defined using the tropical version of the transition maps in \eqref{eqn:y-mutation} and \eqref{eqn:x-mutation}.

\begin{definition}[Tropical points]
    Let $\mathcal S_Y=\{({\bf y}_t, \widehat B_t)\mid t\in \mathbb T_n\}$ be a $Y$-pattern of rank $n$ in $\mathbb F$ and $\mathcal S_X=\{({\bf x}_t, \widetilde B_t)\mid t\in \mathbb T_n\}$ a cluster pattern of rank $n$ in $\mathbb F$.
\begin{itemize}
    \item [(i)] A tropical point $[{\bf g}]=\{{\bf g}^t\in\mathbb Z^m\mid t\in\mathbb T_n\}$ associated to the $Y$-pattern $\mathcal S_Y$ is an assignment of a (column) vector ${\bf g}^t=(g_{1}^t,\ldots,g_{m}^t)^T$ in $\mathbb Z^m$ to each vertex $t$ of $\mathbb T_n$ such that
    \begin{eqnarray}\label{eqn:y-trop}
g_{i}^{t'}=\begin{cases}-g_{k}^t,& \text{if}\;i=k;\\ g_{i}^t+[\hat b_{ki;t}]_+g_{k}^t+(-\hat b_{ki;t})[g_{k}^t]_+,&\text{if}\;i\neq k. \end{cases}
\end{eqnarray}
whenever
	\begin{xy}(0,1)*+{t}="A",(10,1)*+{t'}="B",\ar@{-}^k"A";"B" \end{xy} in $\TT_n$. We denote by $\mathcal S_Y(\mathbb Z^{\rm max})$ the set of tropical points associated to the $Y$-pattern $\mathcal S_Y$.

 \item[(ii)] A tropical point $[{\bf a}]=\{{\bf a}^t\in\mathbb Z^m\mid t\in\mathbb T_n\}$ associated to the cluster pattern $\mathcal S_X$ is an assignment of a (column) vector ${\bf a}^t=(a_{1}^t,\ldots,a_{m}^t)^T$ in $\mathbb Z^m$ to each vertex $t$ of $\mathbb T_n$ such that
    \begin{eqnarray}\label{eqn:x-trop}
a_{i}^{t'}=\begin{cases}-a_{k}^t+\max\{\sum_{j=1}^m[b_{jk;t}]_+a_{j}^t, \; \sum_{j=1}^m[-b_{jk;t}]_+a_{j}^t\}
,& \text{if}\;i=k;\\ a_{i}^t,&\text{if}\;i\neq k. \end{cases}
\end{eqnarray}
whenever
	\begin{xy}(0,1)*+{t}="A",(10,1)*+{t'}="B",\ar@{-}^k"A";"B" \end{xy} in $\TT_n$.
 We denote by $\mathcal S_X(\mathbb Z^{\rm max})$ the set of tropical points associated to the cluster pattern $\mathcal S_X$.
\end{itemize}
\end{definition}

\begin{remark}
The relations in \eqref{eqn:y-trop} and  \eqref{eqn:x-trop}  are obtained from mutation relations \eqref{eqn:y-mutation} and \eqref{eqn:x-mutation} by tropicalization over the tropical semifield $\mathbb Z^{\rm max}=(\mathbb Z,+,\max\{-,-\})$. Namely, we make the following replacements.
\begin{itemize}
    \item Replace the multiplication and addition in \eqref{eqn:y-mutation} and \eqref{eqn:x-mutation}  by ``+" and ``{\rm max}\{--,--\}" over $\mathbb Z$, respectively.
     \item Replace ``1" in \eqref{eqn:y-mutation} by ``0";
\end{itemize}
\end{remark}

The following corollary is obvious.
\begin{corollary}\label{cor:bijection} The following statements hold.
\begin{itemize}
\item[(i)] Each tropical point $[{\bf g}]=\{{\bf g}^t\in\mathbb Z^m\mid t\in\mathbb T_n\}$ in $\mathcal S_Y(\mathbb Z^{\rm max})$ corresponds to a unique semifield homomorphism $\beta\in\Hom_{\rm sf}(\mathbb  F_{>0},\mathbb Z^{\rm max})$ such that $$\beta({\bf y}_t)=(\beta(y_{1;t}),\ldots,\beta(y_{m;t}))=({\bf g}^t)^T$$
for any vertex $t\in\mathbb T$, where $\mathbb F_{>0}=\mathbb Q_{\rm sf}(y_{1;t_0},\ldots, y_{m;t_0})$.
\item[(ii)] Each tropical point $[{\bf a}]=\{{\bf a}^t\in\mathbb Z^m\mid t\in\mathbb T_n\}$ in $\mathcal S_X(\mathbb Z^{\rm max})$ corresponds to a unique semifield homomorphism $\beta\in\Hom_{\rm sf}(\mathbb  F_{>0},\mathbb Z^{\rm max})$ such that $$\beta({\bf x}_t)=(\beta(x_{1;t}),\ldots,\beta(x_{m;t}))=({\bf a}^t)^T$$
for any vertex $t\in\mathbb T$, where $\mathbb F_{>0}=\mathbb Q_{\rm sf}(x_{1;t_0},\ldots, x_{m;t_0})$.
\end{itemize}
\end{corollary} 

\subsection{Compatibly pointed elements and good elements}
In this subsection, we fix a {\em full rank} upper cluster algebra $\mathcal U$, which means its initial mutation matrix $\widetilde B_{t_0}$ has the full rank $n$, equivalently, any mutation matrix $\widetilde B_t$ of $\mathcal U$ has the full rank $n$.

 Since $\mathcal U$ is of full rank, each seed $({\bf x}_t,\widetilde B_t)$ of $\mathcal U$ defines a partial order $\preceq_t$ on $\mathbb Z^m$. For two vectors ${\bf g},{\bf g}'\in\mathbb Z^m$, we write ${\bf g}'\preceq_t{\bf g}$ if there exists some vector ${\bf v}=(v_1,\ldots,v_n)^T\in\mathbb N^n$ such that $${\bf g}'={\bf g}+\widetilde B_t{\bf v},$$ equivalently, ${\bf x}_t^{{\bf g}'}={\bf x}_t^{\bf g}\cdot \widehat {\bf y}_t^{\bf v}$, where $\widehat {\bf y}_t^{\bf v}={\bf x}_t^{\widetilde B_t{\bf v}}$. We denote by ${\bf g}'\prec_t{\bf g}$ for the case ${\bf g}'\preceq_t{\bf g}$ and ${\bf g}'\neq {\bf g}$.

The partial order  $\preceq_t$ above is known as {\em dominance order} associated to seed $({\bf x}_t,\widetilde B_t)$ in \cite{Qin_2017}. The dominance order plays an important role in study of the bases problem of (upper) cluster algebras.

For any seed $({\bf x}_t,\widetilde B_t)$ of $\mathcal U$, we denote by $R_t$ the collection of the formal Laurent series
$$u=\sum_{{\bf h}\in\mathbb Z^m}b_{\bf h}{\bf x}_t^{\bf h},\;\;\;b_{\bf h}\in\mathbb Z,$$
such that the set $\{{\bf h}\in\mathbb Z^m\mid b_{\bf h}\neq 0\}$ has finitely many maximal elements with respect to the dominance order $\preceq_t$ on $\mathbb Z^m$. It is easy to see that the collection $R_t$ has a ring structure.

Now we recall the definition of pointed elements introduced in \cite{Qin_2017}. Let $$u=\sum_{{\bf h}\in\mathbb Z^m}b_{\bf h}{\bf x}_t^{\bf h}\in R_t.$$ We say that $u$ is {\em pointed} for the seed $({\bf x}_t,\widetilde B_t)$ if the following two conditions are satisfied.
\begin{itemize}
\item The set $\{{\bf h}\in\mathbb Z^m\mid b_{\bf h}\neq 0\}$ has a unique maximal element ${\bf g}$ under the dominance order $\preceq_t$ on $\mathbb Z^m$. We denote by $\de^t(u) \coloneqq   {\bf g}$ and call it the {\em degree} of $u$ with respect to the seed $({\bf x}_t,\widetilde B_t)$.
    \item The coefficient corresponds to the degree term is $1$,  \ie $b_{\bf g}=1$ for ${\bf g}=\de^t(u)$.
\end{itemize}

By definition, each pointed element $u$ for the seed $({\bf x}_t,\widetilde B_t)$ takes the following form:
$$u={\bf x}_t^{\bf g}+\sum_{{\bf h}\prec_t {\bf g}}b_{\bf h}{\bf x}_t^{\bf h}={\bf x}_t^{\bf g}F(\hat y_{1;t},\ldots,\hat y_{n;t}),$$
where $b_{\bf h}\in\mathbb Z$ and $F\in\mathbb Z[[y_1,\ldots,y_n]]$ is a polynomial series with constant term $1$.

In cluster algebras, we are often interested in the pointed elements which can be controlled by tropical points associated to a $Y$-pattern. Such elements correspond to compatibly pointed elements introduced in \cite{qin_2019}.

\begin{definition}[Compatibly pointed elements and good elements]\label{def:pointed}
 Let $(\mathcal S_X,\mathcal S_Y)$ be a Langlands dual pair and $\mathcal U$ the upper cluster algebra associated to the cluster pattern $\mathcal S_X$. Assume that $\mathcal U$ is of full rank.

 \begin{itemize}
 \item[(i)] An element $u\in\mathcal U$ is said to be {\em compatibly pointed}, if $u$ is pointed for any seed $({\bf x}_t,\widetilde B_t)$ of $\mathcal U$ and the collection $[{\bf g}] \coloneqq   \{\de^t(u)\in\mathbb Z^m\mid t\in\mathbb T_n\}$ forms a tropical point in $\mathcal S_Y(\mathbb Z^{\rm max})$. In this case, we also say that $u$ is {\em $[{\bf g}]$-pointed}.

     \item[(ii)] A $[{\bf g}]$-pointed element $u\in\mathcal U$ is said to be {\em $[{\bf g}]$-good}, if it is universally positive,  \ie $u\in\mathbb Z_{\geq 0}[x_{1;t}^{\pm 1},\ldots,x_{m;t}^{\pm 1}]$ for any vertex $t\in\mathbb T_n$.
\end{itemize}
\end{definition}

 Let $u$ be a $[{\bf g}]$-pointed element,  then the Laurent expansion of  $u$ with respect to any seed $({\bf x}_t,\widetilde B_t)$ of $\mathcal U$ has a {\em canonical expression:}
\begin{eqnarray}\label{eqn:upointed}
   u={\bf x}_t^{{\bf g}_u^t}F_u^t(\hat y_{1;t},\ldots,\hat y_{n;t}),
 \end{eqnarray}
where ${\bf g}_u^t=\de^t(u)$ is the value of $[{\bf g}]$ at vertex $t$ and $F_u^t$ is a polynomial in $\mathbb Z[y_1,\ldots,y_n]$ with constant term $1$.
\begin{itemize}
\item The degree vector ${\bf g}_u^t=\de^t(u)$ is often called the {\em extended $g$-vector} of $u$ with respect to vertex $t$; 

\item The polynomial $F_u^t$ is called the {\em $F$-polynomial} of $u$ with respect to vertex $t$;
\item For $k\in[1,n]$, we denote by $f_{k;u}^t$ the maximal degree of $y_k$ in $F_u^t$. The vector $${\bf f}_u^t \coloneqq   (f_{1;u}^t,\ldots,f_{n;u}^t)^T\in\mathbb N^n$$ is called
 the {\em $f$-vector} of $u$ with respect to vertex $t$;
 \item The $[{\bf g}]$-pointed element $u$ is said to be {\em $[{\bf g}]$-bipointed}, if for any vertex  $t\in\mathbb T_n$ the monomial $y_1^{f_{1;u}^t}\cdots y_n^{f_{n;u}^t}$ appears in $F_u^t$ and it has  coefficient $1$;
     \item The $[{\bf g}]$-pointed element $u$ is said to be {\em $[{\bf g}]$-bigood}, if it is $[{\bf g}]$-bipointed and universally positive.
\end{itemize}

 Recall that a {\em cluster monomial} in  $\mathcal U$ is a Laurent monomial of the form
\[{\bf x}_w^{\bf v} \coloneqq   \prod_{j=1}^mx_{j;w}^{v_j}\]
for some cluster ${\bf x}_w$ of $\mathcal U$ and some vector ${\bf v}=(v_1,\ldots,v_m)^T\in \mathbb Z^m$ with $v_i\geq 0$ for any $i\in[1,n]$.  Notice that we allow $v_j$ to be negative for $j\in[n+1,m]$.

Keep the setting as in Definition \ref{def:pointed}. Given a vertex $w\in\mathbb T_n$ and a vector ${\bf v}\in \mathbb Z^m$, we denote by $[({\bf v},w)]$ the unique tropical point $\{{\bf g}^t\in\mathbb Z^m\mid t\in\mathbb T_n\}$ in $\mathcal S_Y(\mathbb Z^{\rm max})$ determined by the condition ${\bf g}^w={\bf v}$.

\begin{proposition} \label{pro:ghkk}
 Let ${\bf x}_w^{\bf v}$ be a cluster monomial in $\mathcal U$. Then ${\bf x}_w^{\bf v}$ is $[({\bf v},w)]$-bigood in $\mathcal U$.
\end{proposition}
\begin{proof}
By the results in \cite{GHKK18}, we know that $u \coloneqq   {\bf x}_w^{\bf v}$ is $[({\bf v},w)]$-good in $\mathcal U$.
Let ${\bf f}_u^t=(f_{1;u}^t,\ldots,f_{n;u}^t)^T$ be the $f$-vector of  $u={\bf x}_w^{\bf v}$ with respect to vertex $t\in\mathbb T_n$. Since the $F$-polynomials of cluster variables have constant term $1$ and by \cite[Proposition 5.3]{fomin_zelevinsky_2007}, we know that the monomial $y_1^{f_{1;u}^t}\cdots y_n^{f_{n;u}^t}$ appears in the $F$-polynomial $F_u^t$ with coefficient $1$. Hence, $u={\bf x}_w^{\bf v}$ is $[({\bf v},w)]$-bigood in $\mathcal U$.
\end{proof}

\begin{example}\label{ex:A2-2}
Let us continue with Example \ref{ex:A2}. We see that $\widetilde B=\begin{bmatrix}
    0&1\\-1&0
\end{bmatrix}$ has full rank and $\hat y_1=x_2^{-1},\;\hat y_2=x_1$. The canonical expressions of the five cluster variables
$$x_1, \;x_2, \;x_3=\frac{x_2+1}{x_1}, \;x_4=\frac{x_1+x_2+1}{x_1x_2}, \;x_5=\frac{x_1+1}{x_2}$$
with respect to the initial seed $({\bf x},\widetilde B)$ are given as follows:
    \[x_1=x_1\cdot 1,\;\;\;x_2=x_2\cdot 1,\;\;\;x_3={x_1^{-1}x_2\cdot(1+\widehat y_1),}\;\;\;  x_4={x_1^{-1}\cdot (1+\widehat y_1+\widehat y_1\widehat y_2),}\;\;\;
    x_5={x_2^{-1}\cdot (1+\widehat y_2)}.\]
\end{example}

\subsection{Basic properties of pointed elements}
As before, we fix a full rank upper cluster algebra $\mathcal U$.  Let us recall some basic properties of pointed elements in $R_t$.

\begin{lemma}
[{\cite{qin_2019}*{Lemmas 3.2.5, 3.2.6}}]
\label{lem:qin1}
Suppose that $u, u^\prime\in R_t$ are two pointed elements for seed  $({\bf x}_t,\widetilde B_t)$. Then the following statements hold.

\begin{itemize}
    \item [(i)] The product $u u^\prime\in R_t$ is  pointed  for seed  $({\bf x}_t,\widetilde B_t)$ and we have   $$\de^{t}(u u^\prime)=\de^{t}(u)+\de^{t}(u^\prime).$$
    \item[(ii)] $u$ has a multiplicative inverse $u^{-1}$ in $R_{t}$. Moreover, $u^{-1}$ is pointed for $({\bf x}_t,\widetilde B_t)$ and  $$\de^{t}(u^{-1})=-\de^{t}(u).$$
\end{itemize}
\end{lemma}

Let  $({\bf x}_t,\widetilde B_t)$ and $({\bf x}_{w},\widetilde B_{w})$ be two seeds of $\mathcal U$ at vertices $t,w\in\mathbb T_n$.
 By Proposition \ref{pro:ghkk}, we know that any cluster variable $x_{j;t}$ ($j\in[1,m]$) is pointed for $({\bf x}_{w},\widetilde B_{w})$.
 We denote by $\widetilde G_t^w$ the $m\times m$ matrix whose $j$th column is given by ${\bf g}_{x_{j;t}}^w=\de^w(x_{j;t})\in\mathbb Z^m$. Notice that $\widetilde G_t^w$ takes the form $$\widetilde G_t^w=\begin{bmatrix}G_t^w&0\\
 \ast&I_{m-n} \end{bmatrix},$$
 where $G_t^w$ is the $n\times n$ submatrix of $\widetilde G_t^w$.  We call $\widetilde G_t^w$ the {\em extended $G$-matrix} and $G_t^w$ the {\em $G$-matrix} of vertex $t$ with respect to $w$. It is easy to see that if $t=w$, then $\widetilde G_t^w=I_m$.

 Let $\overleftarrow{\mu}$ be the mutation sequence corresponding to the unique path from vertex $w$ to $t$ in $\mathbb T_n$.
 We apply the mutation sequence $\overleftarrow{\mu}$ to $\begin{pmatrix} B_w\\ I_{n}\end{pmatrix}$, where $B_w$ is the principal part of $\widetilde B_w$. The
 resulting matrix $\overleftarrow{\mu}\begin{pmatrix} B_w\\ I_{n}\end{pmatrix}$ takes the form $\begin{pmatrix} B_t\\  C_t^w\end{pmatrix}$. We call the $n\times n$ matrix $C_t^w$ the {\em  $C$-matrix} of the vertex $t$ with respect to $w$.

\begin{theorem}[{\cite{GHKK18}, \cite[(3.11)]{NZ12}}]
\label{thm:CG}
Let $t,w$ be two vertices in $\mathbb T_n$. The following statements hold.
\begin{itemize}
\item[(i)]  $\det(\widetilde G_t^w)=\pm 1=\det(G_t^w)$ and $\det(C_t^w)=\pm 1$.
    \item [(ii)] Each column of $C_t^w$ is either a non-negative vector or a non-positive vector.

    \item [(iii)]  $SC_t^wS^{-1}(G_t^w)^{\rm T}=I_n$, where $S$ is any skew-symmetrizer of $B_{w}$.
\end{itemize}
\end{theorem}

\begin{lemma}[{\cite{qin_2019}*{Lem. 3.3.6, Prop. 3.3.10}}]

\label{lem:deg}
Let $\mathcal U$ be a full rank upper cluster algebra and $t, w$ two vertices of $\mathbb T_n$.  The following statements hold.
\begin{itemize}
    \item [(i)] Any Laurent monomial
${\bf x}_t^{\bf h}$ is pointed  for the seed $({\bf x}_w,\widetilde B_w)$ and we have  $$\de^{w}({\bf x}_t^{\bf h})=\widetilde G_t^w{\bf h}=\de^{w}({\bf x}_w^{\widetilde G_t^w{\bf h}}).$$
\item [(ii)] 
Let $\widehat {\bf y}_t$ and $\widehat {\bf y}_w$ be the collections of $\hat y$-variables at vertices $t$ and $w$. Then for any ${\bf v}\in \mathbb Z^n$, the Laurent monomial $\widehat{\bf y}_t^{\bf v}={\bf x}_t^{\widetilde B_t{\bf v}}$ is pointed for the seed $({\bf x}_w,\widetilde B_w)$  and
  $$\de^{w}(\widehat{\bf y}_t^{\bf v})=\widetilde G_t^w\widetilde B_t{\bf v}= \widetilde B_wC_t^w{\bf v}=\de^{w}(\widehat {\bf y}_w^{C_t^w{\bf v}}).$$
\end{itemize}
\end{lemma}

Notice that any element in $R_w$ can be written as a finite $\mathbb Z$-linear combinations of pointed elements for seed $({\bf x}_w,\widetilde B_w)$. By the above lemma, any Laurent monomial ${\bf x}_t^{\bf h}$ in $\mathcal L(t) \coloneqq   \mathbb Z[x_{1;t}^{\pm 1},\ldots,x_{m;t}^{\pm 1}]$ is pointed for the seed $({\bf x}_w,\widetilde B_w)$. Hence, $\mathcal L(t)$ is contained in $R_w$.

\begin{corollary}\label{cor:deg}
Let $\mathcal U$ be a full rank upper cluster algebra and $t, w$ two vertices of $\mathbb T_n$. Let  $$u=\sum\limits_{{\bf h}\in \mathbb Z^m}b_{\bf h}{\bf x}_t^{\bf h}$$ be a Laurent polynomial in $\mathcal L(t)=\mathbb Z[x_{1;t}^{\pm 1},\ldots,x_{m;t}^{\pm 1}]$. Then the following statements hold.

\begin{itemize}
    \item [(i)]   $u$ is pointed for the seed $({\bf x}_w,\widetilde B_w)$ if and only if there exists a (unique) ${\bf g}\in\mathbb Z^m$ with $b_{\bf g}=1$ such that $\de^{w}({\bf x}_t^{\bf h}) \preceq_{w}\de^{w}({\bf x}_t^{\bf g})$ for any ${\bf h}\in\mathbb Z^m$ with $b_{\bf h}\neq 0$, equivalently, $$\widetilde G_t^w{\bf h}\preceq_{w} \widetilde G_t^w{\bf g}$$  for any ${\bf h}\in\mathbb Z^m$ with $b_{\bf h}\neq 0$.

     \item[(ii)] Suppose that $u$ is pointed for the seed $({\bf x}_t,\widetilde B_t)$, that is, $u$ can be written as
     $$u= {\bf x}_t^{\bf g}\cdot \sum\limits_{{\bf v}\in\mathbb N^{n}}c_{\bf v}\widehat{\bf y}_t^{\bf v},$$
     where $c_{\bf v}\in\mathbb Z$ and $c_0=1$. Then  $u$ is pointed for the new seed $({\bf x}_w,\widetilde B_w)$ if and only if
     there exists a (unique) ${\bf u}\in \mathbb N^{n}$ with $c_{\bf u}=1$ such that $\de^{w}(\widehat{\bf y}_t^{\bf v})\preceq_{w}\de^{w}(\widehat {\bf y}_t^{\bf u})$ for any ${\bf v}\in\mathbb N^{n}$ with $c_{\bf v}\neq 0$, equivalently,  $$C_t^w{\bf v}-C_t^w{\bf u}=C_t^w({\bf v}-{\bf u})\in\mathbb N^n$$ for any ${\bf v}\in\mathbb N^{n}$ with $c_{\bf v}\neq 0$.
\end{itemize}
\end{corollary}

\begin{proof}
Thanks to Theorem \ref{thm:CG} (i) and by Lemma \ref{lem:deg}, we know that $\de^{w}({\bf x}_t^{\bf h})=\de^{w}({\bf x}_t^{{\bf h}'})$ if and only if ${\bf h}={\bf h}'$. By the similar argument and since $\widetilde B_w$ has full rank, we know that $\de^{w}(\widehat{\bf y}_t^{\bf v})=\de^{w}(\widehat{\bf y}_t^{{\bf v}'})$ if and only if ${\bf v}={\bf v}'$.

Based on the above observations, the results in (i) and (ii) follow from the definition of pointed elements and Lemma \ref{lem:deg}.
\end{proof}

Let $(\mathcal S_X,\mathcal S_Y)$ be a Langlands dual pair and $\mathcal U$ the upper cluster algebra associated to the cluster pattern $\mathcal S_X$. As before, we assume that $\mathcal U$ is of full rank.

\begin{corollary} \label{cor:deg2}
Keep the above setting.  Let
$[{\bf g}]=\{{\bf g}^t\in\mathbb Z^m\mid t\in\mathbb T_n\}$ be a tropical point  in $\mathcal S_Y(\mathbb Z^{\rm max})$, and write ${\bf g}^t=(g_{1}^t,\ldots,g_{m}^t)^T$ for $t\in\mathbb T_n$.  Let
\begin{xy}(0,1)*+{w}="A",(10,1)*+{s}="B",\ar@{-}^k"A";"B" \end{xy} be an edge in $\TT_n$.  The following statements hold.

\begin{itemize}
\item[(i)] $\de^{w}({\bf x}_s^{{\bf g}^s}\cdot \widehat{\bf y}_s^{[-g_{k}^s]_+{\bf e}_k})={\bf g}^w$,  where ${\bf e}_k$ is the $k$th column of $I_n$.
    \item[(ii)] Suppose that $u\in\mathcal U$ is a $[{\bf g}]$-pointed element. Let $u= {\bf x}_s^{{\bf g}^s}\cdot \sum\limits_{{\bf v}\in\mathbb N^{n}}c_{\bf v}\widehat{\bf y}_s^{\bf v}$ be the Laurent expansion of $u$ with respect to seed $({\bf x}_s,\widetilde B_s)$. Then
        the coefficient of $\widehat {\bf y}_s^{[-g_{k}^s]_+{\bf e}_k}$ in $\sum\limits_{{\bf v}\in\mathbb N^{n}}c_{\bf v}\widehat{\bf y}_s^{\bf v}$ is $1$ and for any integer $d$ with $d>[-g_{k}^s]_+$, the coefficient of $\widehat {\bf y}_s^{d{\bf e}_k}$ in $\sum\limits_{{\bf v}\in\mathbb N^{n}}c_{\bf v}\widehat{\bf y}_s^{\bf v}$ is zero.
\end{itemize}

\end{corollary}

\begin{proof}
(i) Let $\widetilde G_s^w$ and  $C_s^w$ be the extended $G$-matrix and  $C$-matrix of vertex $s$ with respect to $w$.
By Lemma \ref{lem:qin1} and Lemma \ref{lem:deg}, we have
$$\de^{w}({\bf x}_s^{{\bf g}^s}\cdot \widehat{\bf y}_s^{[-g_{k}^s]_+{\bf e}_k})=\widetilde G_s^w{\bf g}^s+\widetilde B_{w}C_s^w\cdot ([-g_{k}^s]_+{\bf e}_k).$$
Since the two vertices $w$ and $s$ are jointed in $\mathbb T_n$ by an edge labeled by $k$,
 we know that $\widetilde G_s^w=(g_{ij;s})$ and $C_s^w=(c_{ij;s})$ are given as follows:
\begin{align}
g_{ij;s}&=\begin{cases}-1,&i=j=k;\\
[-b_{ik;w}]_+=[b_{ik;s}]_+,&j=k,\;i\neq k;\\
\delta_{i,j},&\text{else},\end{cases}\nonumber\\
c_{ij;s}&=\begin{cases}-1,&i=j=k;\\
[b_{kj;w}]_+=[-b_{kj;s}]_+,&i=k,\;j\neq k;\\
\delta_{i, j},&\text{else}. \end{cases}\nonumber
\end{align}
Because the $k$th column of $C_s^w$ equals $-{\bf e}_k$,  we have
$$\widetilde G_s^w{\bf g}^s+\widetilde B_{w}C_s^w\cdot ([-g_{k}^s]_+{\bf e}_k)=\widetilde G_s^w{\bf g}^s-\widetilde B_{w}\cdot ([-g_{k}^s]_+{\bf e}_k).$$
Now we look at the $i$th component $g_i^\prime$ of the vector  $${\bf g}^\prime \coloneqq   \widetilde G_s^w{\bf g}^s-\widetilde B_{w}\cdot ([-g_{k}^s]_+{\bf e}_k).$$
We have
$$g_i^\prime=\begin{cases}-g_{k}^s-0=-g_{k}^s,& i=k;\\
(g_{i}^s+[b_{ik;s}]_+g_{k}^s)-b_{ik;w}[-g_{k}^s]_+=g_{i}^s+[b_{ik;s}]_+g_{k}^s+b_{ik;s}[-g_{k}^s]_+,&i\neq k.\end{cases}$$
Since $\widehat B_s=-\widetilde B_s^T$, we can rewrite the above equality as follows:
$$g_i^\prime=\begin{cases}-g_{k}^s,& i=k;\\
g_{i}^s+[-\hat b_{ki;s}]_+g_{k}^s-\hat b_{ki;s}[-g_{k}^s]_+,&i\neq k.\end{cases}$$
Using the equality $[b]_+g-[-b]_+g=bg=b[g]_+-b[-g]_+$ for any $b,g\in\mathbb R$, one can see the above equality is the same as
$$g_i^\prime=\begin{cases}-g_{k}^s,& i=k;\\
g_{i}^s+[\hat b_{ki;s}]_+g_{k}^s-\hat b_{ki;s}[g_{k}^s]_+,&i\neq k.\end{cases}$$
 Then by \eqref{eqn:y-trop}, we get ${\bf g}^w={\bf g}'=\widetilde G_s^w{\bf g}^s-\widetilde B_{w}\cdot ([-g_{k}^s]_+{\bf e}_k)$. Hence,
$\de^{w}({\bf x}_s^{{\bf g}^s}\cdot \widehat{\bf y}_s^{[-g_{k}^s]_+{\bf e}_k})={\bf g}^w$.

(ii) Since $u$ is $[{\bf g}]$-pointed, we know that it is pointed for the seed $({\bf x}_w,\widetilde B_w)$.
By Corollary \ref{cor:deg} (ii),  there exists a unique ${\bf u}\in\mathbb N^n$ with $c_{\bf u}=1$ such that $$\de^{w}({\bf x}_s^{{\bf g}^s}\cdot \widehat{\bf y}_s^{\bf v})\preceq_{w}\de^{w}({\bf x}_s^{{\bf g}^s}\cdot \widehat {\bf y}_s^{\bf u})=\de^{w}(u)={\bf g}^w$$ for any ${\bf v}\in\mathbb N^{n}$ with $c_{\bf v}\neq 0$. By the result in (i), we know that ${\bf u}=[-g_{k}^s]_+{\bf e}_k$.

Notice that $\de^{w}({\bf x}_s^{{\bf g}^s}\cdot \widehat{\bf y}_s^{\bf v})\preceq_{w}\de^{w}({\bf x}_s^{{\bf g}^s}\cdot \widehat {\bf y}_s^{\bf u})$ holds if and only if the following relation holds.
$$C_s^w{\bf v}-C_s^w{\bf u}=C_s^w({\bf v}-{\bf u})\in\mathbb N^n.$$
Since the $k$th column of $C_s^w$ is $-{\bf e}_k$, we obtain that $C_s^w(d{\bf e}_k-[-g_{k}^s]_+{\bf e}_k)$ belongs to $\mathbb N^n$ if and only if $d\leq [-g_{k}^s]_+$. Hence, for any integer $d$ with $d>[-g_{k}^s]_+$, the coefficient of $\widehat {\bf y}_s^{d{\bf e}_k}$ in $\sum\limits_{{\bf v}\in\mathbb N^{n}}c_{\bf v}\widehat{\bf y}_s^{\bf v}$ is zero.
\end{proof}

\begin{corollary}\label{cor:f-vector}
Let
$u$ be a $[{\bf g}]$-pointed element in $\mathcal U$ and let ${\bf f}_u^t=(f_{1;u}^t,\ldots,f_{n;u}^t)^T$ be the $f$-vector of $u$ with respect to vertex $t\in\mathbb T_n$. The following statements hold.
\begin{itemize}
\item[(i)] If $f_{k;u}^t=0$ for some $k\in[1,n]$, then the $k$th component $g_{k}^t$ of ${\bf g}_u^t=\de^t(u)$ is non-negative;

    \item [(ii)] If $f_{i;u}^t=0$ for any $i\in[1,n]$, then $u$ is a cluster monomial in ${\bf x}_t$;
    \item[(iii)] Let $k$ be an integer in $[1,n]$ and $({\bf x}_{t'},\widetilde B_{t'})=\mu_k({\bf x}_t,\widetilde B_t)$. Suppose that
    $f_{k;u}^t\neq 0$ but $f_{i;u}^t=0$ for any $i\in[1,n]\backslash\{k\}$.  Then $u$ is a cluster monomial in ${\bf x}_{t'}$.
\end{itemize}

\end{corollary}

\begin{proof}
 Since $u$ is $[{\bf g}]$-pointed in $\mathcal U$, the Laurent expansion of $u$ with respect to ${\bf x}_t$ has a canonical expression:
    $$u={\bf x}_t^{{\bf g}_u^t}F_u^t(\hat y_{1;t},\ldots,\hat y_{n;t}),$$
    where ${\bf g}_u^t=(g_1^t,\ldots,g_m^t)^T$ is the value of $[{\bf g}]$ at vertex $t$ and $F_u^t$ is a polynomial in $\mathbb Z[y_1,\ldots,y_n]$ with constant term $1$.

(i) By Corollary \ref{cor:deg2} (ii), we know that $y_k^{[-g_{k}^t]_+}$ must be a monomial appearing in $F_u^t\in\mathbb Z[y_1,\ldots,y_n]$. Recall that $f_{k;u}^t$ is the maximal degree of $y_k$ in $F_u^t$. So we have $f_{k;u}^t\geq [-g_{k}^t]_+\geq 0$. Hence, if $f_{k;u}^t=0$, then $[-g_{k}^t]_+=0$. So $g_{k}^t$ is non-negative.

(ii)  Since $f_{i;u}^t=0$ for any $i\in[1,n]$, we get $F_u^t=1$ and $u={\bf x}_t^{{\bf g}_u^t}$.  On the other hand, by (i), we have $g_{i}^t\geq 0$ for any $i\in[1,n]$. Hence, $u={\bf x}_t^{{\bf g}_u^t}$ is a cluster monomial in ${\bf x}_t$.

    (iii) Since $f_{i;u}^t=0$ holds for any $i\in[1,n]\setminus \{k\}$, we get that $F_u^t$ is a polynomial in the single variable $y_k$. In this case, Corollary \ref{cor:deg2} (ii) implies $f_{k;u}^t=[-g_{k}^t]_+$. Since $f_{k;u}^t\neq 0$, we have $g_{k}^t<0$ and thus $f_{k;u}^t=-g_{k}^t$. Thus $u$ has the following form:
    $$u=x_{1;t}^{g_{1}^t}\cdots x_{m;t}^{g_{m}^t}(1+\sum_{0<d<-g_{k}^t}c_d\hat y_{k;t}^d+\hat y_{k;t}^{-g_{k}^t}),$$
    where $c_d\in\mathbb Z$. By the exchange relations, we have $x_{j;t'}=x_{j;t}$ for any $j\neq k$ and $x_{k;t}\cdot x_{k;t'}=M\cdot(1+\hat y_{k;t})$, where $M=\prod_{j=1}^mx_{j;t}^{[-b_{jk;t}]_+}$.
    By substituting these relations into the above equality and using the fact $u\in\mathbb Z[x_{1;t'}^{\pm 1},\ldots,x_{m;t'}^{\pm 1}]$, we obtain that  $$(1+\sum_{0<d<-g_{k}^t}c_d\hat y_{k;t}^d+\hat y_{k;t}^{-g_{k}^t})=(1+\hat y_{k;t})^{-g_{k}^t}$$ and the Laurent expansion of $u$ with respect to seed $({\bf x}_{t'},\widetilde B_{t'})$ is a Laurent monomial in $\mathbb Z[x_{1;t'}^{\pm 1},\ldots,x_{m;t'}^{\pm 1}]$. Thus the $F$-polynomial $F_u^{t'}$ of $u$ with respect to vertex $t'$ is $1$. In particular, the $f$-vector ${\bf f}_u^{t'}$ of $u$ with respect to vertex $t'$ is zero. By (ii), we know that $u$ is a cluster monomial in ${\bf x}_{t'}$.
\end{proof}

\section{Tropical invariant and $F$-invariant in cluster algebras}\label{sec4}

\subsection{Semifield homomorphisms arising from good elements}\label{sec41}
Let $(\mathcal S_X,\mathcal S_Y, \mathsf{\Lambda})$ be a Langlands-Poisson triple. In this subsection, we give a map from the tropical points in $\mathcal S_Y(\mathbb Z^{\rm max})$ to those in $\mathcal S_X(\mathbb Z^{\rm max})$. Thanks to this map, for each good element in a $\mathsf{\Lambda}$-upper cluster algebra, we introduce a semifield homomorphism, which will be used to define tropical invariant and $F$-invariant.

\begin{lemma}\label{lem:g-mutation}
Let $(\mathcal S_X,\mathcal S_Y)$ be a Langlands dual pair and  $\{{\bf g}^t\in\mathbb Z^m\mid t\in\mathbb T_n\}$ a collection of vectors in $\mathbb Z^m$ indexed by the vertices in $\mathbb T_n$. Then the collection  $\{{\bf g}^t\in\mathbb Z^m\mid t\in\mathbb T_n\}$ is a tropical point in
 $\mathcal S_Y(\mathbb Z^{\rm max})$ if and only if for any edge \begin{xy}(0,1)*+{t}="A",(10,1)*+{t'}="B",\ar@{-}^k"A";"B" \end{xy} in $\TT_n$, one has ${\bf g}^{t'}=E_{k,\varepsilon(g_{k}^t)}^{\widetilde B_t}{\bf g}^t$,
 where
\[\varepsilon(g_{k}^t)=\begin{cases}
   -,&\text{if }g_{k}^t\geq 0;\\
    +,&\text{if }g_{k}^t<0,
\end{cases} \]
and
$E_{k,\pm}^{\widetilde B_t}$ is as defined in \eqref{eqn:ebf}.
\end{lemma}

\begin{proof}
By definition,  the collection  $\{{\bf g}^t\in\mathbb Z^m\mid t\in\mathbb T_n\}$ is a tropical point in
 $\mathcal S_Y(\mathbb Z^{\rm max})$ if and only if the relations in \eqref{eqn:y-trop} hold. Since $(\mathcal S_X,\mathcal S_Y)$ is a Langlands dual pair, we have $\widehat B_t=-\widetilde B_t^T$. Thus $\hat b_{ij;t}=-b_{ji;t}$ for any $i\in[1,n]$ and $j\in[1,m]$. Then the relations in \eqref{eqn:y-trop} can be rewritten as
\begin{align}
    g_{i}^{t'}&=\begin{cases}-g_{k}^t,& \text{if}\;i=k;\\ g_{i}^t+[-b_{ik;t}]_+g_{k}^t+b_{ik;t}[g_{k}^t]_+,&\text{if}\;i\neq k. \end{cases} \nonumber\\
    &=\begin{cases}-g_{k}^t,& \text{if}\;i=k;\\ g_{i}^t+[b_{ik;t}]_+g_{k}^t,&\text{if}\;i\neq k\;\;\text{and }\;g_{k}^t\geq 0;\\
    g_{i}^t+[-b_{ik;t}]_+g_{k}^t,&\text{if}\;i\neq k\;\;\text{and }\;g_{k}^t< 0. 
    \end{cases}\nonumber
\end{align}
This is exactly the same with  ${\bf g}^{t'}=E_{k,\varepsilon(g_{k}^t)}^{\widetilde B_t}{\bf g}^t$.
\end{proof}

\begin{lemma}\label{lem:a-mutation}
 Let $\mathcal S_X$ be a cluster pattern of rank $n$ in $\mathbb F$ and  $\{{\bf a}^t\in\mathbb Z^m\mid t\in\mathbb T_n\}$ a
 collection of vectors in $\mathbb Z^m$ indexed by the vertices of $\mathbb T_n$. Then the collection  $\{{\bf a}^t\in\mathbb Z^m\mid t\in\mathbb T_n\}$ is a
 tropical point in  $\mathcal S_X(\mathbb Z^{\rm max})$ if and only if  for any edge \begin{xy}(0,1)*+{t}="A",(10,1)*+{t'}="B",\ar@{-}^k"A";"B" \end{xy} in $\TT_n$, one has ${\bf a}^{t'}= (E_{k,\varepsilon_0}^{\widetilde B_t})^T{\bf a}^t$, where
\[\varepsilon_0=\begin{cases}
   -,&\text{if }\sum_{j=1}^mb_{jk;t}a_{j}^t\geq 0;\\
    +,&\text{if }\sum_{j=1}^mb_{jk;t}a_{j}^t<0,
\end{cases} \]
and
$E_{k,\pm}^{\widetilde B_t}$ is as defined in \eqref{eqn:ebf}.
\end{lemma}

\begin{proof}
By definition, the collection  $\{{\bf a}^t\in\mathbb Z^m\mid t\in\mathbb T_n\}$ is a
 tropical point in  $\mathcal S_X(\mathbb Z^{\rm max})$ if and only if the relations in \eqref{eqn:x-trop} hold. The relations in \eqref{eqn:x-trop} can be rewritten as follows:
\[
a_{i}^{t'}=\begin{cases}-a_{k}^t+\sum_{j=1}^m[b_{jk;t}]_+a_{j}^t
,& \text{if}\;i=k\;\;\text{and }\;\sum_{j=1}^mb_{jk;t}a_{j}^t\geq 0
;\\
-a_{k}^t+\sum_{j=1}^m[-b_{jk;t}]_+a_{j}^t
,&  \text{if}\;i=k\;\;\text{and }\;\sum_{j=1}^mb_{jk;t}a_{j}^t<0
;\\
a_{i}^t,&\text{if}\;i\neq k. \end{cases}
\]
This is exactly the same with  ${\bf a}^{t'}= (E_{k,\varepsilon_0}^{\widetilde B_t})^T{\bf a}^t$.
\end{proof}

\begin{proposition}\label{pro:a-point}
Let $(\mathcal S_X,\mathcal S_Y, \mathsf{\Lambda})$ be a Langlands-Poisson triple.  Then $$[{\bf g}]=\{{\bf g}^t\in\mathbb Z^m\mid t\in\mathbb T_n\}\mapsto \{\Lambda_t{\bf g}^t\in\mathbb Z^m\mid t\in\mathbb T_n\}$$ defines a map from the tropical points in $\mathcal S_Y(\mathbb Z^{\rm max})$ to those in $\mathcal S_X(\mathbb Z^{\rm max})$.
\end{proposition}

\begin{proof}
Denote by ${\bf a}^t \coloneqq   \Lambda_t{\bf g}^t=(a_{1}^t,\ldots,a_{m}^t)^T$ for any vertex $t\in\mathbb T_n$.
We will use Lemma \ref{lem:a-mutation} to show the collection $\{{\bf a}^t\in\mathbb Z^m\mid t\in\mathbb T_n\}$ is a tropical point in $\mathcal S_X(\mathbb Z^{\rm max})$.

Let \begin{xy}(0,1)*+{t}="A",(10,1)*+{t'}="B",\ar@{-}^k"A";"B" \end{xy} be an edge in $\TT_n$.
Since $[{\bf g}]=\{{\bf g}^t\in\mathbb Z^m\mid t\in\mathbb T_n\}$ is a tropical point in $\mathcal S_Y(\mathbb Z^{\rm max})$ and by Lemma \ref{lem:g-mutation}, we have  ${\bf g}^{t'}=E_{k,\varepsilon(g_{k}^t)}^{\widetilde B_t}{\bf g}^t$,
 where
\[\varepsilon(g_{k}^t)=\begin{cases}
   -,&\text{if }g_{k}^t\geq 0;\\
    +,&\text{if }g_{k}^t<0.
\end{cases} \]
By Proposition \ref{pro:pm} (i), we have $\Lambda_{t'}=(E_{k,\varepsilon(g_{k}^t)}^{\widetilde B_t})^T\Lambda_t E_{k,\varepsilon(g_{k}^t)}^{\widetilde B_t}$.
Hence,
\begin{eqnarray}
{\bf a}^{t'}&\xlongequal{\text{by def.}}&\Lambda_{t'}{\bf g}^{t'}=\left((E_{k,\varepsilon(g_{k}^t)}^{\widetilde B_t})^T\Lambda_tE_{k,\varepsilon(g_{k}^t)}^{\widetilde B_t}\right)\cdot \left(E_{k,\varepsilon(g_{k}^t)}^{\widetilde B_t}{\bf g}^t\right)\nonumber\\
&\xlongequal{(E_{k,\pm }^{\widetilde B_t})^2=I_m}&(E_{k,\varepsilon(g_{k}^t)}^{\widetilde B_t})^T(\Lambda_t{\bf g}^t)=(E_{k,\varepsilon(g_{k}^t)}^{\widetilde B_t})^T{\bf a}^t. \nonumber
\end{eqnarray}

Now we prove $\varepsilon(g_{k}^t)=\varepsilon_0$, where $\varepsilon_0$ is as given in Lemma \ref{lem:a-mutation}. By the definitions of $\varepsilon(g_{k}^t)$ and  $\varepsilon_0$, it suffices to show that
$g_{k}^t\geq 0$ if and only if $\sum_{j=1}^mb_{jk;t}a_{j}^t\geq 0$.
Since $(\widetilde B_t, \Lambda_t)$ is a compatible pair, we know that
$\widetilde B_t^T\Lambda_t$ has the form
$$\widetilde B_t^T\Lambda_t=(S\mid {\bf 0}),$$ where $S=diag(s_1,\ldots,s_n)$ is a diagonal matrix  with $s_i>0$, $i\in [1,n]$ and ${\bf 0}$ is the $n\times(m-n)$ zero matrix.
Thus we have $\widetilde B_t^T{\bf a}^t=\widetilde B_t^T\Lambda_t{\bf g}^t=(S\mid {\bf 0}){\bf g}^t$. By comparing the $k$th component on both sides of the equality, we get
\[\sum_{j=1}^mb_{jk;t}a_{j}^t=s_kg_{k}^t.\]
Since $s_k$ is a positive integer, we get that $g_{k}^t\geq 0$ if and only if $\sum_{j=1}^mb_{jk;t}a_{j}^t\geq 0$. So we have $\varepsilon(g_{k}^t)=\varepsilon_0$.

Hence, ${\bf a}^{t'}=(E_{k,\varepsilon(g_{k}^t)}^{\widetilde B_t})^T{\bf a}^t=(E_{k,\varepsilon_0}^{\widetilde B_t})^T{\bf a}^t$. By Lemma \ref{lem:a-mutation}, we know the collection $$\{{\bf a}^t=\Lambda_t{\bf g}^t\in\mathbb Z^m\mid t\in\mathbb T_n\}$$ is a tropical point in $\mathcal S_X(\mathbb Z^{\rm max})$.
This completes the proof.
\end{proof}

\begin{proposition}\label{pro:beta-map}
    Let $(\mathcal S_X,\mathcal S_Y, \mathsf{\Lambda})$ be a Langlands-Poisson triple and let $\mathcal U$ be the corresponding $\mathsf{\Lambda}$-upper cluster algebra. Then any good element $u\in \mathcal U$ defines a unique semifield homomorphism 
  $$\beta_{u}:\; \mathbb F_{>0} :=   \mathbb Q_{\rm sf}(x_{1;t_0},\ldots,x_{m;t_0})\rightarrow \mathbb Z^{\rm max}$$ such that $\beta_{u}({\bf x}_t)=(\Lambda_t{\bf g}_u^t)^T$ for any vertex $t\in\mathbb T_n$, where ${\bf g}_u^t\in\mathbb Z^m$ is the extended $g$-vector of $u$ with respect to vertex $t$.
\end{proposition}

\begin{proof}
Since $u$ is a good element, 
we know that the collection  $$[{\bf g}_{u}] \coloneqq   \{{\bf g}_{u}^t\in\mathbb Z^m\mid t\in\mathbb T_n\}$$ is a tropical point in $\mathcal S_Y(\mathbb Z^{\rm max})$. Then by Proposition \ref{pro:a-point}, the collection $$[{\bf a}] \coloneqq   \{{\bf a}^t=\Lambda_t{\bf g}_{u}^{t}\in\mathbb Z^m\mid t\in\mathbb T_n\}$$ is a tropical point in $\mathcal S_X(\mathbb Z^{\rm max})$. By Corollary \ref{cor:bijection}, the  tropical point $[{\bf a}]\in \mathcal S_X(\mathbb Z^{\rm max})$  corresponds to a unique semifield homomorphism $$\beta_{u} \colon\mathbb F_{>0} \coloneqq   \mathbb Q_{\rm sf}(x_{1;t_0},\ldots,x_{m;t_0})\rightarrow \mathbb Z^{\rm max}$$ such that $\beta_{u}({\bf x}_t)=({\bf a}^t)^T=(\Lambda_t{\bf g}_u^t)^T$ for any vertex $t\in\mathbb T_n$. 
\end{proof}

\begin{remark}
 We remark that the above proof  still works if we replace the good elements by compatibly pointed elements. But in this paper, we only use good elements.   
\end{remark}

  \subsection{Tropical invariant and $F$-invariant}\label{sec42}

In this subsection, we fix a Langlands-Poisson triple $(\mathcal S_X,\mathcal S_Y,\mathsf{\Lambda})$, and let $\mathcal U$ be the corresponding $\mathsf{\Lambda}$-upper cluster algebra. Denote by $S=diag(s_1,\ldots,s_n)$ the type of the compatible pair $(\widetilde B_{t_0},\Lambda_{t_0})$, namely, $\widetilde B_{t_0}^T\Lambda_{t_0}=(S\mid {\bf 0})$, where ${\bf 0}$ is the $n\times (m-n)$ zero matrix.
By Proposition \ref{pro:fullrank}, the mutation matrices of $\mathcal U$ always have full rank.

We denote by $\mathcal U^{[{\bf g}]}$ the set of $[{\bf g}]$-good elements in $\mathcal U$ and by
$$\mathcal U^{\gd} \coloneqq   \bigcup_{[{\bf g}]\in\mathcal S_Y(\mathbb Z^{\rm max})}\mathcal U^{[\bf g]}$$ the set of all good elements in $\mathcal U$. Notice that the union $\bigcup_{[{\bf g}]\in\mathcal S_Y(\mathbb Z^{\rm max})}\mathcal U^{[\bf g]}$ is  disjoint.

\begin{definition}[Tropical invariant and $F$-invariant] \label{def:f-invariant}
Let $\mathcal U$ be a $\mathsf{\Lambda}$-upper cluster algebra. 

(i) For  
 $u,u'\in \mathcal U^{\gd}$, the {\em tropical invariant} $\langle u,u'\rangle_{\rm trop}$ of the pair $(u,u')$
 is  defined by 
 $$\langle u,u'\rangle_{\rm trop} \coloneqq   \beta_{u'}(u)\;\in \mathbb Z,$$ where
 $\beta_{u'}\colon\mathbb F_{>0} \coloneqq   \mathbb Q_{\rm sf}(x_{1;t_0},\ldots,x_{m;t_0})\rightarrow \mathbb Z^{\rm max}$ is the semifield homomorphism defined in Proposition \ref{pro:beta-map}. The pairing $
    \langle-,-\rangle_{\rm trop}\colon\mathcal U^{\gd}
    \times \mathcal U^{\gd} \rightarrow \mathbb Z$
is called the {\em tropical pairing} on the $\mathsf{\Lambda}$-upper cluster algebra $\mathcal U$.

(ii) For  
 $u,u'\in \mathcal U^{\gd}$, the {\em $F$-invariant} $(u\mid\mid u')_F$ of the pair $(u,u')$
 is defined by
 $$(u\mid\mid u')_F \coloneqq   \beta_{u'}(u)+\beta_u(u')=\langle u,u'\rangle_{\rm trop}+\langle u',u\rangle_{\rm trop}.$$
 In the case $(u\mid\mid u')_F=0$, we say that $u$ and $u'$ are {\em $F$-compatible}.
\end{definition}
\begin{remark}
   Notice that in general $\langle u,u'\rangle_{\rm trop}\neq \langle u',u\rangle_{\rm trop}$, but we always have $(u\mid\mid u')_F=(u'\mid\mid u)_F$, by definition.
\end{remark}

 Given a non-zero polynomial $$F=\sum_{{\bf v}\in\mathbb N^n}c_{{\bf v}}{\bf y}^{{\bf v}}\in\mathbb Z[y_1,\ldots,y_n]$$ and a vector ${\bf r}\in \mathbb Z^n$, we denote by
$$F[{\bf r}] \coloneqq   \max\{{\bf v}^T{\bf r}\mid c_{{\bf v}}\neq 0\}\in\mathbb Z.$$
We call the map $F[-]\colon \mathbb Z^n\rightarrow\mathbb Z$ a {\em tropical polynomial}.

\begin{example}
 Take $F=1+y_1+y_1y_2\in\mathbb Z[y_1,y_2]$ and ${\bf r}=\begin{bmatrix}
     -2\\1
 \end{bmatrix}$. Then
 \begin{eqnarray}
     F[{\bf r}]=\max\{
     \begin{bmatrix}
         0,0
     \end{bmatrix}\begin{bmatrix}
     -2\\1
 \end{bmatrix}, \begin{bmatrix}
         1,0
     \end{bmatrix}\begin{bmatrix}
     -2\\1
 \end{bmatrix}, \begin{bmatrix}
         1,1
     \end{bmatrix}\begin{bmatrix}
     -2\\1
 \end{bmatrix}
     \}=\max\{0,-2,-1\}=0.
     \nonumber
 \end{eqnarray}
\end{example}

\begin{proposition}\label{pro:trop}
The following statements hold.
\begin{itemize}
    \item [(i)]  If $F\in\mathbb Z[y_1,\ldots,y_n]$ has constant term $1$, then  $F[{\bf r}]\geq 0$ for any ${\bf r}\in\mathbb Z^n$.
    \item[(ii)] If $0\neq F_1,F_2\in\mathbb Z_{\geq 0}[y_1,\ldots,y_n]$, then  $(F_1F_2)[{\bf r}]=F_1[{\bf r}]+F_2[{\bf r}]$
for any ${\bf r}\in\mathbb Z^n$.
\end{itemize}
\end{proposition}

\begin{proof}
(i) This is clear.  (ii)  Fix ${\bf r}=(r_1,\ldots,r_n)^T\in\mathbb Z^n$ and let $$\pi_{\bf r}\colon \mathbb Q_{\rm sf}(y_1,\ldots,y_n)\rightarrow\mathbb Z^{\max}$$ be the semifield homomorphism induced by $y_i\mapsto r_i$. For any $0\neq F\in\mathbb Z_{\geq 0}[y_1,\ldots,y_n]$, we have $F\in \mathbb Q_{\rm sf}(y_1,\ldots,y_n)$ and  $\pi_{\bf r}(F)=F[{\bf r}]$. Since $0\neq F_1,F_2\in\mathbb Z_{\geq 0}[y_1,\ldots,y_n]$, we have $\pi_{\bf r}(F_1F_2)=\pi_{\bf r}(F_1)+\pi_{\bf r}(F_2)$,  \ie $(F_1F_2)[{\bf r}]=F_1[{\bf r}]+F_2[{\bf r}]$. \end{proof}

The following result shows that for each $\mathsf{\Lambda}$-seed $({\bf x}_w,\widetilde B_w, \Lambda_w)$ of $\mathcal U$, we have an explicit formula to calculate the tropical invariant and $F$-invariant.  
\begin{theorem} \label{thm:mutation-inv}
Let $u$ and $u'$ be two good elements in a $\mathsf{\Lambda}$-upper cluster algebra $\mathcal U$. Let
\[ u={\bf x}_w^{{\bf g}_u^w}F_u^w(\hat y_{1;w},\ldots,\hat y_{n;w}),\;\;\;u'={\bf x}_w^{{\bf g}_{u'}^w}F_{u'}^w(\hat y_{1;w},\ldots,\hat y_{n;w})
\]
be the canonical expressions of $u$ and $u'$ with respect to any vertex $w\in \mathbb T_n$. Then we have
\vspace{1mm}
\begin{itemize}
    \item [(i)] $\langle u,u'\rangle_{\rm trop}= ({\bf g}_u^w)^T\Lambda_w{\bf g}_{u'}^w+F_u^w[(S\mid {\bf 0}){\bf g}_{u'}^w]$.\vspace{1.5mm}
    \item[(ii)] $(u\mid\mid u')_F= F_u^w[(S\mid {\bf 0}){\bf g}_{u'}^w]+F_{u'}^w[(S\mid {\bf 0}){\bf g}_{u}^w]$. In particular,  $(u\mid\mid u')_F\in\mathbb Z_{\geq 0}$.\vspace{1.5mm}
    \item[(iii)] $(u\mid\mid u)_F=0$ whenever $u$ is a cluster monomial.
\end{itemize}
\end{theorem}
\begin{proof}
(i) By definition, $\langle u,u'\rangle_{\rm trop}=\beta_{u'}(u)$, where $\beta_{u'}$ is the
unique semifield homomorphism $$\beta_{u'} \colon\mathbb F_{>0} \coloneqq   \mathbb Q_{\rm sf}(x_{1;t_0},\ldots,x_{m;t_0})\rightarrow \mathbb Z^{\rm max}$$ such that $\beta_{u'}({\bf x}_t)=(\Lambda_t{\bf g}_{u'}^t)^T$ for any vertex $t\in\mathbb T_n$. 
By applying the semifield homomorphism $\beta_{u'}$ to  $u={\bf x}_w^{{\bf g}_u^w}F_u^w(\widehat {\bf y}_w) \in\mathbb F_{>0}$, we have
$$\beta_{u'}(u)=\beta_{u'}({\bf x}_w^{{\bf g}_u^w})+\beta_{u'}(F_u^w(\widehat {\bf y}_w)).$$
Since $\beta_{u'}({\bf x}_w)=(\Lambda_w{\bf g}_{u'}^w)^T$, we have $$\beta_{u'}({\bf x}_w^{\bf h})=(\Lambda_w{\bf g}_{u'}^w)^T{\bf h}=[(\Lambda_w{\bf g}_{u'}^w)^T{\bf h}]^T={\bf h}^T\Lambda_w{\bf g}_{u'}^w,\quad \forall {\bf h}\in\mathbb Z^m.$$
We write $F_u^w({\bf y})=\sum_{{\bf v}\in\mathbb N^n}c_{{\bf v}}{\bf y}^{{\bf v}}\in\mathbb N[y_1,\ldots,y_n]$. Then  $F_u^w(\widehat {\bf y}_w)=\sum_{{\bf v}\in\mathbb N^n}c_{{\bf v}}{\bf x}_w^{\widetilde B_w{\bf v}}$. Thus 
\begin{align}
  \beta_{u'}(F_u^w(\widehat {\bf y}_w))&= \beta_{u'}(\sum_{{\bf v}\in\mathbb N^n}c_{{\bf v}}{\bf x}_w^{\widetilde B_w{\bf v}})=\max\{{\bf v}^T\widetilde B_w^T\Lambda_w{\bf g}_{u'}^w\mid c_{\bf v}\neq 0\}\nonumber\\
    &=\max\{{\bf v}^T(S\mid {\bf 0}){\bf g}_{u'}^w\mid c_{\bf v}\neq 0\}
     = F_u^w[(S\mid{\bf 0}){\bf g}_{u'}^w].\nonumber
\end{align}
Hence, we have
\[\langle u,u'\rangle_{\rm trop}=\beta_{u'}(u)=\beta_{u'}({\bf x}_w^{{\bf g}_u^w})+\beta_{u'}(F_u^w(\widehat {\bf y}_w))=({\bf g}_u^w)^T\Lambda_w{\bf g}_{u'}^w +F_u^w[(S\mid{\bf 0}){\bf g}_{u'}^w].\]

(ii) By (i) and the definition of $F$-invariant, we have 
\begin{align}
    (u\mid\mid u')_F&=\langle u,u' \rangle_{\rm trop}+\langle u',u \rangle_{\rm trop}\nonumber\\
    &=({\bf g}_u^w)^T\Lambda_w{\bf g}_{u'}^w+F_u^w[(S\mid {\bf 0}){\bf g}_{u'}^w]+ ({\bf g}_{u'}^w)^T\Lambda_w{\bf g}_{u}^w+F_{u'}^w[(S\mid {\bf 0}){\bf g}_{u}^w].\nonumber
\end{align}
Since $\Lambda_w$ is skew-symmetric, we know that $({\bf g}_u^w)^T\Lambda_w{\bf g}_{u'}^w+({\bf g}_{u'}^w)^T\Lambda_w{\bf g}_{u}^w=0$. Thus
$$(u\mid\mid u')_F=F_u^w[(S\mid {\bf 0}){\bf g}_{u'}^w]+F_{u'}^w[(S\mid {\bf 0}){\bf g}_{u}^w]. $$
 Since both $F_u^w$ and $F_{u'}^w$ are polynomials with constant term $1$, we have \[F_u^{w}[{\bf r}], \;F_{u'}^w[{\bf r}]\;\in\mathbb Z_{\geq 0}\]
for any ${\bf r}\in\mathbb Z^n$.
So   $(u\mid\mid u')_F=F_u^w[(S\mid {\bf 0}){\bf g}_{u'}^w]+F_{u'}^w[(S\mid {\bf 0}){\bf g}_{u}^w]\;\in\mathbb Z_{\geq 0}$. 

(iii) Say $u$ is a cluster monomial in seed $({{\bf x}_w,\widetilde B_w})$. Then $F_u^w=1$. Then by (ii), we see that $(u\mid\mid u)_F=0$.
\end{proof}

Since the $F$-invariant $(u\mid\mid u')_F$ can be written as 
\[(u\mid\mid u')_F= F_u^w[(S\mid {\bf 0}){\bf g}_{u'}^w]+F_{u'}^w[(S\mid {\bf 0}){\bf g}_{u}^w],\] 
we call the non-negative integer $F_u^w[(S\mid {\bf 0}){\bf g}_{u'}^w]$  the {\em partial $F$-invariant} of $(u,u')$ at vertex $w\in\mathbb T_n$.
\begin{remark}
     Notice that  $(S\mid {\bf 0}) {\bf g}_{u'}^w=S({\bf g}_{u'}^w)^\circ$, where $({\bf g}_{u'}^w)^\circ\in\mathbb Z^n$ is the principal part of ${\bf g}_u^w\in\mathbb Z^m$. This suggests that the notion of partial $F$-invariant at vertex $w\in\mathbb T_n$ and the $F$-invariant in some sense only depend on 
  the principal part of the mutation matrices (i.e., the exchange matrices) of $\mathcal A$ and a choice of the skew-symmetrizer $S$.
\end{remark}

\begin{remark}\label{rmk:fei}
Recently, Fei introduced and studied the tropical $F$-polynomials in \cite{fei_2019a,fei_2019b} from the viewpoint of representation theory of finite-dimensional algebras and  \cite[Theorem 3.22]{fei_2019b} suggests that the partial $F$-invariant coincides with the partial $E$-invariant for Jacobi-finite quivers with potentials. 
\end{remark}

\begin{proposition}\label{pro:f-degree}
 Let $x_{i;t}$ be a cluster variable
of $\mathcal U$, where $i\in[1,m]$ and $t\in\mathbb T_n$. Then for any good element $u\in\mathcal U$, we have
$$(x_{i;t}\mid\mid u)_F=\begin{cases}s_if_{i;u}^t,&i\in[1,n];\\
0,&i\in[n+1,m],\end{cases}
$$
where $s_i$ is the $(i,i)$-entry of the diagonal matrix $S=diag(s_1,\ldots,s_n)$ and
$f_{i;u}^t$ is the $i$th component of the $f$-vector of $u$ with respect vertex $t$.
\end{proposition}
\begin{proof}
We know that $F_{x_{i;t}}^t=1$ and
${\bf g}_{x_{i;t}}^t={\bf e}_i$, where ${\bf e}_i$ is the $i$th column of $I_m$.  Now we use the vertex $t$ to calculate $(x_{i;t}\mid\mid u)_F$. By using the formula in Theorem \ref{thm:mutation-inv} (ii), we have
\[ (x_{i;t}\mid\mid u)_F=0+F_u^t[(S\mid{\bf 0}){\bf e}_i].\]
Then the result follows.
\end{proof}

\begin{remark}\label{rmk:f-degree}
(i) The $F$-invariant $(-\mid\mid-)_F\colon\mathcal U^{\gd}\times \mathcal U^{\gd}\rightarrow\mathbb Z_{\geq 0}$ is a great improvement of the  $f$-compatibility degree introduced by Fu and Gyoda in \cite{Fu-Gyoda} and thus the compatibility degree introduced by Fomin and Zelevinsky in \cite{fz-2003y}.  The $f$-compatibility degree, defined using the components of $f$-vectors of cluster variables, is a function
 $$(-\mid\mid -)_f\colon\mathcal X \times \mathcal X\rightarrow\mathbb Z_{\geq 0},$$
 where $\mathcal X$ is the set of unfrozen cluster variables in $\mathcal U$. By Proposition \ref{pro:f-degree}, when restricted to $\mathcal X$, the $F$-invariant and the $f$-compatibility degree are related as follows:
$$(x_{i;t}\mid\mid x_{j;w})_F=s_i(x_{i;t}\mid\mid x_{j;w})_f.$$

(ii) By construction, $F$-invariant is symmetric,  \ie $(u \mid\mid u')_F=(u'\mid\mid u)_F$ for any two good elements in $\mathcal U$. Applying to the case $u=x_{i;t}$ and $u=x_{j;w}$ with $i,j\in[1,n]$, we get
$$s_i(x_{i;t}\mid\mid x_{j;w})_f=(x_{i;t}\mid\mid x_{j;w})_F=(x_{j;w}\mid\mid x_{i;t})_F=s_j(x_{j;w}\mid\mid x_{i;t})_f
.$$
This recovers the result in \cite[Proposition 4.14 (2)]{Fu-Gyoda}.
\end{remark}

\begin{corollary}\label{cor:xu=0}
Let ${\bf x}_t$ be a cluster of $\mathcal U$ and $u\in\mathcal U^{\gd}$.  The following statements hold.
\begin{itemize}
\item[(i)] If $(x_{k;t}\mid\mid u)_F=0$ for some $k\in[1,n]$, then the $k$th component $g_{k}^t$ of ${\bf g}_u^t=\de^t(u)$ is non-negative.

    \item [(ii)] If $(x_{i;t}\mid\mid u)_F=0$ holds for any $i\in[1,n]$, then $u$ is a cluster monomial in ${\bf x}_t$.
    \item[(iii)] Let $k$ be an integer in $[1,n]$ and let $({\bf x}_{t'},\widetilde B_{t'})=\mu_k({\bf x}_t,\widetilde B_t)$. Suppose that $(x_{k;t}\mid\mid u)_F\neq 0$ but $(x_{i;t}\mid\mid u)_F=0$ for any $i\in[1,n]\backslash\{k\}$.  Then $(x_{i;t'}\mid\mid u)_F=0$ for any $i\in[1,n]$. In particular, $u$ is a cluster monomial in ${\bf x}_{t'}$.
\end{itemize}
\end{corollary}

\begin{proof}
The results follow from Proposition \ref{pro:f-degree} and Corollary \ref{cor:f-vector}.
\end{proof}

\begin{proposition}\label{pro:bigood}
  Let  $u$ be  a $[{\bf g}]$-bigood element in $\mathcal U$, and let ${\bf x}_t^{\bf h}=\prod_{j=1}^{m}x_{j;t}^{h_j}$ be a cluster monomial of $\mathcal U$,  i.e., ${\bf h}^\circ \coloneqq   (h_1,\ldots,h_n)^T\in\mathbb Z_{\geq 0}^n$. Then 
  \begin{itemize}
      \item [(i)] $({\bf x}_t^{\bf h}\mid\mid u)_F=\sum_{j=1}^nh_{j}\cdot (x_{j;t}\mid\mid u)_F$. \vspace{1.5mm}
      \item[(ii)] $F_{u}^w[(S\mid {\bf 0}){\bf g}_{{\bf x}_t^{\bf h}}^w]=\sum_{j=1}^nh_{j}\cdot F_{u}^w[(S\mid {\bf 0}){\bf g}_{x_{j;t}}^w]$ for any vertex $w\in\mathbb T_n$.
  \end{itemize}
 
\end{proposition}
\begin{proof}
(i)  Consider the canonical expression of $u$ with respect to the vertex $t$,
     $$u={\bf x}_t^{{\bf g}_u^t}F_u^t(\hat y_{1;t},\ldots,\hat y_{n;t}),$$
    where ${\bf g}_u^t$ is the value of $[{\bf g}]$ at vertex $t$.  Since $u$ is a $[{\bf g}]$-bigood element, $F_u^t$ has the form: 
$$F_u^t=1+\text{middle terms }+y_1^{f_{1}^t}\cdots y_n^{f_{n}^t},$$
    where ${\bf f}_u^t=(f_{1}^t,\ldots,f_{n}^t)^T$ is the $f$-vector of $u$ with respect to vertex $t$.  Since ${\bf x}_t^{\bf h}$ is a cluster momomial, we know that $F_{{\bf x}_t^{\bf h}}^t=1$ and $(h_1,\ldots,h_n)^T\in\mathbb Z_{\geq 0}^n$. By using the cluster ${\bf x}_t$ to calculate $ ({\bf x}_t^{\bf h}\mid\mid u)_F$, we have 
 \begin{align}
     ( {\bf x}_t^{\bf h}\mid\mid u)_F&=0+F_u^t[(S\mid {\bf 0}){\bf h}]=F_u^t[(s_1h_1,\ldots,s_nh_n)^T]\nonumber\\
    &=f_{1}^ts_1h_1+\ldots +f_{n}^ts_nh_n=\sum_{j=1}^nh_{j}\cdot (x_{j;t}\mid\mid u)_F.\nonumber
 \end{align}

 (ii) Let $w$ be a vertex in $\mathbb T_n$. We have 
\begin{eqnarray}
    ( {\bf x}_t^{\bf h}\mid\mid u)_F&\xlongequal{\text{by (i)}}&\sum_{j=1}^nh_{j}\cdot (x_{j;t}\mid\mid u)_F\nonumber\\
    &=&\sum_{j=1}^nh_{j}\cdot \left(F_{x_{j;t}}^w[(S\mid {\bf 0}){\bf g}_{u}^w]+F_u^w[(S\mid {\bf 0}){\bf g}_{x_{j;t}}^w]\right)\nonumber\\
    &=&\sum_{j=1}^nh_{j}\cdot F_{x_{j;t}}^w[(S\mid {\bf 0}){\bf g}_{u}^w]+\sum_{j=1}^nh_{j}\cdot F_u^w[(S\mid {\bf 0}){\bf g}_{x_{j;t}}^w]\nonumber\\
     &=&F_{{\bf x}_t^{\bf h}}^w[(S\mid {\bf 0}){\bf g}_{u}^w]+\sum_{j=1}^nh_{j}\cdot F_u^w[(S\mid {\bf 0}){\bf g}_{x_{j;t}}^w].\nonumber
\end{eqnarray}
The last equality comes from the facts that $F_{{\bf x}_t^{\bf h}}^w=\prod_{j=1}^m(F_{x_{j;t}}^w)^{h_j}=\prod_{j=1}^n(F_{x_{j;t}}^w)^{h_j}$ and that for any two polynomials $F_1,F_2\in\mathbb Z_{\geq 0}[y_1,\ldots,y_n]$ and any vector ${\bf r}\in\mathbb Z^n$, we have $$(F_1F_2)[{\bf r}]=F_1[{\bf r}]+F_2[{\bf r}],$$ by Proposition \ref{pro:trop} (ii).

On the other hand, we know that 
\[({\bf x}_t^{\bf h}\mid\mid u)_F=F_{{\bf x}_t^{\bf h}}^w[(S\mid {\bf 0}){\bf g}_{u}^w]+F_u^w[(S\mid {\bf 0}){\bf g}_{{\bf x}_t^{\bf h}}^w].\]
By comparing the two equalities, we obtain 
 $F_{u}^w[(S\mid {\bf 0}){\bf g}_{{\bf x}_t^{\bf h}}^w]=\sum_{j=1}^nh_{j}\cdot F_{u}^w[(S\mid {\bf 0}){\bf g}_{x_{j;t}}^w]$.
\end{proof}

\begin{lemma}\label{lem:xz=0}
    Let $x$ and $z$ be two unfrozen cluster variables of $\mathcal U$. If $(x\mid\mid z)_F=0$, then they are contained in the same cluster.
\end{lemma}
\begin{proof}
    Since $(x\mid\mid z)_F=0$ and by Remark \ref{rmk:f-degree},  the $f$-compatibility degree $(x\mid\mid z)_f$  is zero. Then the result follows from \cite[Theorem 3.3]{Fu-Gyoda}. 
    
Let us give another proof here. Consider the tropical point $$[{\bf g}_z] \coloneqq   \{{\bf g}_z^t\in\mathbb Z^m\mid t\in\mathbb T_n\}\in\mathcal S_Y(\mathbb Z^{\rm max})$$ corresponding to the cluster variable $z$. 
Since $(x\mid\mid z)_F=0$ and by Corollary \ref{cor:xu=0} (i), we know that the $k$th component $g_{k;z}^t$ of ${\bf g}_z^t$ satisfies $g_{k;z}^t\geq 0$ whenever $x=x_{k;t}$ for a vertex $t$. 

On the other hand, by \cite[Theorem 8]{cao-li-2020}, we know that there exists a seed $({\bf x}_w,\widetilde B_w)$ containing $x$ (say $x=x_{k;w}$) such that $g_{i;z}^w\geq 0$ for any $i\in[1,n]$ with $x_{i;t}\neq x$,  \ie for any $i\in[1,n]\setminus\{k\}$. Thus for the vertex $w$, we have $g_{j;z}^w\geq 0$ for $j=1,\ldots,n$,  \ie the principal part $({\bf g}_z^t)^\circ$ of ${\bf g}_z^t$ belongs to $\mathbb Z_{\geq 0}^n$. Thus ${\bf x}_w^{{\bf g}_z^w}$ is a cluster monomial of $\mathcal U$. Hence, the two cluster monomials ${\bf x}_w^{{\bf g}_z^w}$ and $z$ have the same extended $g$-vector ${\bf g}_z^w$ with respect to vertex $w$. It is well-known that different cluster monomials have different extended $g$-vectors,  \confer\cite{GHKK18}. So we must have $z={\bf x}_w^{{\bf g}_z^w}$. This implies that $z$ is a cluster variable in ${\bf x}_w$. In particular, $z$ and $x=x_{k;w}$ are contained in the same cluster ${\bf x}_w$.
\end{proof}

\begin{lemma}\label{lem:cluster}
Let $\{z_1,\ldots,z_p\}$ be a subset of unfrozen cluster variables of $\mathcal U$. If $(z_i\mid\mid z_j)_F=0$ for any $i,j\in[1,p]$, then $\{z_1,\ldots,z_p\}$ is a subset of some cluster of $\mathcal U$.
\end{lemma}
\begin{proof}
By Lemma \ref{lem:xz=0}, we know that any two cluster variables in $\{z_1,\ldots,z_p\}$ are contained in the same cluster. Then the result follows from \cite[Theorem 13]{cao-li-2020}.
\end{proof}

The following result says that $F$-invariant is a very concrete and easily computable numerical invariant to distinguish whether the product of two cluster monomials is still a cluster monomial.

\begin{theorem}
    \label{pro:uu'}
Let $u$ and $u'$ be two cluster monomials in $\mathcal U$. Then the product $u u'$ is still a cluster monomial in $\mathcal U$ if and only if $(u\mid\mid u')_F=0$,  i.e., $$F_u^t[(S\mid {\bf 0}){\bf g}_{u'}^t]+F_{u'}^t[(S\mid {\bf 0}){\bf g}_u^t]=0,$$ 
for some (equivalently, any) vertex $t\in\mathbb T_n$.
\end{theorem}

\begin{proof}
``$\Longrightarrow$": Suppose that the product $uu'$ is  a cluster monomial in $\mathcal U$. Say, $uu'$ is a cluster monomial in cluster ${\bf x}_t$. Then both $u$ and $u'$ can be viewed as cluster monomials in ${\bf x}_t$. Thus $F_u^t=1=F_{u'}^t$. So we have $(u\mid\mid u')_F=F_u^t[(S\mid {\bf 0}){\bf g}_{u'}^t]+F_{u'}^t[(S\mid {\bf 0}){\bf g}_u^t]=0$.

``$\Longleftarrow$": Suppose  $(u\mid\mid u')_F=0$. We need to show that $uu'$ is a cluster monomial in $\mathcal U$. Since $u$ and $u'$ are cluster monomials in $\mathcal U$, we can assume $u={\bf x}_{s}^{\bf h}$ and $u'={\bf x}_{s'}^{\bf r}$, where $s,s'\in\mathbb T_n$ and ${\bf h}=(h_1,\ldots,h_m)^T, {\bf r}=(r_1,\ldots,r_m)^T\in\mathbb Z^m$ with $h_i,r_i\geq 0$ for any $i\in[1,n]$.

By Proposition \ref{pro:ghkk} and Proposition \ref{pro:bigood}, we have
$$(u\mid\mid u')_F=\sum_{i=1}^n\sum_{j=1}^nh_ir_j(x_{i;s}\mid\mid x_{j;s'})_F.$$ Since all the numbers $h_i,r_j$ and $(x_{i;s}\mid\mid x_{j;s'})_F$ are non-negative for $i,j\in[1,n]$, the assumption $(u\mid\mid u')_F=0$ implies  $h_ir_j(x_{i;s}\mid\mid x_{j;s'})_F=0$ for any $i,j\in[1,n]$.
Hence, $(x_{i;s}\mid\mid x_{j;s'})_F=0$ for any $i,j\in[1,n]$ with $h_ir_j\neq 0$.

Let $V=\{x_{i;s}\mid i\in[1,n], h_i\neq 0\}$ and  $V'=\{x_{j;s'}\mid j\in[1,n], r_j\neq 0\}$. We have  $(z\mid\mid z')_F=0$ for any $z\in V$ and $z'\in V'$. Since $V$ and $V'$ are subsets of clusters of $\mathcal U$, we have $$(z_1\mid\mid z_2)_F=0\;\;\;\text{and}\;\;\; (z_1'\mid\mid z_2')_F=0$$ for any $z_1,z_2\in V$ and $z_1',z_2'\in V'$.

Hence, $(z\mid\mid z')_F=0$ holds for
any two cluster variables $z,z'$ in $V\cup V'$.  Then by Lemma \ref{lem:cluster}, we know that $V\cup V'$ is a subset of some cluster ${\bf x}_t$ of $\mathcal U$. Hence, the product $uu'$ is a cluster monomial in ${\bf x}_t$.
\end{proof}

\begin{example}\label{ex:A2-3}
    Let continue with  Example \ref{ex:A2-2}. Take $\Lambda=\begin{bmatrix}
    0&1\\-1&0
\end{bmatrix}= \widetilde B$. We know that $\widetilde B^T\Lambda=(S\mid {\bf 0})=I_2$ and the canonical expressions of $x_3,x_4,x_5$ with respect to the initial seed $t_0=({\bf x},\widetilde B)$ are given as follows:
\[x_3={x_1^{-1}x_2\cdot(1+\widehat y_1),}\;\;\;  x_4={x_1^{-1}\cdot (1+\widehat y_1+\widehat y_1\widehat y_2),}\;\;\;
    x_5={x_2^{-1}\cdot (1+\widehat y_2)}.\]
Thus we have
\begin{align}
    (x_3\mid\mid x_4)_F&=F_{x_3}^{t_0}[(S\mid{\bf 0}){\bf g}_{x_4}^{t_0}]+F_{x_4}^{t_0}[(S\mid{\bf 0}){\bf g}_{x_3}^{t_0}]\nonumber\\
    &=(1+y_1)\begin{bmatrix}
        -1\\0
    \end{bmatrix}+(1+y_1+y_1y_2)\begin{bmatrix}
        -1\\
        1
    \end{bmatrix}\nonumber\\
    &=\max\{0,\;-1\}+\max\{0,\;-1,\;0\}\nonumber\\
    &=0,\nonumber\\
    (x_3\mid\mid x_5)_F&=(1+y_1)\begin{bmatrix}
        0\\
        -1
    \end{bmatrix}+(1+y_2)\begin{bmatrix}
        -1\\
        1
    \end{bmatrix}\nonumber\\
    &=\max\{0,\;0\}+\max\{0,\;1\}\nonumber\\
    &=1.\nonumber
\end{align}
Then by Theorem \ref{pro:uu'}, we know that $x_3x_4$ is a cluster monomial, while $x_3x_5$ is not.
\end{example}

Recall that the {\em Newton polytope} of a  polynomial $F=\sum_{{\bf v}\in\mathbb N^n}c_{{\bf v}}{\bf y}^{{\bf v}}\in\mathbb Z[y_1,\ldots,y_n]$ is
 the convex hull of $\{{\bf v}\in\mathbb N^n\mid c_{\bf v}\neq 0\}$ in $\mathbb R^n$.

\begin{corollary}
Let $u$ and $u'$ be two cluster monomials in $\mathcal U$ and $t$ a vertex of $\mathbb T_n$.  If the two $F$-polynomials  $F_u^t$ and $F_{u'}^t$ have the same Newton polytope, then the product $uu'$ is a cluster monomial of $\mathcal U$.
\end{corollary}

\begin{proof}
Since $u$ and $u'$ are cluster monomials and by Theorem \ref{thm:mutation-inv} (iii), we have
\[
(u\mid\mid u)_F=2F_u^t[(S\mid {\bf 0}){\bf g}_u^t]=0\;\;\;\text{and}\;\;\;(u'\mid\mid u')_F=2F_{u'}^t[(S\mid {\bf 0}){\bf g}_{u'}^t]=0.
\]
We get $F_u^t[(S\mid {\bf 0}){\bf g}_u^t]=0=F_{u'}^t[(S\mid {\bf 0}){\bf g}_{u'}^t]$.
Since $F_u^t$ and $F_{u'}^t$ have the same Newton polytope, we have $F_u^t[{\bf h}]=F_{u'}^t[{\bf h}]$ for any ${\bf h}\in\mathbb Z^n$. Hence,
we have
\begin{align}
(u\mid\mid u')_F&=F_u^t[(S\mid {\bf 0}){\bf g}_{u'}^t]+F_{u'}^t[(S\mid {\bf 0}){\bf g}_u^t]\nonumber\\
&=F_{u'}^t[(S\mid {\bf 0}){\bf g}_{u'}^t]+F_u^t[(S\mid {\bf 0}){\bf g}_u^t]=0.\nonumber
\end{align}
Then by Theorem \ref{pro:uu'}, we know that the product $uu'$ is  a cluster monomial.
\end{proof}

\begin{definition}\label{def:sign-coherent}
Two tropical points $[{\bf g}]=\{{\bf g}^t\in\mathbb Z^m\mid t\in\mathbb T_n\}$ and  $[{\bf g}']=\{({\bf g}')^t\in\mathbb Z^m\mid t\in\mathbb T_n\}$ in $\mathcal S_Y(\mathbb Z^{\rm max})$
are said to be {\em sign-coherent}, if for any vertex $t\in\mathbb T_n$ and any $k\in[1,n]$, we have $g_{k}^t\cdot (g_{k}')^t\geq 0$, where $g_{k}^t$ and $(g_{k}')^t$ are the $k$th components of  ${\bf g}^t$ and $({\bf g}')^t$.
\end{definition}
\begin{lemma}\label{lem:sum}
Let $[{\bf g}]=\{{\bf g}^t\in\mathbb Z^m\mid t\in\mathbb T_n\}$ and  $[{\bf g}']=\{({\bf g}')^t\in\mathbb Z^m\mid t\in\mathbb T_n\}$ be two tropical points in $\mathcal S_Y(\mathbb Z^{\rm max})$. If $[{\bf g}]$ and $[{\bf g}']$ are sign-coherent, then the collection $\{{\bf g}^t+({\bf g}')^t\in\mathbb Z^m\mid t\in\mathbb T_n\}$ is still a tropical point in $\mathcal S_Y(\mathbb Z^{\rm max})$. We denote it by $[{\bf g}\uplus {\bf g}']$.
\end{lemma}
\begin{proof}
It suffices to show that the collection  $\{{\bf g}^t+({\bf g}')^t\in\mathbb Z^m\mid t\in\mathbb T_n\}$ satisfies the relations in \eqref{eqn:y-trop}. Thanks to the sign-coherence assumption, this is clear.
\end{proof}

\begin{theorem}\label{thm:sign-coherent}
Let $u$ be a $[{\bf g}]$-good element in $\mathcal U$ and $u'$  a $[{\bf g}']$-good element in $\mathcal U$. Suppose that $(u\mid\mid u')_F=0$. Then the following statements hold.
\begin{itemize}
\item[(i)] The two tropical points $[{\bf g}]$ and $[{\bf g}']$ are sign-coherent. In particular, the tropical point $[{\bf g}\uplus{\bf g}']$ can be defined.
\item[(ii)] The product $uu'$ is a $[{\bf g}\uplus{\bf g}']$-good element in $\mathcal U$.
\end{itemize}
\end{theorem}

\begin{proof}
(i) We write $[{\bf g}]=\{{\bf g}_u^t=(g_{1;u}^t,\ldots,g_{m;u}^t)^T\in\mathbb Z^m\mid t\in\mathbb T_n\}$ and  $[{\bf g}']=\{{\bf g}_{u'}^t=(g_{1;u'}^t,\ldots,g_{m;u'}^t)^T\in\mathbb Z^m\mid t\in\mathbb T_n\}$.
 Assume the contrary  that the two tropical points $[{\bf g}]$ and $[{\bf g}']$ are not sign-coherent to deduce a contradiction.
Then there exist some vertex $s\in\mathbb T_n$ and some $k\in[1,n]$ such that $g_{k;u}^s\cdot g_{k;u'}^s<0$. Without loss of generality, we can assume $g_{k;u}^s<0$ and $g_{k;u'}^s>0$. 

By Corollary \ref{cor:deg2}, we know that the monomial ${\bf y}^{[-g_{k;u}^s]_+{\bf e}_k}=y_k^{[-g_{k;u}^s]_+}$ appears in the $F$-polynomial $F_u^s\in\mathbb Z_{\geq 0}[y_1,\ldots,y_n]$ with coefficient $1$, where ${\bf e}_k$ is the $k$th column of $I_n$. Thus we have $$F_u^s[(S\mid {\bf 0}){\bf g}_{u'}^s]\geq ([-g_{k;u}^s]_+{\bf e}_k)^T(S\mid {\bf 0}){\bf g}_{u'}^s=[-g_{k;u}^s]_+\cdot s_kg_{k;u'}^s.$$
Then by $g_{k;u}^s<0$, $g_{k;u'}^s>0$ and $s_k>0$, we get $F_u^s[(S\mid {\bf 0}){\bf g}_{u'}^s]>0$.

By using the cluster ${\bf x}_s$ to calculate the $F$-invariant $(u\mid\mid u')_F$, we have
$$(u\mid\mid u')_F=F_u^s[(S\mid {\bf 0}){\bf g}_{u'}^s]+F_{u'}^s[(S\mid {\bf 0}){\bf g}_u^s]\geq F_u^s[(S\mid {\bf 0}){\bf g}_{u'}^s]>0,$$
which contradicts $(u\mid\mid u')_F=0$. Hence, the two tropical points $[{\bf g}]$ and $[{\bf g}']$ must be sign-coherent.

(ii) For any vertex $t$, consider the Laurent expansions of  $u$  and  $u'$ with respect to vertex $t$:
$$u={\bf x}_t^{{\bf g}_u^t}F_u^t(\hat y_{1;t},\ldots,\hat y_{n;t}),\;\;\; u'={\bf x}_t^{{\bf g}_{u'}^t}F_{u'}^t(\hat y_{1;t},\ldots,\hat y_{n;t}),$$
where $F_u^t$ and $F_{u'}^t$ are polynomials in $\mathbb Z_{\geq 0}[y_1,\ldots,y_n]$ with constant term $1$. Hence, we have
$$uu'={\bf x}_t^{{\bf g}_u^t+{\bf g}_{u'}^t}H_{uu'}^t(\hat y_{1;t},\ldots,\hat y_{n;t}),$$
where $H_{uu'}^t \coloneqq   F_u^tF_{u'}^t$ is a polynomial in $\mathbb Z_{\geq 0}[y_1,\ldots,y_n]$ with constant term $1$.

By (i) and Lemma \ref{lem:sum}, we know that $[{\bf g}\uplus{\bf g}'] \coloneqq   \{{\bf g}_u^t+{\bf g}_{u'}^t\in\mathbb Z^m\mid t\in\mathbb T_n\}$ forms a tropical point in $\mathcal S_Y(\mathbb Z^{\rm max})$. Then by the definition of good elements, we know that the product $uu'$ is a $[{\bf g}\uplus{\bf g}']$-good element in $\mathcal U$.
\end{proof}

\subsection{Log-canonical cluster variables are compatible}\label{sec43}
In this subsection, we show that if two cluster variables are log-canonical in a $\mathsf{\Lambda}$-upper cluster algebra, then they are contained in the same cluster.

As before, we fix a Langlands-Poisson triple $(\mathcal S_X,\mathcal S_Y,\mathsf{\Lambda})$, and let $\mathcal U$ be the corresponding $\mathsf{\Lambda}$-upper cluster algebra. Denote by $S=diag(s_1,\ldots,s_n)$ the type of the compatible pair $(\widetilde B_{t_0},\Lambda_{t_0})$, namely, $\widetilde B_{t_0}^T\Lambda_{t_0}=(S\mid {\bf 0})$, where ${\bf 0}$ is the $n\times (m-n)$ zero matrix.

Recall that we can define a Poisson bracket $\{-,-\}$ on the ambient field $\mathbb F=\mathbb Q(x_{1;t_0},\ldots,x_{m;t_0})$ using the $m\times m$ skew-symmetric matrix $\Lambda_{t_0}=(\lambda_{ij;t_0})$ at the rooted vertex $t_0$:
$$\{x_{i;t_0},x_{j;t_0}\} \coloneqq   \lambda_{ij;t_0}\cdot x_{i;t_0}x_{j;t_0}.$$
This Poisson bracket is compatible with the cluster pattern $\mathcal S_X$, that is, for any cluster ${\bf x}_t$ of $\mathcal S_X$, we have $$\{x_{i;t},x_{j;t}\}=\lambda_{ij;t}\cdot x_{i;t}x_{j;t},$$
where $\lambda_{ij;t}$ is the $(i,j)$-entry of $\Lambda_t$, \confer \cite{gsv-2003}, \cite[Remark 4.6]{bz-2005}.

Two elements $f_1,f_2$ in $\mathbb F$ are said to be {\em log-canonical} with respect to the Poisson bracket $\{-,-\}$, if $$\{f_1,f_2\}=cf_1f_2$$ for some $c\in\mathbb Q$.

Clearly, any two cluster variables in the same cluster are log-canonical. In  Theorem \ref{thm:log} (ii), we show the converse statement is also true.

\begin{lemma}\label{lem:f=0}
Let $w$ be a vertex of $\mathbb T_n$ and $u$ an element in $\mathbb F$ with the form
$$u={\bf x}_w^{\bf g}F(\hat y_{1;w},\ldots,\hat y_{n;w}),$$
where $F=\sum_{{\bf v}\in\mathbb N^n}c_{{\bf v}}{\bf y}^{{\bf v}}$ is a polynomial in $\mathbb Z[y_1,\ldots,y_n]$ with constant term $1$.
 If $u$ and $x_{k;w}$ are log-canonical for some $k\in[1,n]$, then $f_{k}=0$, where $f_k$ is the maximal degree of $y_{k}$ in $F$.
\end{lemma}

\begin{proof}
We know that $\Lambda_w=(\lambda_{ij;w})$ is the coefficient matrix of $\{-,-\}$ with respect to the cluster ${\bf x}_w$. Since $(\widetilde B_w,\Lambda_w)$ is a compatible pair of type $S$, we have $\widetilde B_w^T\Lambda_w=(S\mid {\bf 0})$. For $i\in[1,n]$ and $j\in[1,m]$, we denote by ${\bf e}_i$ the $i$th column of $I_n$ and $\widetilde {\bf e}_j$ the $j$th column of $I_m$. We have
\begin{align}
\{\hat y_{i;w},x_{j;w}\}&=\{{\bf x}_w^{\widetilde B_w{\bf e}_i},{\bf x}_w^{\widetilde {\bf e}_j}\}=\left((\widetilde B_w{\bf e}_i)^T\Lambda_w\widetilde {\bf e}_j\right)\hat y_{i;w}x_{j;w}
\nonumber\\
&=\left({\bf e}_i^T(\widetilde B_w)^T\Lambda_w\widetilde {\bf e}_j\right)\hat y_{i;w}x_{j;w}=\left({\bf e}_i^T(S\mid {\bf 0})\widetilde {\bf e}_j\right)\hat y_{i;w}x_{j;w}\nonumber\\
&=\begin{cases}s_j\cdot \hat y_{j;w}x_{j;w},&i=j;\\
0,&i\neq j.
\end{cases}\nonumber
\end{align}
Now we consider the integer $k\in[1,n]$ such that $u$ and $x_{k;w}$ are log-canonical. For ${\bf v}=(v_1,\ldots,v_n)^T\in\mathbb N^n$, we have
$$\{\widehat {\bf y}_w^{\bf v},x_{k;w}\}=s_k(v_k\widehat {\bf y}_w^{\bf v})x_{k;w}.$$
For simplicity, we write $F(\widehat {\bf y}_w)=F(\hat y_{1;w},\ldots,\hat y_{n;w})$.
We have $$\{F(\widehat {\bf y}_w),x_{k;w}\}=\{\sum_{{\bf v}\in\mathbb N^n}c_{{\bf v}}\widehat {\bf y}_w^{{\bf v}},x_{k;w}\}=s_k(\sum_{{\bf v}\in\mathbb N^n}c_{{\bf v}}v_k\widehat {\bf y}_w^{\bf v})x_{k;w}.$$
Now let us calculate $\{u,x_{k;w}\}=\{{\bf x}_w^{\bf g}F(\widehat {\bf y}_w),x_{k;w}\}$.
\begin{align}
\{{\bf x}_w^{\bf g}F(\widehat {\bf y}_w),x_{k;w}\}&=\{{\bf x}_w^{\bf g},x_{k;w}\}F(\widehat {\bf y}_{w})+{\bf x}_w^{\bf g}\{F(\widehat {\bf y}_{w}),x_{k;w}\}\nonumber\\
&=({\bf g}^T\Lambda_w\widetilde {\bf e}_k)\cdot {\bf x}_w^{\bf g}\cdot x_{k;w}\cdot F(\widehat {\bf y}_w)+{\bf x}_w^{\bf g}\cdot s_k(\sum_{{\bf v}\in\mathbb N^n}c_{{\bf v}}v_k\widehat {\bf y}_w^{\bf v})x_{k;w}.\nonumber
\end{align}
Since $u$ and $x_{k;w}$ are log-canonical, there exists some $c\in\mathbb Q$ such that
$$\{u,x_{k;w}\}=c\cdot ux_{k;w}=c\cdot {\bf x}_w^{\bf g}\cdot F(\widehat {\bf y}_w)\cdot x_{k;w}.$$
Thus we have the following equality in $\mathbb F$.
$$({\bf g}^T\Lambda_w\widetilde {\bf e}_k)\cdot {\bf x}_w^{\bf g}\cdot x_{k;w}\cdot F(\widehat {\bf y}_w)+{\bf x}_w^{\bf g}\cdot s_k(\sum_{{\bf v}\in\mathbb N^n}c_{{\bf v}}v_k\widehat {\bf y}_w^{\bf v})x_{k;w}= c\cdot {\bf x}_w^{\bf g}\cdot F(\widehat {\bf y}_w)\cdot x_{k;w}.$$
We get
\begin{eqnarray}\label{eqn:f=0}
s_k(\sum_{{\bf v}\in\mathbb N^n}c_{{\bf v}}v_k\widehat {\bf y}_w^{\bf v})=(c-{\bf g}^T\Lambda_w\widetilde {\bf e}_k)\cdot F(\widehat {\bf y}_w).
\end{eqnarray}

Assume by contradiction that the maximal degree $f_{k}$ of $y_k$ in $F=\sum_{{\bf v}\in\mathbb N^n}c_{{\bf v}}{\bf y}^{{\bf v}}$ is non-zero.
Then the left side of \eqref{eqn:f=0} is non-zero, which implies $c-{\bf g}^T\Lambda_w\widetilde {\bf e}_k\neq 0$. Now we compare the constant term on both sides of \eqref{eqn:f=0}. The constant term on the left side is $0$ while the constant term on the right side is $(c-{\bf g}^T\Lambda_w\widetilde {\bf e}_k)\cdot 1=c-{\bf g}^T\Lambda_w\widetilde {\bf e}_k\neq 0$. This is a contradiction. Hence, $f_{k}=0$.
\end{proof}

\begin{theorem}\label{thm:log}
Let  $u\in\mathcal U^{\gd}$ and $w\in\mathbb T_n$. The following statements hold.
\begin{itemize}
\item[(i)] If $u$ and $x_{k;w}$ are log-canonical for some $k\in[1,n]$, then  $(x_{k;w}\mid\mid u)_F=0$.

    \item[(ii)] If $u$ is a cluster monomial and it is log-canonical with $x_{k;w}$, then the product $u\cdot x_{k;w}$ is a cluster monomial of $\mathcal U$. In particular, if two cluster variables are log-canonical, then they are contained in the same cluster.
    \item[(iii)] If $u$ and $x_{k;w}$ are log-canonical for any $k\in[1,n]$, then $u$ is a cluster monomial in ${\bf x}_w$.
        \item[(iv)] Let  $({\bf x}_t,\widetilde B_t)=\mu_k({\bf x}_w,\widetilde B_w)$ with $k\in[1,n]$.
         Suppose that $u$ and $x_{k;w}$ are not log-canonical but $u$ and $x_{i;w}$ are log-canonical for any $i\in[1,n]\backslash\{k\}$,
             then $u$ is a cluster monomial in ${\bf x}_t$.
\end{itemize}
\end{theorem}

\begin{proof}
(i) This follows from Lemma \ref{lem:f=0} and Proposition \ref{pro:f-degree}.

(ii) By (i), we have $(x_{k;u}\mid\mid u)_F=0$. Since $u$ is a cluster monomial and by Theorem \ref{pro:uu'}, we obtain that the product $u\cdot x_{k;w}$ is a cluster monomial of $\mathcal U$.

Both (iii) and (iv) follow from (i) and Corollary \ref{cor:xu=0}.
\end{proof}

\subsection{$F$-invariant for cluster algebras with trivial coefficients}
In this subsection we show that the notion of $F$-invariant for a pair of cluster monomials can be defined for any upper cluster algebra with trivial coefficients, 
regardless of whether it is a $\mathsf{\Lambda}$-upper cluster algebra.

Recall that in the process of our discovery and definition of $F$-invariant, we work on $\mathsf{\Lambda}$-(upper) cluster algebras and thus (upper) cluster algebras of full rank, because the $F$-invariant is defined as the symmetrized sum of tropical invariant 
$$(u\mid\mid u')_F=\beta_{u'}(u)+\beta_u(u')=\langle u,u'\rangle_{\rm trop}+\langle u',u\rangle_{\rm trop}$$
and the tropical invariant heavily depends on the Poisson coefficient matrices $\mathsf{\Lambda}=\{\Lambda_t\mid t\in\mathbb T_n\}$. However, if we look at Theorem  \ref{thm:mutation-inv} (ii), we know that $(u\mid\mid u')_F$ is given as follows:
 $$(u\mid\mid u')_F= F_u^w[(S\mid {\bf 0}){\bf g}_{u'}^w]+F_{u'}^w[(S\mid {\bf 0}){\bf g}_{u}^w],$$ where $w$ is any vertex in $\mathbb T_n$. This formula eliminates the dependence of $(u\mid\mid u')_F$ on the Poisson coefficient matrices $\mathsf{\Lambda}=\{\Lambda_t\mid t\in\mathbb T_n\}$. Thus we can define $F$-invariant for any (upper) cluster algebra with trivial coefficients whenever we have well-defined $g$-vectors and $F$-polynomials.  

The philosophy in this subsection is that for each (upper) cluster algebra with trivial coefficients there is a natural $\mathsf{\Lambda}$-(upper) cluster algebra behind it.

\begin{setting}\label{setting}
(i) Let $B$ be an $n\times n$ skew-symmetrizable matrix with a fixed skew-symmetrizer $S=diag(s_1,\ldots,s_n)$. Let ${\bf z}=(z_1,\ldots,z_n)$ and ${\bf x}=(x_1,\ldots,x_m)$ be the ordered sets of indeterminates, where $m=2n$.

(ii)  Set $\widetilde B=\begin{bmatrix}
    B\\ I_n
\end{bmatrix}$ and $\Lambda=\begin{bmatrix}
    {\bf 0}&-S\\
    S&-SB
\end{bmatrix}$. Since $\widetilde B^T\Lambda=(S\mid {\bf 0})$, the pair $(\widetilde B,\Lambda)$ forms a compatible pair of type $S$. 

(iii) Let $\mathcal U$ be the $\mathsf{\Lambda}$-upper cluster algebra with initial $\mathsf{\Lambda}$-seed $({\bf x},\widetilde B,\Lambda)$ and let $\mathcal U^\circ$ be the upper cluster algebra with trivial coefficients whose initial seed is $({\bf z},B)$. Denote by $({\bf x}_t,\widetilde B_t,\Lambda_t)$ the $\mathsf{\Lambda}$-seed of $\mathcal U$ at vertex $t\in\mathbb T_n$ and by  $({\bf z}_t, B_t)$ the seed of $\mathcal U^\circ$ at vertex $t\in\mathbb T_n$, where ${\bf x}_t=(x_{1;t},\ldots,x_{m;t})$ and ${\bf z}_t=(z_{1;t},\ldots,z_{n;t})$.

(iv) 
For a vector  ${\bf h}=(h_1,\ldots,h_{m})^T\in\mathbb Z^{m}$, denote by 
 ${\bf h}^\circ \coloneqq   (h_1,\ldots,h_n)^T\in\mathbb Z^n$ the {\em principal part} of ${\bf h}$.
\end{setting}

\begin{theorem}[Separation formula, \cite{fomin_zelevinsky_2007}]
 Keep the above setting. Let $u^\circ=\prod_{i=1}^nz_{i;t}^{a_i}$ be a cluster monomial in $\mathcal U^\circ$ and $u=\prod_{i=1}^nx_{i;t}^{a_i}$ the corresponding cluster monomial in $\mathcal U$. Denote by ${\bf g}_u^w\in\mathbb Z^m$ and $F_u^w\in\mathbb Z[y_1,\ldots,y_n]$ the extended $g$-vector and $F$-polynomial of $u$ with respect to vertex $w\in\mathbb T_n$. Then the Laurent expansion of $u^\circ$ with respect to seed $({\bf z}_w,B_w)$ has a cannonical expression in terms of ${\bf g}_u^w$ and $F_u^w$: 
 \[ u^\circ={\bf z}_w^{({\bf g}_u^w)^\circ}F_{u}^w({\bf z}_w^{B_w{\bf e}_1},\ldots,{\bf z}_w^{B_w{\bf e}_n}),
 \]
 where $({\bf g}_u^w)^\circ\in\mathbb Z^n$ is the principal part of ${\bf g}_u^w\in\mathbb Z^m$ and ${\bf e}_k$ is the $k$th column of $I_n$. 
\end{theorem}
\begin{definition}
    Keep the notation in the above theorem. We call ${\bf g}_{u^\circ}^w \coloneqq   ({\bf g}_u^w)^\circ\in\mathbb Z^n$ and $F_{u^\circ}^w \coloneqq   F_u^w\in\mathbb Z[y_1,\ldots,y_n]$ the {\em $g$-vector} and {\em $F$-polynomial} of the cluster monomial $u^\circ\in\mathcal U^\circ$ with respect to vertex $w\in\mathbb T_n$.
\end{definition}

In the following proposition, we prove that $F$-invariant can be naturally defined for upper cluster algebras with trivial coefficients.
\begin{proposition}\label{pro:trivial-coef}
Let $B$ be an $n\times n$ skew-symmetrizable matrix with a fixed skew-symmetrizer $S$. Let $\mathcal U^\circ$ be the upper cluster algebra with trivial coefficients whose initial seed is $({\bf z}, B)$. Let $u^\circ$ and $v^\circ$ be two cluster monomials in  $\mathcal U^\circ$. Then the integer $$(u^\circ\mid\mid v^\circ)_F \coloneqq   F_{u^\circ}^w[S{\bf g}_{v^\circ}^w]+F_{v^\circ}^w[S{\bf g}_{u^\circ}^w]$$ only depends on $u^\circ$ and $v^\circ$, not on the choice of the vertex $w\in\mathbb T_n$.   
\end{proposition}
\begin{proof}
Let $\mathcal U$ be the $\mathsf{\Lambda}$-upper cluster algebra as in the Setting \ref{setting} and let $u,v$ be the cluster monomials of $\mathcal U$ corresponding to the cluster monomials $u^\circ,v^\circ$ of $\mathcal U^\circ$. By Theorem \ref{thm:mutation-inv} (ii), we have
 $$(u\mid\mid v)_F= F_u^w[(S\mid {\bf 0}){\bf g}_{v}^w]+F_{v}^w[(S\mid {\bf 0}){\bf g}_{u}^w],$$ where $w$ is any vertex in $\mathbb T_n$.  Clearly, we have $(S\mid{\bf 0}){\bf g}_u^w=S({\bf g}_u^w)^\circ$ and $(S\mid{\bf 0}){\bf g}_{v}^w=S({\bf g}_{v}^w)^\circ$, where $({\bf g}_u^w)^\circ,\; ({\bf g}_{v}^w)^\circ\in\mathbb Z^n$ are the principal parts of ${\bf g}_u^w, \;{\bf g}_v^w\in\mathbb Z^m$. Thus 
 $$(u\mid\mid v)_F= F_u^w[S({\bf g}_{v}^w)^\circ]+F_{v}^w[S({\bf g}_{u}^w)^\circ].$$
 By definition, we know that ${\bf g}_{u^\circ}^w=({\bf g}_u^w)^\circ, \;{\bf g}_{v^\circ}^w=({\bf g}_v^w)^\circ$ and $F_{u^\circ}^w=F_u^w,\; F_{v^\circ}^w=F_v^w$.
 So we have
 $$(u\mid\mid v)_F=F_{u^\circ}^w[S{\bf g}_{v^\circ}^w]+F_{v^\circ}^w[S{\bf g}_{u^\circ}^w]=(u^\circ\mid\mid v^\circ)_F,$$
 which is independent of the choice of the vertex $w\in\mathbb T_n$.
\end{proof}

\begin{remark}
 The key ingredients to define $F$-invariant is the (extended) $g$-vectors and $F$-polynomials. The $g$-vectors and $F$-polynomials are natural concepts in the setting of (upper) cluster algebras of full rank or $\mathsf{\Lambda}$-(upper) cluster algebras. In this setting, we can even define $g$-vectors and $F$-polynomials for many elements beyond cluster monomials. 
 
 In contrast, when we work on (upper) cluster algebras with trivial coefficents, we need to restrict to cluster monomials in order to have a natural definition of $g$-vectors and $F$-polynomials. On the other hand, we see that during the process of defining $g$-vectors and $F$-polynomials for an upper cluster algebra with trivial coefficients, we lift to an upper cluster algebra of full rank.

\end{remark}

\section{Comparison with ${\Lambda}$/$\mathfrak{d}$-invariant in monoidal cluster categorification}\label{sec:monoidal}

\subsection{Monoidal categorification of cluster algebras}
Let $\mathbf{k}$ be a base field and let $\mathcal C$ be a $\mathbf{k}$-linear monoidal abelian category  such that any object of $\mathcal C$ has finite length. We denote by $[M]$ the isomorphism class of an object $M\in\mathcal C$ and by $\Irr(\mathcal C)$ the set of isomorphism classes of simple objects in $\mathcal C$. Let $\mathcal K_0(\mathcal C)$ be the  Grothendieck group  of $\mathcal C$. We know that 
\begin{itemize}
    \item  $\mathcal K_0(\mathcal C)$ is a free $\mathbb Z$-module with a basis given by $\Irr(\mathcal C)$,  \ie $\mathcal K_0(\mathcal C)=\bigoplus_{[S]\in\Irr(\mathcal C)}\mathbb Z[S]$.
    \item  $\mathcal K_0(\mathcal C)$ is a ring with multiplicity induced by $[M][N] \coloneqq   [M\otimes N]$, where $M,N\in \mathcal C$. 
\end{itemize}

\begin{definition} Let $M$ and $N$ be two simple objects in $\mathcal C$.
\begin{itemize}
    \item[(i)] We say that $M$ and $N$ {\em commute}, if $M\otimes N\cong N\otimes M$.  
    \item[(ii)] We say that $M$ and $N$ {\em strongly commute}, if $M$ and $N$ commute and  $M\otimes N$ is simple.
    \item[(iii)] We say that $M$ is {\em real}, if $M\otimes M$ is simple. 
\end{itemize}
\end{definition}


\begin{definition} Let $\mathcal M=(M_1,\ldots,M_r)$ be an ordered set of simple objects in $\mathcal C$.
\begin{itemize}
    \item [(i)] We say that $\mathcal M$ a  {\em commuting set}, if $M_i$ and $M_j$ commute for any $i,j\in[1,r]$.
    \item[(ii)] We say that $\mathcal M$ is a {\em strongly commuting set}, if $M_i$ and $M_j$ strongly commute for any $i,j\in[1,r]$. In particular, each $M_i$ is a real simple object.
    \end{itemize}
\end{definition}

Given a commuting  set $\mathcal M=(M_1,\ldots,M_r)$ and a vector ${\bf a}=(a_1,\ldots,a_r)^T\in\mathbb Z_{\geq 0}^r$, 
denote 
$$\mathcal M({\bf a}) \coloneqq M_1^{\otimes a_1}\otimes M_2^{\otimes a_2}\otimes \ldots \otimes M_r^{\otimes a_r}.$$
Clearly, for any ${\bf a}, {\bf c}\in\mathbb Z_{\geq 0}^r$, we have
 $\mathcal{M}(\mathbf{a}) \otimes \mathcal{M}(\mathbf{c}) \cong \mathcal{M}(\mathbf{a }+ \mathbf{c})$.

\begin{definition}[Monoidal categorification]
  A finite length, $\mathbf{k}$-linear, monoidal abelian category $\mathcal C$ is a {\em monoidal categorification} of a cluster algebra $\mathcal A^+$ (with non invertible frozen variables), if there exists a $\mathbb Z$-algebra isomorphism $$\varphi\colon\mathcal K_0(\mathcal C)\rightarrow \mathcal A^+$$ such that the cluster monomials of $\mathcal A^+$ correspond to a subset of simple objects (up to isomorphism) in $\mathcal C$. We denote by $(\mathcal A^+, \mathcal C, \varphi)$ for this monoidal categorification.
\end{definition}

The notion of monoidal categorification of cluster algebras is originally introduced by Hernandez and Leclerc \cite{HL_2010}. 
Examples of cluster monoidal categorification include various monoidal subcategories of finite-dimensional modules over quantum affine algebras and quiver Hecke algebras \cite{HL16, kkko-2018, kkop-2024}, and the category of perverse coherent sheaves on the affine Grassmannian of $GL_n$ \cite{CW19}.

From now on, we fix a monoidal categorification $(\mathcal A^+, \mathcal C, \varphi)$ and denote \[\varphi(M) \coloneqq \varphi([M])\in\mathcal A^+\]
for any object $M\in\mathcal C$. Let $({\bf x}_t,\widetilde B_t)$ a seed of $\mathcal A^+$ and let $M_{i;t}\in\mathcal C$ be a simple object such that $\varphi(M_{i;t})=x_{i;t}$, which is unique up to isomorphism.  We call $\mathcal M_t \coloneqq (M_{1;t},\ldots,M_{m;t})$ a {\em monoidal cluster} and $(\mathcal M_t,\widetilde B_t)$ a {\em monoidal seed}. Following the terminology in the additive categorification, a simple object $M\in\mathcal C$ is said to be {\em reachable} if $\varphi(M)$ is a cluster monomial.

The following facts can be checked easily.
\begin{itemize}
    \item Each monoidal cluster $\mathcal M_t=(M_{1;t},\ldots,M_{m;t})$ is a strongly commuting set.
    \item Let ${\bf a}\in\mathbb Z_{\geq 0}^m$, then $\varphi(\mathcal M_t({\bf a}))={\bf x}_t^{\bf a}$. In particular, $\mathcal M_t({\bf a})$ is a reachable simple object.
    \item Each reachable simple object of $\mathcal C$ is real.
\end{itemize}

\begin{proposition}[{\cite[Proposition 2.2]{HL_2010}}]
\label{prop:HL_2010}
  Let $(\mathcal A^+, \mathcal C, \varphi)$ be a monoidal categorification and $M$ an object in $\mathcal C$. Then $\varphi(M)$ is universally positive in $\mathcal A^+$, that is, for any cluster ${\bf x}=(x_{1},\ldots,x_{m})$ of $\mathcal A^+$, we have $\varphi(M)\in\mathbb Z_{\geq 0}[x_{1}^{\pm 1},\ldots,x_{m}^{\pm 1}]$.
\end{proposition}

\subsection{$\mathsf{\Lambda}$-structure and $\mathsf{\Lambda}$-monoidal categorification}
In this subsection, we fix a
finite length, $\mathbf{k}$-linear, monoidal abelian category $\mathcal C$. For an object $M\in\mathcal C$, we denote by $\hd(M)$ the head  of  $M$. For any two objects $M,N$ in $\mathcal C$, we denote by
\[M\nabla N \coloneqq  \hd(M\otimes N)
\]
 the head of the tensor product $M\otimes N$.

The $\Lambda$-invariant and $\mathfrak{d}$-invariant are introduced by Kang, Kashiwara, Kim and Oh \cite{kkko-2018} in the representation theory of quiver Hecke algebras and by Kashiwara, Kim, Oh and Park \cite{kkop-2020} in the representation theory of quantum affine algebras in the study of monoidal categorification of cluster algebras. The definitions of $\Lambda$-invariant (and hence $\mathfrak{d}$-invariant) in both settings are constructive. For the purpose of this paper, we will rely only on the properties of the $\Lambda$-invariant, rather than on  the specific constructions. For this reason, we introduce the notion of $\mathsf{\Lambda}$-structure on a monoidal (abelian) category, which allows us to unify the definitions of $\Lambda$-invariant and $\mathfrak{d}$-invariant in different settings and unify the proofs for our results.

\begin{definition}[$\mathsf{\Lambda}$-structure and $\Lambda/\mathfrak{d}$-invariant] \label{def:L-structure}
A {\em $\mathsf{\Lambda}$-structure} on a monoidal category $\mathcal C$ is a pair $\mathsf{\Lambda}=(\Omega, \Lambda)$ such that (i) and (ii) hold.
\begin{itemize}
    \item [(i)] $\Omega$ is a class of non-zero objects of $\mathcal C$ containing all the simple objects and closed under
 isomorphism,  taking tensor products and non-zero subquotients.
    \item [(ii)]  ${\Lambda} \colon\Omega\times \Omega\rightarrow\mathbb Z$ is a map satisfying the following conditions.
\begin{itemize}
    \item[(a1)] The number $\Lambda(M,N)$ depends only on the isomorphism classes of $M$ and $N$.
    \item [(a2)] For $M, N\in\Omega$, we set 
    \[\mathfrak{d}(M,N)\coloneqq\frac{\Lambda(M,N)+\Lambda(N,M)}{2}.\]  If $M$ and $N$ are real simple objects and they strongly commute, then $\mathfrak{d}(M,N)=0$. 
\item[(a3)] Let $M,N, L\in\Omega$. If $L$ is  a simple object, then
\begin{align}
    \Lambda(M\otimes N,L)&=\Lambda(M,L)+\Lambda(N,L),\nonumber\\
  \Lambda(L,M\otimes N)&=\Lambda(L,M)+\Lambda(L,N).\nonumber  
\end{align}
\item[(a4)]  Let $M,N, L\in\Omega$ be simple objects.  If $L$ is real and it strongly commutes with $N$, then
\[
\Lambda(M\nabla N,L)=\Lambda(M,L)+\Lambda(N,L).
\]
Dually,
if $L$ is real and it strongly commutes with $M$, then
\[ \Lambda(L,M\nabla N)=\Lambda(L,M)+\Lambda(L,N).
\]
\item[(a5)] Let $M,N\in\Omega$. Then for any subquotient $0\neq M'$ of $M$ and subquotient $0\neq N'$ of $N$, we have 
$\Lambda(M',N')\leq \Lambda(M,N)$.
\end{itemize}
\end{itemize}
We call the numbers $\Lambda(M,N)$ and $\mathfrak{d}(M,N)$ the {\em $\Lambda$-invariant} and  {\em $\mathfrak{d}$-invariant} of the pair $(M,N)$. If $\mathsf{\Lambda}=(\Omega, \Lambda)$ is a $\mathsf{\Lambda}$-structure on a  monoidal category $\mathcal C$, we simply say that $\Lambda \colon\Omega\times \Omega\rightarrow\mathbb Z$ is a $\mathsf{\Lambda}$-structure on  $\mathcal C$.
\end{definition}
\begin{remark}\label{Rem:Rren}
The above axioms of $\mathsf{\Lambda}$-structure extract some of key properties of the $\Lambda$-invariant appearing in the representation theory of quiver
Hecke algebras \cite[Lem.\ 3.1.5, Lem.\ 3.2.3, Prop.\ 3.2.8, Lem.\ 3.2.13 ]{kkko-2018} and that of quantum affine algebras \cite[Prop.\ 3.9, Lem.\ 3.10, Cor.\ 3.17, Lem.\ 4.3]{kkop-2020}. 
More generally, these properties are consequences of the much more natural assumption that the monoidal abelian category $\mathcal C$ admits a system of renormalized $r$-matrices and has separable triple products in the sense of \cite[Section 4]{CW19}.

A system of renormalized $r$-matrices is an assignment of a non-zero homomorphism 
\[\mathbf{r}_{M,N} \colon M \otimes N \to N \otimes M\] together with an integer $\Lambda(M,N)$ for each ordered pair $(M,N)$ in $\Omega$ satisfying certain compatibility conditions between $\mathbf{r}_{M,N}$ and $\Lambda(M,N)$, which we can regard as a weaker version of the structure of a braided monoidal category. 
In general, the homomorphisms $\mathbf{r}_{M,N}$ need neither to be isomorphisms nor to be natural in $(M,N)$.
The $\Lambda$-invariant somehow controls how close these homomorphisms are to being natural.  
(Although \cite[Definition 4.1]{CW19} requires $\mathbf{r}_{M,N}$ and $\Lambda(M,N)$ to be defined for all $M,N \in \mathcal{C}$, there is no harm to weaken this requirement so that they are defined only for $M,N \in \Omega$ to obtain all the necessary consequences.) 
Having separable triple products is a weaker version of the rigidity in the monoidal category.

Many interesting non-commutative monoidal abelian categories admitting generic braidings satisfy these conditions.
Examples include the category of finite-dimensional modules over symmetric quiver Hecke algebras and the quantum affine algebras \cite{KKKO-2015, kkko-2018, kkop-2020}, the category of finite-length representations of general linear groups over $p$-adic fields \cite{LM18, Dros}, and the category of equivariant Koszul-perverse coherent sheaves on the BFN spaces of triples (including the equivariant perverse coherent sheaves on affine Grassmannians as special cases) categorifying the $K$-theoretic Coulomb branches \cite{CW19, CW23}. 
In all these examples, the renormalized $r$-matrices ${\bf r}_{M,N}$ are obtained from a canonical family of braiding operators $R_{M,N}(z)$, meromorphically depending on an additional continuous parameter $z$ by a suitable renormalization and a specialization as ${\bf r}_{M,N} = \lim_{z\to1} (z-1)^s R_{M,N}(z)$ for some $s \in \mathbb Z$.   
Then, roughly speaking, {\em the invariant $\Lambda(M,N)$ encodes the pole/zero order  $s$ of $R_{M,N}(z)$ at $z=1$.}
In all the known cases, the $\mathsf{\Lambda}$-structure arises in this way.
\end{remark}

\begin{lemma}\label{lem:L-g1}
    Let $\Lambda\colon\Omega\times \Omega\rightarrow\mathbb Z$ be a $\mathsf{\Lambda}$-structure on $\mathcal C$. Let $\mathcal M_w=(M_1,\ldots,M_m)$ be a strongly commuting set of real simple objects and set $\lambda_{ij} \coloneqq   \Lambda(M_i,M_j)$ for  $i,j\in[1,m]$. Then the following statements hold.
    \begin{itemize}
        \item [(i)] The matrix $\Lambda_w \coloneqq   (\lambda_{ij})$ is an $m\times m$ skew-symmetric integer matrix.
        \item[(ii)] For any ${\bf a},{\bf b}\in\mathbb Z_{\geq 0}^m$, we have
    \[ \Lambda(\mathcal M_w({\bf a}), \mathcal M_w({\bf b}))={\bf a}^T\Lambda_w{\bf b}.\]
    \end{itemize}
\end{lemma}
\begin{proof}
(i) Since $\mathcal M_w$ is a strongly commuting set of real simple objects and by the condition (a2) in the definition of ${\mathsf{\Lambda}}$-structure, we know that $\Lambda(M_i,M_j)+\Lambda(M_j,M_i)=0$ for $i,j\in[1,m]$. Thus the matrix $\Lambda_w$ is skew-symmetric.

(ii) Suppose that ${\bf b}={\bf b}_1+{\bf b}_2$ with ${\bf b}_1,{\bf b}_2\in\mathbb Z_{\geq 0}^m$. Then $\mathcal M_w({\bf b}_1)\otimes\mathcal M_w({\bf b}_2)\cong\mathcal M_w({\bf b})$. Since $\mathcal M_w({\bf a})$ is simple and by the condition (a3), we have
\[
\Lambda(\mathcal M_w({\bf a}),\mathcal M_w({\bf b}))=\Lambda(\mathcal M_w({\bf a}),\mathcal M_w({\bf b}_1))+\Lambda(\mathcal M_w({\bf a}),\mathcal M_w({\bf b}_2)).
\]
The similar reduction works for ${\bf a}\in\mathbb Z_{\geq 0}^m$ as well. 
Repeated reductions lead to the result.
\end{proof}

The following lemma plays a key role in the proof of the main theorem (Theorem \ref{thm:trop-Lambda}) in this section.
\begin{lemma}\label{lem:simpe-head}
 Let $\Lambda\colon\Omega\times \Omega\rightarrow\mathbb Z$ be a $\mathsf{\Lambda}$-structure on $\mathcal C$. Let $M,N,L\in\Omega$ be simple objects, and let $S_1,\ldots,S_l$ be a list of composition factors of $M\otimes N$.  Suppose that  the head $M\nabla N$ of $M\otimes N$ is simple.  Then the following statements hold.
 \begin{itemize}
     \item [(i)]
 If $L$ is real and it strongly commutes with $N$, then
\[
\Lambda(M,L)+\Lambda(N,L)=\Lambda(M \nabla N, L)=\max\{ \Lambda(S_i, L)\mid i\in[1,l]\}.
\]
\item[(ii)] If $L$ is real and it strongly commutes with $M$, then
\[
\Lambda(L,M)+\Lambda(L,N)=\Lambda(L, M \nabla N)=\max\{ \Lambda(L, S_i)\mid i\in[1,l]\}.
\]
\end{itemize}
\end{lemma}
\begin{proof}
 (i) Since the head $M\nabla N$ is simple, we know that $M\nabla N\cong S_k$ for some $k\in[1,l]$. As $L$ is real and it strongly commutes with $N$, the conditions (a3) and (a4) for the $\mathsf{\Lambda}$-structure $\Lambda\colon\Omega\times \Omega\rightarrow \mathbb Z$ imply
\[
   \Lambda(S_k, L)= \Lambda(M \nabla N, L)=\Lambda(M, L)+\Lambda(N,L)=\Lambda(M \otimes N, L).
   \]
On the other hand, the condition (a5) implies
\[\Lambda(S_i, L) \leq \Lambda(M\otimes N,L ) =\Lambda(M, L)+\Lambda(N,L)
\]
for any $i\in[1,l]$.
Therefore, we have 
\[
\Lambda(M,L)+\Lambda(N,L)=\Lambda(M \nabla N, L)=\max\{ \Lambda(S_i, L)\mid i\in[1,l]\}.
\]

(ii) The proof is similar to (i).
\end{proof}
We can see that the simple head condition in the above lemma plays an important role in the proof. In order to maintain this condition, we make the following assumption on the monoidal category $\mathcal C$, which is  automatically satisfied once $\mathcal{C}$ admits a system of renormalized $r$-matrices and has the separable triple products in the sense of \cite[Section 4]{CW19}. In particular, this is true for all  known examples of $\mathsf{\Lambda}$-structures.

\begin{assumption}\label{assumption}
Let $M, N \in \mathcal C$ be simple objects such that at least one of them is real. Then the head $M \nabla N$ of $M\otimes N$ is simple.
\end{assumption}

\begin{definition}[$\mathsf{\Lambda}$-monoidal categorification] \label{def:L-monoidal}
Let $(\mathcal A^+,\mathcal C,\varphi)$ be a monoidal categorification and  $\Lambda\colon\Omega\times \Omega\rightarrow\mathbb Z$ a $\mathsf{\Lambda}$-structure on $\mathcal C$. We say that  $(\mathcal A^+,\mathcal C,\varphi, \Omega, \Lambda)$ is a {\em ${\mathsf{\Lambda}}$-monoidal categorification} if the following conditions hold. 
\begin{itemize}
    \item [(c1)] The monoidal category $\mathcal C$ satisfies Assumption \ref{assumption}. 
    \item[(c2)] The $\mathsf{\Lambda}$-structure and cluster structure on $\mathcal C$ are compatible in the sense that there exists a vertex $w \in \mathbb{T}_n$ such that $(\widetilde B_w, \Lambda_w)$ forms a compatible pair,  \ie $\widetilde B_w^T\Lambda_w=(S\mid {\bf 0})$ for some diagonal matrix $S=diag(s_1,\ldots,s_n)$ with $s_j\in\mathbb Z_{\geq 1}$. Here  $\Lambda_w$ is the $m\times m$ skew-symmetric matrix given as in Lemma \ref{lem:L-g1}.
\end{itemize}
\end{definition}
The teminology of $\mathsf{\Lambda}$-monoidal categorification is borrowed from \cite[Definition 6.7]{kkop-2020}. 
\begin{remark} Thanks to Proposition \ref{prop:B-L-t} below, for a $\mathsf{\Lambda}$-monoidal categorification, we know that 
    $(\widetilde B_t)^T\Lambda_t=(S\mid {\bf 0})$ holds for any vertex $t\in\mathbb T_n$.
\end{remark}
\begin{remark}
  The examples of $\mathsf{\Lambda}$-monoidal categorifications contain those monoidal categorifications constructed by  Kang,  Kashiwara, Kim, and  Oh  \cite{kkko-2018} using monoidal subcategories of representations of (symmetric) quiver Hecke algebras
    and  by Kashiwara, Kim, Oh, and  Park \cite{kkop-2024} using monoidal subcategories of representations of quantum affine algebras. Notice  that in \cite{kkko-2018,kkop-2024}, the authors always take $S=2I_n$ in condition (c2) of Definition \ref{def:L-monoidal}.
\end{remark}

\begin{proposition}[cf.\ {\cite[Proposition 7.1.2(e)]{kkko-2018}, \cite[Proposition 6.4]{kkop-2020}}] \label{prop:B-L-t}
Let $(\mathcal A^+,\mathcal C,\varphi, \Omega, \Lambda)$ be a $\mathsf{\Lambda}$-monoidal categorification.
Then the following statements hold.
\begin{enumerate}
\item[(i)] For any vertex $t\in\mathbb T_n$, $(\widetilde B_t, \Lambda_t)$ forms a compatible pair.
\item[(ii)] For any edge \begin{xy}(0,1)*+{t}="A",(10,1)*+{t'}="B",\ar@{-}^k"A";"B" \end{xy} in $\mathbb T_n$, $(\widetilde B_{t'}, \Lambda_{t'})=\mu_k(\widetilde B_{t}, \Lambda_{t})$ as mutation of compatible pairs.
\end{enumerate}
In particular, $\mathcal{A}^+$ is a $\mathsf{\Lambda}$-cluster algebra.
\end{proposition}
\begin{proof}
For the convenience of the reader, we display the proof here. Since $(\mathcal A^+,\mathcal C,\varphi, \Omega, \Lambda)$ is a $\mathsf{\Lambda}$-monoidal categorification, there exists a vertex $w \in \mathbb{T}_n$ such that $(\widetilde B_w, \Lambda_w)$ forms a compatible pair. Denote by $(\widetilde B',\Lambda')=\mu_k(\widetilde B_w, \Lambda_w)$ the compatible pair obtained from $(\widetilde B_w, \Lambda_w)$ by mutation in direction $k\in[1,n]$, where $\Lambda'$ is given by the equation \eqref{eqn:L-mutation}.

Let $w'=\mu_k(w)$ be the vertex in $\mathbb T_n$ such that $w$ and $w'$ are joined by the edge labeled by $k\in[1,n]$. We know that $\widetilde B'=\widetilde B_{w'}$.  Denote by $\Lambda_t=(\lambda_{ij;t})$ for $t\in\mathbb T_n$ and $\Lambda'=(\lambda_{ij}')$. All we need to do is to prove $\Lambda_{w'}=\Lambda'$, \ie $\lambda_{ij;w}=\lambda_{ij}'$ for any $i,j\in[1,m]$.

We first consider the case $i\neq k$ and $j\neq k$. In this case,
\[\lambda_{ij}'=\lambda_{ij;w}=\Lambda(M_{i;w},M_{j;w})=\Lambda(M_{i;w'},M_{j;w'})=\lambda_{ij;w'}.\]
Now we consider the case $i\neq k$ and $j=k$. 
 Denote by ${\bf b}_k=(b_{1k},\ldots,b_{mk})^T\in\mathbb Z^m$ the $k$th column of $\widetilde B_w$ and by
\[[{\bf b}_k]_+=([b_{1k}]_+,\ldots,[b_{mk}]_+)^T,\quad [-{\bf b}_k]_+=([-b_{1k}]_+,\ldots,[-b_{mk}]_+)^T.
\]
Since $(\widetilde B_w,\Lambda_w)$ is a compatible pair and $i\neq k$,  the $(i,k)$-entry of $\Lambda_w\widetilde B_w$ is zero and thus
\[\Lambda(M_{i;w},\mathcal M_w([-{\bf b}_k]_+))=\sum_{l=1}^m[-b_{lk}]_+\lambda_{il;w}=\sum_{l=1}^m[b_{lk}]_+\lambda_{il;w}=\Lambda(M_{i;w},\mathcal M_w([{\bf b}_k]_+)).
\]
In the Grothendieck ring $\mathcal K_0(\mathcal C)$, we have the relation:
\[ [M_{k;w'}][M_{k;w}]=[M_{k;w}\otimes M_{k;w'}]=[\mathcal M_w([{\bf b}_k]_+)]+[\mathcal M_w([-{\bf b}_k]_+)].
\]
So either $M_{k;w}\nabla M_{k;w'}\cong \mathcal M_w([-{\bf b}_k]_+)$ or $M_{k;w}\nabla M_{k;w'}\cong \mathcal M_w([{\bf b}_k]_+)$. In either case, 
 we  have \[
\Lambda(M_{i;w},M_{k;w}\nabla M_{k;w'})=\Lambda(M_{i;w},\mathcal M_w([-{\bf b}_k]_+))=\Lambda(M_{i;w},\mathcal M_w([{\bf b}_k]_+)).
\]
Since $M_{i;w}$ strongly commutes with $M_{k;w}$ and by the condition (a4) in Definition \ref{def:L-structure}, we have
\begin{eqnarray}\label{eqn:L-ww}
  \Lambda(M_{i;w},\mathcal M_w([-{\bf b}_k]_+))=  \Lambda(M_{i;w}, M_{k;w}\nabla M_{k;w'})=
\Lambda(M_{i;w}, M_{k;w})+\Lambda(M_{i;w},  M_{k;w'}).\nonumber
\end{eqnarray}
Thus  
\begin{align*}
   \lambda_{ik;w'}= \Lambda(M_{i;w},  M_{k;w'})&=-\Lambda(M_{i;w}, M_{k;w})+\Lambda(M_{i;w},\mathcal M_w([-{\bf b}_k]_+))\\
    &=-\lambda_{ik;w}+\sum_{l=1}^m[-b_{lk}]_+\lambda_{il;w}.
\end{align*}
By comparing with equation \eqref{eqn:L-mutation}, we see that the $(i,k)$-entries of $\Lambda'$ and $\Lambda_{w'}$ are the same. Since both $\Lambda'$ and $\Lambda_{w'}$ is skew-symmetric, we know that the $(k,i)$-entries of  $\Lambda'$ and $\Lambda_{w'}$ are also the same. Thus $\Lambda'=\Lambda_{w'}$.  So $(\widetilde B_{w'}, \Lambda_{w'})=\mu_k(\widetilde B_w, \Lambda_w)$ as mutation of compatible pairs.
\end{proof}
\subsection{Comparison with ${\Lambda}$-invariant and $\mathfrak{d}$-invariant}
In this subsection, we compare the tropical invariant and $F$-invariant with the ${\Lambda}$-invariant and $\mathfrak{d}$-invariant.

\begin{lemma}\label{cor:head}
Let  $(\mathcal A^+,\mathcal C,\varphi, \Omega, \Lambda)$ be a $\mathsf{\Lambda}$-monoidal categorification and  $M$ a simple object in $\mathcal C$. Then for any monoidal cluster $\mathcal M_w$,  there exist ${\bf d,h,h'} \in \mathbb Z_{\ge 0}^m$ satisfying
\[
\mathcal M_w({\bf d})\nabla M\cong\mathcal M_w({\bf h}) \quad \text{and} \quad 
M \nabla \mathcal M_w({\bf d})\cong\mathcal M_w({\bf h'}). 
 \]
\end{lemma}
\begin{proof}
For the vertex $w\in\mathbb T_n$, by the Laurent phenomenon, there exist a vector ${\bf d}\in\mathbb Z_{\geq 0}^m$ and a finite subset $\mathbb B\subseteq\mathbb Z_{\geq 0}^m$ such that
\begin{equation} \label{eq:pos_exp}
 \varphi(M)\cdot \varphi(\mathcal M_w({\bf d})) = \varphi(M\otimes \mathcal M_w({\bf d}))= \varphi(\mathcal M_w({\bf d}) \otimes M) = \sum_{{\bf b}\in\mathbb B}c_{{\bf b}}\varphi(\mathcal M_w({\bf b})),
\end{equation}
 where $\{\mathcal M_w({\bf b})\mid {\bf b}\in\mathbb B\}$ is the set of composition factors of $M\otimes \mathcal M_w({\bf d})$ (and also those of $\mathcal M_w({\bf d}) \otimes M$) and $c_{\bf b}\in\mathbb Z_{\geq 1}$ is the  multiplicity of $\mathcal M_w({\bf b})$.
Then, by the condition (c1) in Definition \ref{def:L-monoidal}, we know that the heads 
$\mathcal M_w({\bf d}) \nabla M$ and $M \nabla \mathcal M_w({\bf d})$ are simple. Hence, 
$\mathcal M_w({\bf d}) \nabla M\cong \mathcal M_w({\bf h})$ and $M \nabla \mathcal M_w({\bf d})\cong\mathcal M_w({\bf h'}) $ for some ${\bf h, h'}\in\mathbb B$.
\end{proof}

From now on, we fix a $\mathsf{\Lambda}$-monoidal categorification $(\mathcal A^+,\mathcal C,\varphi, \Omega, \Lambda)$. We know that $\mathcal A^+$ is a $\mathsf{\Lambda}$-cluster algebra.

{\em Notations:} Denote by $(\mathcal S_X,\mathcal S_Y,\mathsf{\Lambda})$ the Langlands-Poisson triple corresponding to the $\mathsf{\Lambda}$-cluster algebra $\mathcal A^+$. Let $M\in\mathcal C$ be an object such that $\varphi(M)$ is a good element in $\mathcal A^+$. For any vertex $t\in\mathbb T_n$, we denote by 
\[{\bf g}_M^t \coloneqq   {\bf g}_{\varphi(M)}^t\in\mathbb Z^m \;\;\text{and}\;\;F_M^t \coloneqq   F_{\varphi(M)}^t\in\mathbb Z[y_1,\ldots,y_n]
\]
the extended $g$-vector and  $F$-polynomial of $\varphi(M)$ with respect to vertex $t$ and  denote by 
\[
\beta_M \coloneqq   \beta_{\varphi(M)}\colon \mathbb F_{>0}=\mathbb Q_{\rm sf}(x_{1;t_0},\ldots,x_{m;t_0})\rightarrow\mathbb Z^{\rm max}
\]
the semifield homomorphism given by the good element $\varphi(M)$ as in Proposition \ref{pro:beta-map}.

Let $N$ be a reachable simple object, say $N\cong \mathcal M_w({\bf a})$. We know that $\varphi(N)={\bf x}_w^{\bf a}$ is a cluster monomial in $\mathcal A^+$ and thus it is a good element. We denote by $$[{\bf g}_{N, \pm}]=\{{\bf g}_{N, \pm}^t\in\mathbb Z^m\mid t\in\mathbb T_n\}$$ the unique tropical point  in $\mathcal S_Y(\mathbb Z^{\rm max})$ determined by the condition ${\bf g}_{N, \pm}^w=\pm{\bf a}$. By Proposition \ref{pro:a-point}, we know that 
\[\{\Lambda_t{\bf g}_{N, \pm}^t\in\mathbb Z^m\mid t\in\mathbb T_n\}\]
are tropical points  in $\mathcal S_X(\mathbb Z^{\rm max})$.  Then by Corollary \ref{cor:bijection}, there exist unique semifield homomorphisms 
$$\beta_{N,\pm}\colon\mathbb F_{>0}=\mathbb Q_{\rm sf}(x_{1;t_0},\ldots,x_{m;t_0})\rightarrow\mathbb Z^{\rm max}$$ such that 
\begin{eqnarray}\label{eqn:beta-N}
      \beta_{N, \pm}({\bf x}_t)=(\Lambda_t{\bf g}_{N, \pm}^t)^T
  \end{eqnarray}
  where $t$ is any vertex in $\mathbb T_n$. In particular, $\beta_{N, \pm}({\bf x}_w)= \pm(\Lambda_w{\bf a})^T$.
\begin{remark}
     Since ${\bf g}_{N,+}^w={\bf a}$ is the extended $g$-vector of the cluster monomial $\varphi(N)={\bf x}_w^{\bf a}$ with respect to vertex $w$,  we know that 
${\bf g}_{N,+}^t={\bf g}_N^t$ for any vertex $t\in\mathbb T_n$ and thus $$\beta_{N,+} = \beta_N=\beta_{\varphi(N)}.$$
\end{remark}

For each object $M\in \mathcal C$, we know that $\varphi(M)$ is an element in $\mathbb F_{>0}=\mathbb Q_{\rm sf}(x_{1;t_0},\ldots,x_{m;t_0})$, by Proposition \ref{prop:HL_2010}. Thus $\beta_{N, \pm}(\varphi(M))$ can be defined. For simplicity, we denote by $$\beta_{N, \pm}(M)\coloneqq\beta_{N, \pm}(\varphi(M))$$ for any object $M\in\mathcal C$.  Hence, if $\varphi(M)$ is a good element, we have
\begin{equation} \label{eq:beta+=trop}
\beta_{N,+}(M) =\beta_{\varphi(N)}(\varphi(M))= \langle \varphi(M), \varphi(N) \rangle_{\rm trop}. 
\end{equation}

\begin{theorem} \label{thm:trop-Lambda}
Let  $(\mathcal A^+,\mathcal C,\varphi, \Omega, \Lambda)$ be a $\mathsf{\Lambda}$-monoidal categorification. Then
for any reachable simple object $N= \mathcal M_w({\bf a})$, the following statements hold.
 \begin{itemize}
    \item [(i)]  If $M$ is a simple object of $\mathcal{C}$, we have
    \[ \Lambda(M,N) = \beta_{N,+}(M), \quad \Lambda(N,M) = \beta_{N,-}(M). \]
    \item [(ii)] If $M$ is a simple object of $\mathcal{C}$ such that $\varphi(M)$ is a good element, we have
    \[\Lambda(M,N)=\langle \varphi(M), \varphi(N)\rangle_{\rm trop}, \quad 
    \Lambda(N,M) = \langle \varphi(N), \varphi(M)\rangle_{\rm trop}\] 
    and hence
    \[ 2\mathfrak{d}(M,N) = (\varphi(M)\mid\mid \varphi(N))_F. \]
\end{itemize}  
\end{theorem}
\begin{proof}
(i)  For the vertex $w\in\mathbb T_n$, as in the proof of Lemma \ref{cor:head}, there are a vector ${\bf d}\in\mathbb Z_{\geq 0}^m$ and a finite subset $\mathbb B\subseteq\mathbb Z_{\geq 0}^m$ such that the equation \eqref{eq:pos_exp} holds. 
Then $\{\mathcal M_w({\bf b})\mid {\bf b}\in\mathbb B\}$ is the set of composition factors of $\mathcal M_w({\bf d}) \otimes M$.

On the one hand, the head $\mathcal M_w({\bf d}) \nabla M$ is simple, by the condition (c1) in Definition \ref{def:L-monoidal}. On the other hand, we know that  $N = \mathcal M_w({\bf a})$ is real and it strongly commutes with $\mathcal M_w({\bf d})$.  Thus Lemma \ref{lem:simpe-head} (ii) can be applied and we have
\begin{eqnarray}
   \Lambda(N, \mathcal M_w({\bf d}))+\Lambda(N,M)=\Lambda(N, \mathcal M_w({\bf d}) \nabla M)=\max\{ \Lambda(N, \mathcal M_w({\bf b}))\mid {\bf b}\in\mathbb B\}.\nonumber
\end{eqnarray}
Thus
\begin{eqnarray}\label{eqn:Lambda-max}
\Lambda(N,M)&=&\max\{ \Lambda(N,\mathcal M_w({\bf b}))\mid {\bf b}\in\mathbb B\}- \Lambda(N, \mathcal M_w({\bf d}))\nonumber\\
&\xlongequal{N= \mathcal M_w({\bf a})}&\max\{ \Lambda(\mathcal M_w({\bf a}), \mathcal M_w({\bf b}))\mid {\bf b}\in\mathbb B\}- \Lambda(\mathcal M_w({\bf a}),\mathcal M_w({\bf d}))\label{eq:Lmax}\\
&\xlongequal{\text{Lem.} \ref{lem:L-g1}}&\max\{{\bf a}^T\Lambda_w{\bf b}\mid {\bf b}\in\mathbb B\}-{\bf a}^T\Lambda_w{\bf d}\nonumber.
\end{eqnarray}

 Now we calculate $\beta_{N,-}(M)$.
By applying the semifield homomorphism $\beta_{N,-}\colon\mathbb F_{>0}\rightarrow\mathbb Z^{\rm max}$ to \eqref{eq:pos_exp},  we 
have $
\beta_{N,-}(M)+\beta_{N,-}(\mathcal M_w({\bf d}))=\max\{ \beta_{N,-}(\mathcal M_w({\bf b}))\mid {\bf b}\in\mathbb B\}$,
 \ie $$\beta_{N,-}(M)+\beta_{N,-}({\bf x}_w^{\bf d})=\max\{ \beta_{N,-}({\bf x}_w^{\bf b})\mid {\bf b}\in\mathbb B\}.$$ 
Since $N= \mathcal M_w({\bf a})$, we have $\beta_{N,-}({\bf x}_w)=(\Lambda_w{\bf g}_{N,-}^w)^T=-(\Lambda_w{\bf a})^T = {\bf a}^T \Lambda_w$ and hence
\begin{align}
  \beta_{N,-}(M)&= \max\{ \beta_{N,-}({\bf x}_w^{\bf b})\mid {\bf b}\in\mathbb B\}-\beta_{N,-}({\bf x}_w^{\bf d})\nonumber\\
  &=\max\{{\bf a}^T\Lambda_w{\bf b}\mid {\bf b}\in\mathbb B \}-{\bf a}^T\Lambda_w{\bf d}.\nonumber
  \end{align}
Thus we obtain $\Lambda(N,M) = \beta_{N,-}(M)$. 
The proof of $\Lambda(M,N) = \beta_{N,+}(M)$ is similar.

(ii) The equality $\Lambda(M,N)=\langle \varphi(M), \varphi(N)\rangle_{\rm trop}$ is immediate from the first equality in (i) and \eqref{eq:beta+=trop}. 
Let us prove the other one,  \ie $\Lambda(N,M) = \langle \varphi(N), \varphi(M)\rangle_{\rm trop}$. As $\varphi(M)$ is a good element and by \eqref{eq:pos_exp}, we know that the finite set $\mathbb B\subseteq\mathbb Z_{\geq 0}^m$   contains ${\bf b}_0 \coloneqq {\bf g}_M^w + {\bf d}$ satisfying ${\bf b} - {\bf b}_0 \in  \widetilde{B}_w (\mathbb{Z}_{\ge 0}^{n}) $ for all ${\bf b} \in \mathbb{B}$.
Since 
\[
{\bf a}^T \Lambda_w ({\bf b}-  {\bf b}_0) \in {\bf a}^T\Lambda_w\widetilde{B}_w (\mathbb{Z}_{\ge 0}^{n}) = -{\bf a}^T(S \mid {\bf 0})^T(\mathbb{Z}_{\ge 0}^n) \subseteq \mathbb{Z}_{\le 0}
\]
for any ${\bf b} \in \mathbb{B}$,
we deduce
\[ \max\{{\bf a}^T\Lambda_w{\bf b}\mid {\bf b}\in\mathbb B\} = {\bf a}^T\Lambda_w{\bf b}_0 = {\bf a}^T\Lambda_w({\bf g}_M^w + {\bf d}).\]
Combined with \eqref{eq:Lmax}, it yields
\begin{equation} \label{eq:aLM}
\Lambda(N,M) = {\bf a}^T\Lambda_w({\bf g}_M^w + {\bf d}) - {\bf a}^T\Lambda_w {\bf d} = {\bf a}^T \Lambda_w {\bf g}_M^w.
\end{equation}
On the other hand, as $N=\mathcal M_w({\bf a})$, we have ${\bf g}_N^w={\bf a}$, $F_{N}^w=F_{{\bf x}_w^{\bf a}}^w =1$  and thus
\[\langle \varphi(N),\varphi(M)\rangle_{\rm trop}= ({\bf g}_N^w)^T\Lambda_w{\bf g}_{M}^w+F_{N}^w[(S\mid {\bf 0}){\bf g}_{M}^w]={\bf a}^T\Lambda_w{\bf g}_{M}^w.
\]
Thus, we get the desired equality $\Lambda(N,M) = \langle \varphi(N), \varphi(M)\rangle_{\rm trop}$.
\end{proof}

Since the reachable simple objects correspond to cluster monomials, which are always good elements in $\mathcal A^+$, we have the following consequence.
\begin{corollary}
    Let  $(\mathcal A^+,\mathcal C,\varphi, \Omega, \Lambda)$ be a $\mathsf{\Lambda}$-monoidal categorification. If $M$ and $N$ are reachable simple objects, then we have 
    \[\Lambda(M,N)=\langle \varphi(M), \varphi(N)\rangle_{\rm trop}\;\;\;\text{and}\;\;\;2\mathfrak{d}(M,N) = (\varphi(M)\mid\mid \varphi(N))_F.\] 
\end{corollary}
\section{Comparison with $E$-invariant in additive cluster categorification}\label{sec:E-invariant}

\subsection{$E$-invariant in the theory of quivers with potentials}
Throughout, let $Q=(Q_0,Q_1)$ be a quiver with vertex set $Q_0=\{1,\ldots,n\}$. We will assume that $Q$ has no cycles of length $\leq 2$. Such a quiver $Q$ corresponds to a skew-symmetric matrix $B_Q=(b_{ij})$, where
\begin{eqnarray}\label{eqn:Q-B-2}
   b_{ij}=\#\{j\rightarrow i \text{ in }Q_1\}-\#\{i\rightarrow j \text{ in }Q_1\}. 
\end{eqnarray}
For example, if  $Q=1\leftarrow 2$, then $B_Q=\begin{bmatrix}
    0&1\\
    -1&0
\end{bmatrix}$.

 For $m\in\mathbb Z_{\geq 0}$, let $\mathbb CQ_m$ be a $\mathbb C$-vector space with a $\mathbb C$-basis labeled by the set $Q_m$ of paths
of length $m$ in $Q$. Clearly, $\mathbb CQ_m$  is finite-dimensional.

The {\em completed path algebra} of $Q$ is denoted by $\widehat{\mathbb CQ} \coloneqq   \prod_{m\geq 0}\mathbb CQ_m$. A {\em potential} $W$ of $Q$ is an element $W=\sum_{m\geq 1}w_m$ of  $\widehat{\mathbb CQ}$, where $w_m$ is a $\mathbb C$-linear combination of $m$-cycles of $Q$.

Let $(Q,W)$ be  a quiver with potential, and  $J=J(Q,W)$ the associated {\em Jacobian algebra} \cite[Section 3]{DWZ08}, which is a quotient of $\widehat{\mathbb CQ}$ by an  ideal $I_W$ (known as {\em Jacobian ideal}) determined by $W$. The quiver with potential $(Q,W)$ is said to be {\em Jacobi-finite}, if the corresponding Jacobian algebra $J=J(Q,W)=\widehat{\mathbb CQ}/I_W$ is finite-dimensional.

\begin{example}
    Consider the quiver $Q$ as follows:
    $$\xymatrix{&1\ar[rd]^{\beta_3}&\\
3\ar[ru]^{\beta_2}&&2\ar[ll]_{\beta_1}}$$
Then  $Q_2=\{\beta_2\beta_1,\;\beta_3\beta_2,\;\beta_1\beta_3\}$.
Notice that  we read paths from right to left. If we take $W=\beta_3\beta_2\beta_1$, the corresponding Jacobian algebra $J=J(Q,W)$ is given by the quiver $Q$ with relations $\beta_2\beta_1=0,\;\beta_3\beta_2=0,\;\beta_1\beta_3=0$. In this case, $(Q,W)$ is Jacobi-finite.
\end{example}

A (finite-dimensional) {\em representation} of $Q=(Q_0,Q_1)$ is a tuple $M=(M_i, M_\alpha)_{i\in Q_0,\alpha\in Q_1}$, where each $M_i$ is a finite-dimensional $\mathbb C$-vector space and $M_\alpha\colon M_i\rightarrow M_j$ is a linear map for each arrow $\alpha\colon i\rightarrow j$. Denote by $\underline \dim M=(\dim M_1,\ldots,\dim M_n)^T\in\mathbb N^n$ the {\em dimension vector} of $M$. Notice that in this paper, we always view  vectors as  {\em column vectors}.

A {\em representation} of $(Q,W)$ is a  representation of $Q$ which is annihilated by the Jacobian ideal $I_W$. A {\em decorated representation}
of $(Q,W)$ is a pair $\mathcal M=(M,V)$, where $M$ is a representation of $(Q,W)$ and $V=(V_i)_{i\in Q_0}=(V_1,\ldots,V_n)$ is a tuple of finite-dimensional $\mathbb C$-vector spaces.  We identify the representations of $(Q,W)$ with the finite-dimensional {\em left modules} of $J=J(Q,W)$.

A decorated representation $\mathcal M=(M,V)$ of $(Q,W)$ is said to be {\em negative}, if $M=0$; and it is said to be {\em negative simple}, if $M=0$ and there exists $i\in Q_0$ such that $V_i=\mathbb C$ and $V_j=0$ for any $j\neq i$.

The $g$-vectors, $F$-polynomials and $E$-invariants of decorated representations are introduced by Derksen, Weyman and Zelevinsky \cite{DWZ10}, which are reformulated nicely in \cite{CLS-2015}. Now let us recall these notions following \cite[Section 3.1]{CLS-2015}. 

\begin{definition}[$g$-vector, $F$-polynomial and $E$-invariant] \label{def:g-F-E}
    Let $\mathcal M=(M,V)$ and  $\mathcal M'=(M',V')$ be two decorated representations of $J=J(Q,W)$. 
    \begin{itemize}
        \item [(i)] The  {\em $g$-vector} of $\mathcal M$ is a vector ${\bf g}({\mathcal M})\in\mathbb Z^n$ whose $i$th component is defined by  
$$g_i(\mathcal M) \coloneqq   -\dim \Hom_J(S_i,M)+\dim \Ext_J^1(S_i,M)+\dim V_i.$$

\item [(ii)] The {\em $F$-polynomial} of $\mathcal M=(M,V)$ is defined as follows:
\[ F_{\mathcal M}(y_1,\ldots,y_n) \coloneqq   \sum_{{\bf v}\in\mathbb N^n}\chi(\Gr_{\bf v}(M)){\bf y}^{\bf v}\;\;\in\mathbb Z[y_1,\ldots,y_n],
\]
where $\Gr_{\bf v}(M)$ is the submodule Grassmannian of $M$ with dimension vector ${\bf v}$ and $\chi$ is the Euler–Poincar\'{e} characteristic.

\item[(iii)]  Let $\langle-,-\rangle$ be the standard inner product on $\mathbb R^n$. We define two integers:
\begin{eqnarray}\label{eqn:E-inj}
    E^{\rm inj}(\mathcal M,\mathcal M') &\coloneqq&   \dim \Hom_J(M,M')+\langle \underline{\dim} M ,{\bf g}(\mathcal M')\rangle,\\
    E^{\rm sym}(\mathcal M,\mathcal M') &\coloneqq&    E^{\rm inj}(\mathcal M,\mathcal M')+E^{\rm inj}(\mathcal M',\mathcal M),\label{eqn:E-sym}
\end{eqnarray}
which are called  the {\em partial $E$-invariant} and  {\em $E$-invariant} of the pair $(\mathcal M,\mathcal M')$.
    \end{itemize}
\end{definition}

\begin{remark}\label{rmk:E-inv}
 It is easy to see that the $F$-polynomial $F_{\mathcal M}\in\mathbb Z[y_1,\ldots,y_n]$ has a unique term of maximal degree ${\bf y}^{\underline{\dim }M}$ with coefficient $1$. Every other monomial in $F_{\mathcal M}$ is a factor of ${\bf y}^{\underline{\dim }M}$.
\end{remark}

\begin{corollary}\label{cor:E-inv}
 Let $\mathcal M=(M,V)$ and  $\mathcal M'=(M',V')$ be two decorated representations of $J=J(Q,W)$. If 
 $\mathcal M=(M,V)$ is negative,  \ie $M=0$ and  $V=(\mathbb C^{a_1},\ldots,\mathbb C^{a_n})$ for some ${\bf a}=(a_1,\ldots,a_n)^T\in\mathbb Z_{\geq 0}^n$, then we have
 \[
E^{\rm inj}(\mathcal M,\mathcal M')=0,\;\;
 E^{\rm inj}(\mathcal M',\mathcal M)=\langle \underline{\dim} M', {\bf a}\rangle=E^{\rm sym}(\mathcal M,\mathcal M').
 \]
\end{corollary}
\begin{proof}
By definition, we have ${\bf g}(\mathcal M)={\bf a}$. Then the results from \eqref{eqn:E-inj} and \eqref{eqn:E-sym}.
\end{proof}

\begin{proposition}
[{\cite[Cor.\ 10.9]{DWZ10}, \cite[Lem.\ 3.4]{CLS-2015}}]
Let $\mathcal M=(M, V)$ and $\mathcal M'=(M', V')$ be two decorated representations of $J=J(Q,W)$. Let $d_i(V)=\dim V_i$.  Suppose that  $(Q,W)$ is Jacobi-finite. Then the following statements hold.

\begin{itemize}
\item[(i)] We have  \[E^{\rm inj}(\mathcal M,\mathcal M')=\dim\Hom_J(\tau^{-1}M', M)+
\langle \underline{\dim} M, \underline{\dim} V'\rangle
.\]  In particular, if $V'=0=V$, we have 
\[
E^{\rm inj}(\mathcal M,\mathcal M')=\dim\Hom_J(\tau^{-1}M', M),\;\; E^{\rm inj}(\mathcal M',\mathcal M)=\dim\Hom_J(\tau^{-1}M, M')
\]
and hence \[E^{\rm sym}(\mathcal M,\mathcal M')=\dim\Hom_J(\tau^{-1}M', M)+\dim\Hom_J(\tau^{-1}M, M').\]

  \item[(ii)] Let \[0\rightarrow M\rightarrow \oplus_{i=1}^nI_i^{a_i}\rightarrow \oplus_{i=1}^nI_i^{b_i}\] be  a minimal injective presentation of $M$. Then the $i$th component  of the $g$-vector ${\bf g}(\mathcal M)$ of $\mathcal M$ is given by
  \[g_i(\mathcal M)=-a_i+b_i+d_i(V).
  \]  
\end{itemize}
\end{proposition}

\begin{remark}
We remark that the above results can be extended to the Jacobian-infinite  setting by considering the truncated Jacobian algebras, which are finite-dimensional,  \confer \cite[{Prop.\ 3.5, Lem.\ 3.4}]{CLS-2015}.  In particular, for any two decorated representations $\mathcal M$ and $\mathcal M'$ of $(Q,W)$, we have 
\[
E^{\rm inj}(\mathcal M,\mathcal M')\geq 0 \;\;\;\text{and}\;\;\;\;E^{\rm sym}(\mathcal M,\mathcal M')\geq 0.
\]
\end{remark}

\subsection{Mutation-invariance of $E$-invariant} 
We fix a quiver with potential $(Q,W)$ and assign it to the rooted vertex $t_0$ of the $n$-regular tree $\mathbb T_n$. We will assume that $W$ is a {\em non-degenerate potential} \cite[Section 7]{DWZ08} so that we can obtain a family of quivers with potentials 
\[ [Q,W] \coloneqq   \{(Q_t,W_t)\mid t\in\mathbb T_n\} \;\;\;\text{with}\;\;\; (Q_{t_0},W_{t_0})=(Q,W)\]
by  {\em mutations of quivers with potentials} \cite[Section 5]{DWZ08}. Denote by $J_t=J(Q_t,W_t)$ the Jacobian algebra of $(Q_t,W_t)$.

 In \cite[Section 10]{DWZ08} Derkesn, Weyman and Zelevinsky introduced {\em mutations of decorated representations}, which are generalizations of BGP-reflections.
 Let  $\mathcal M=(M,V)$ be a decorated representation of $(Q,W)=(Q_{t_0},W_{t_0})$. We can produce a family of decorated representations 
 \[[\mathcal M] \coloneqq   \{\mathcal M^t=(M^t,V^t)\mid t\in\mathbb T_n\}\;\;\;\text{with}\;\;\;\; \mathcal M^{t_0}=\mathcal M
 \]
  by mutations of decorated representations, where $\mathcal M^t$ is a decorated representation of the quiver with potential $(Q_t,W_t)$ at vertex $t\in\mathbb T_n$. The decorated representation $\mathcal M$ of $(Q,W)$ is said to be {\em negative-reachable}, if there exists a vertex $w\in \mathbb T_n$ such that $\mathcal M^w=(M^w,V^w)$ is  negative,  \ie $M^w=0$. 

\begin{remark}
The negative-reachable decorated representations of $(Q,W)$ play the role of cluster monomials in the additive categorification of cluster algebras and  the negative simple ones play the role of initial cluster variables. See Theorem \ref{thm:DWZ-g-F} below for a precise statement.
\end{remark}

\begin{theorem}[{\cite[Thm.\ 7.1]{DWZ10}}] \label{thm:E-inv}
Let $\mathcal M=(M,V)$ and $\mathcal N=(N,U)$ be two decorated representations of $(Q,W)$, and let 
$\{\mathcal M^t\mid t\in\mathbb T_n\}$ and $\{\mathcal N^t\mid t\in\mathbb T_n\}$ be the families of decorated representations 
corresponding to $\mathcal M$ and $\mathcal N$ respectively. Then for any two vertices $t,t'\in\mathbb T_n$, we have
  \begin{eqnarray}
      E^{\rm sym}(\mathcal M^t,\mathcal N^t)=E^{\rm sym}(\mathcal M^{t'},\mathcal N^{t'}).\nonumber
  \end{eqnarray}
In particular, the $E$-invariant is a mutation invariant.
\end{theorem}

\subsection{Comparison with $E$-invariant}
In this subsection, we show that $E$-invariant coincides with $F$-invariant for skew-symmetric cluster algebras.
\begin{setting}
 (i) Let $(Q,W)=(Q_{t_0},W_{t_0})$ be a quiver with  non-degenerate potential and $\mathcal M=\mathcal M^{t_0}=(M^{t_0},V^{t_0})$ a decorated representation of $(Q_{t_0},W_{t_0})$. By applying mutations, we have a family of quivers with potentials and a family of decorated representations
\begin{eqnarray}
   [(Q,W)] \coloneqq   \{(Q_t,W_t)\mid t\in\mathbb T_n\},\;\;\;\;[\mathcal M] \coloneqq   \{\mathcal M^t\mid t\in\mathbb T_n\},\nonumber
\end{eqnarray}
where $\mathcal M^t=(M^t, V^t)$ is a decorated representation of $(Q_t,W_t)$.

(ii) Let $B_Q$ be the skew-symmetric matrix corresponding to $Q$ (see \ref{eqn:Q-B-2}), and $\mathcal A_Q$ the cluster algebra with trivial coefficients such that its exchange matrix at vertex $t_0$ is $B_Q$. 

(iii) For a cluster monomial $u={\bf x}_w^{\bf a}=\prod_{i=1}^nx_{i;w}^{a_i}$ of $\mathcal A_Q$, let $\{\mathcal M_u^t\mid t\in\mathbb T_n\}$ be the family of decorated representations determined by the following conditions:
\begin{itemize}
    \item For the vertex $w\in\mathbb T_n$, set $\mathcal M_u^w=(M_u^w,V_u^w)$ to be the negative decorated representation of 
    $(Q_w,W_w)$ given by
$M_u^w=0$ and $V_u^w=(\mathbb C^{a_1},\ldots,\mathbb C^{a_n})$.
\item For any edge \begin{xy}(0,1)*+{t}="A",(10,1)*+{t'}="B",\ar@{-}^k"A";"B" \end{xy}  in ${\mathbb T}_n$, $\mathcal M_u^t$ and $\mathcal M_u^{t'}$ are related by mutation in direction $k$.
\end{itemize}
\end{setting}

The following theorem summarizes some results on categorification of cluster algebras using decorated representations of quivers with potentials
by Derksen, Weyman and Zelevinsky \cite{DWZ10}.

\begin{theorem}[{\cite{DWZ10}}]
\label{thm:DWZ-g-F}
Keep the above setting. The following statements hold.
\begin{itemize}
    \item [(i)] The family $\{\mathcal M_u^t=(M_u^t,V_u^t)\mid t\in\mathbb T_n\}$  of decorated representations only depends on $u$, not on the choice of $w\in\mathbb T_n$ and ${\bf a}\in\mathbb Z_{\geq 0}^n$ such that $u={\bf x}_w^{\bf a}$.
    \item [(ii)] The correspondence $u\mapsto \{\mathcal M_u^t\mid t\in\mathbb T_n\}\mapsto \mathcal M_u:= \mathcal M_u^{t_0}$ induces a bijection from  cluster monomials of $\mathcal A_Q$ to the isomorphism classes of negative-reachable decorated representations of $(Q,W)=(Q_{t_0},W_{t_0})$. 
    \item[(iii)] For any vertex $t\in\mathbb T_n$, we have 
    \[{\bf g}_u^t={\bf g}(\mathcal M_u^t)\;\;\;\text{and}\;\;\;F_u^t(y_1,\ldots,y_n)=F_{\mathcal M_u^t}(y_1,\ldots,y_n).
    \]
    In particular, ${\bf f}_u^t=\underline{\dim}(M_u^t)$, where ${\bf f}_u^t$ is the $f$-vector of $u$ with respect to vertex $t$ and the dimension vector here is viewed as column vector.
\end{itemize}
\end{theorem}

We take $S=I_n$ to be the fixed skew-symmetrizer for the exchange matrices of $\mathcal A_Q$, when we consider the $F$-invariant for cluster monomials in $\mathcal A_Q$.

\begin{theorem}\label{thm:F=E}
Let $(Q,W)$ be a quiver with non-degenerate potential and $\mathcal A_Q$ the corresponding cluster algebra with trivial coefficients. 
Let $u, u'$ be two cluster monomials of $\mathcal A_Q$, and let  $\{\mathcal M_u^t\mid t\in\mathbb T_n\}$ and $\{\mathcal M_{u'}^t\mid t\in\mathbb T_n\}$ be the families of decorated representations corresponding to $u$ and $u'$. Then for any vertex $t\in\mathbb T_n$, we have
\[(u\mid\mid u')_F=E^{\rm sym}(\mathcal M_u^t,\mathcal M_{u'}^t).\] 
In particular, we have $(u\mid\mid u')_F=E^{\rm sym}(\mathcal M_u,\mathcal M_{u'})$, where $\mathcal M_u:=\mathcal M_u^{t_0}$ and $\mathcal M_{u'}:=\mathcal M_{u'}^{t_0}$ are the decorated representations of $(Q,W)$ corresponding to $u$ and $u'$.
\end{theorem}
\begin{proof}
We can assume that $u$ is a cluster monomial in ${\bf x}_{w}$ for some $w\in\mathbb T_n$, say $u=\prod_{i=1}^nx_{i;w}^{a_i}$, where ${\bf a}=(a_1,\ldots,a_n)^T\in\mathbb Z_{\geq 0}^n$. In this case, $\mathcal M_u^w=(M_u^w,V_u^w)$ is given by $M_u^w=0$, $V_u^w=(\mathbb C^{a_1},\ldots, \mathbb C^{a_n})$. Thus we have
 \[
 F_u^{w}=1=F_{\mathcal M_u^w},\;\;\;\text{and}\;\;\;{\bf g}_{u}^{w}={\bf a}={\bf g}({\mathcal M_u^w}).
 \]
By using the vertex $w$ to calculate  $(u\mid\mid u')_F$, we have
\begin{eqnarray}
    (u\mid\mid u')_F=0+F_{u'}^{w}[{\bf g}_u^{w}]=F_{u'}^{w}[{\bf a}]=({\bf f}_{u'}^{w})^T{\bf a}=\langle {\bf f}_{u'}^{w}, {\bf a}\rangle,\nonumber
\end{eqnarray}
where ${\bf f}_{u'}^{w}$ is the $f$-vector of $u'$ with respect to the vertex $w$ and $\langle-,-\rangle$ is the standard inner product on $\mathbb R^n$. On the other hand,  since $\mathcal M_u^w$ is negative and  by Corollary \ref{cor:E-inv}, we have 
$$E^{\rm sym}(\mathcal M_u^w,\mathcal M_{u'}^w)=\langle \underline{\dim} M_{u'}^w, {\bf a}\rangle.$$ By Theorem \ref{thm:DWZ-g-F} (iii), we have ${\bf f}_{u'}^{w}=\underline\dim(M_{u'}^w)$ and thus $(u\mid\mid u')_F=E^{\rm sym}(\mathcal M_u^w,\mathcal M_{u'}^w)$.
 Then by the mutation-invariance of the $E$-invariant, we have \[(u\mid\mid u')_F=E^{\rm sym}(\mathcal M_u^t,\mathcal M_{u'}^t)\] for any vertex $t\in\mathbb T_n$.
\end{proof}

\section{Dominant sets and oriented exchange graphs}\label{sec-dom-sets}
In this section, we introduce the dominant sets for seeds of cluster algebras as a replacement of torsion classes for $\tau$-tilting pairs in $\tau$-tilting theory. Then we give applications to the exchange graphs of cluster algebras. 

Throughout this section, we will focus on cluster variables and cluster monomials. For this reason, we simply work on cluster algebra with trivial coefficients, rather than $\mathsf{\Lambda}$-cluster algebras. By Proposition \ref{pro:trivial-coef}, the $F$-invariant between cluster monomials is well-defined in this case.

 \subsection{Oriented exchange graphs are acyclic}\label{sec51}

 Let $B$ be an $n\times n$ skew-symmetrizable matrix with a fixed skew-symmetrizer $S=diag(s_1,\ldots,s_n)$. Let $\mathcal A$ be the cluster algebra with trivial coefficients whose initial seed is $({\bf x}, B)$. The seed of $\mathcal A$ at vertex $t\in\mathbb T_n$ is denoted by $({\bf x}_t, B_t)$. In particular, we have $({\bf x}_{t_0}, B_{t_0})=({\bf x}, B)$.

Two seeds $({\bf x}_t,  B_t)$ and  $({\bf x}_w,  B_w)$ of $\mathcal A$ are said to be \emph{equivalent} if they are the same up to relabeling. We denote by  $[{\bf x}_t, B_t]$ the equivalent class of $({\bf x}_t,  B_t)$.

The \emph{exchange graph} $\mathcal H=\mathcal H(\mathcal A)$ of $\Acal$ is the graph whose vertices correspond to the seeds (up to seed equivalence) of $\mathcal A$ and whose edges correspond to seed mutations.

For a seed  $({\bf x}_t, B_t)$ of $\mathcal A$, we denote by  $C_t=C_t^{t_0}$ the $C$-matrix of the vertex $t$ with respect to the rooted vertex $t_0$.
\begin{definition}[Green mutation and oriented exchange graph] 
(i) A seed mutation $\mu_k({\bf x}_t, B_t)$ in $\mathcal A$ is called a {\em green mutation}, if the $k$th column of the $C$-matrix $C_t=C_t^{t_0}$ is a non-negative vector. Otherwise, it is called a {\em red mutation}.

  (ii) The  {\em oriented exchange graph} of $\mathcal A$ with initial seed $({\bf x}_{t_0}, B_{t_0})$ is the quiver $\overrightarrow{\mathcal H}^{t_0}=\overrightarrow{\mathcal H}^{t_0}(\mathcal A)$ whose vertices correspond to the seeds  (up to seed equivalence) of $\mathcal A$  and whose arrows correspond to green mutations.
\end{definition}

We remark that if two seeds $({\bf x}_t, B_t)$ and $({\bf x}_w, B_w)$ are related by once mutation, then either the mutation from $({\bf x}_t, B_t)$ to  $({\bf x}_w, B_w)$ is a green mutation or the mutation from $({\bf x}_w, B_w)$ to  $({\bf x}_t, B_t)$ is a green mutation.

Denote by $\mathcal X$ the set of  cluster variables of $\mathcal A$. Given a cluster monomial $u={\bf x}_t^{\bf h}$ in ${\bf x}_t$, that is, ${\bf h}=(h_1,\ldots,h_n)^T\in\mathbb Z_{\geq 0}^n$, we define two subsets of $\mathcal X$:
\begin{itemize}
    \item The subset $\supp(u)=\{x_{i;t}\mid i\in[1,n], \;h_i\neq 0\}\subseteq \mathcal X$ is called the {\em support} of $u={\bf x}_t^{\bf h}$. We remark that this set only depends on $u$, not on the choice of $t\in\mathbb T_n$ and ${\bf h}\in\mathbb Z^{n}$ such that $u={\bf x}_t^{\bf h}$.
    \item For a vertex $w\in\mathbb T_n$ and a cluster variable $z$, recall from Proposition \ref{pro:trivial-coef} that $$(z\mid\mid u)_F=F_z^w[S{\bf g}_u^w]+F_u^w[S{\bf g}_z^w].$$
   The subset
 $\dom^w(u)=\{z\in\mathcal X\mid F_z^{w}[S{\bf g}_u^{w}]=0\}\subseteq \mathcal X$  is called the {\em dominant set} of $u={\bf x}_t^{\bf h}$ with respect to vertex $w\in\mathbb T_n$. 
\end{itemize}

\begin{remark}
    Notice that $F_z^w[S{\bf g}_u^w]$ is the partial $F$-invariant of the pair $(z,u)$ at vertex $w\in\mathbb T_n$. That is, the dominant set is defined using  the vanishing condition of the partial $F$-invariant.
\end{remark}

\begin{definition}[Dominant set of a seed]
    Let  $({\bf x}_t, B_t)$ be a seed of $\mathcal A$. The {\em dominant set} $\dom^{w}[t]$ of $({\bf x}_t, B_t)$ with respect to a vertex $w\in\mathbb T_n$ is defined to be the dominant set of the  multiplicity free cluster monomial $u_t=\prod_{i=1}^nx_{i;t}$ with full support,  \ie
    \[
    \dom^{w}[t] \coloneqq   \dom^{w}(u_t)=\{z\in\mathcal X\mid F_z^{w}[S{\bf g}_{u_t}^{w}]=0\}.
    \]
\end{definition}
Clearly, if $[{\bf x}_{t'}, B_{t'}]=[{\bf x}_t, B_t]$,  \ie the two seeds are equivalent, then $\dom^w[t']=\dom^w[t]$. If we want to view a cluster ${\bf x}_t$ as a set, we will write
 $$[{\bf x}_t] \coloneqq   \{x_{1;t},\ldots,x_{n;t}\}.$$

\begin{proposition}\label{pro:dom-supp}
Let $u={\bf x}_t^{\bf h}$ be a cluster monomial of  $\mathcal A$ in seed $({\bf x}_t,B_t)$.  Then the following statements hold for any vertex $w\in\mathbb T_n$.

\begin{itemize}
\item [(i)] $[{\bf x}_w]\subseteq\dom^w(u)$ and in particular, $[{\bf x}_w]\subseteq\dom^w[t]$. \vspace{1.5mm}

    \item [(ii)] $[{\bf x}_t]\subseteq\dom^w(u)$ and in particular,   $[{\bf x}_t]\subseteq\dom^w[t]$. \vspace{1.5mm}
    
    \item[(iii)] $\dom^w(u)=\bigcap_{x\in\supp(u)}\dom^w(x)$ and in particular, $\dom^w[t]=\bigcap_{i=1}^n\dom^w(x_{i;t})$.
\end{itemize}
\end{proposition}
\begin{proof}
(i) Since $F_{x_{i;w}}^w=1$ for $i\in[1,n]$, we have $F_{x_{i;w}}^{w}[S{\bf g}_u^{w}]=0$. Thus $x_{i;w}\in\dom^w(u)$ for $i\in[1,n]$, that is, $[{\bf x}_w]\subseteq\dom^w(u)$.  In particular, if we take $u=u_t=\prod_{i=1}^nx_{i;t}$,  we obtain $$[{\bf x}_w]\subseteq\dom^{w}(u_t)=\dom^{w}[t].$$

(ii) Fix $i\in[1,n]$. By using the vertex $t$ to calculate  $(x_{i;t}\mid\mid u)_F=(x_{i;t}\mid\mid {\bf x}_t^{\bf h})_F$, we see $(x_{i;t}\mid\mid u)_F=0$. On the other hand, we know that
 $$(x_{i;t}\mid\mid u)_F=F_{x_{i;t}}^{w}[S{\bf g}_{u}^{w}]+F_{u}^{w}[S{\bf g}_{x_{i;t}}^{w}]$$ and 
 each term in this sum is non-negative. Thus we obtain $F_{x_{i;t}}^{w}[S{\bf g}_{u}^{w}]=0$,  \ie $x_{i;t}\in\dom^{w}(u)$. Since $i\in[1,n]$ is arbitrary, we get $[{\bf x}_t]\subseteq \dom^{w}(u)$. In particular, if we take $u=u_t=\prod_{i=1}^nx_{i;t}$,  we obtain $[{\bf x}_t]\subseteq\dom^{w}(u_t)=\dom^{w}[t]$.

(iii) By definition, a cluster variable $z$ belongs to $\dom^w(u)$ if and only if $$F_z^{w}[S{\bf g}_u^{w}]=0.$$ By Proposition \ref{pro:bigood} (ii), we have
\[F_{z}^w[S{\bf g}_{u}^w]=F_{z}^w[S{\bf g}_{{\bf x}_t^{\bf h}}^w]=\sum_{j=1}^nh_{j}\cdot F_{z}^w[S{\bf g}_{x_{j;t}}^w].\]
Since $h_j\geq 0$ and $F_{z}^w[S{\bf g}_{x_{j;t}}^w]\geq 0$ for any $j\in[1,n]$, we know that $F_z^{w}[S{\bf g}_u^{w}]=0$ if and only if
$F_z^{w}[S{\bf g}_{x_{j;t}}^{w}]=0$ for any $j\in[1,n]$ with $h_j\neq 0$,  \ie for any $x_{j;t}\in\supp(u)$. This is equivalent to say that $z$ belongs to  $\bigcap_{x\in\supp(u)}\dom^w(x)$. 
Hence, $\dom^w(u)=\bigcap_{x\in\supp(u)}\dom^w(x)$. In particular, if we take  $u=u_t=\prod_{i=1}^nx_{i;t}$, we obtain $\dom^w[t]=\dom^w(u_t)=\bigcap_{i=1}^n\dom^w(x_{i;t})$.
\end{proof}

\begin{corollary}\label{cor-supp}
  Let $u$ and $u'$ be two cluster monomials of $\mathcal A$ such that $\supp(u)=\supp(u')$. Then $\dom^w(u)=\dom^w(u')$ for any vertex $w\in\mathbb T_n$. In particular, $\dom^w[t]=\dom^w(u)$ for any cluster monomial $u$ with $\supp(u)=[{\bf x}_t]$.
\end{corollary}
\begin{proof}
This follows from Proposition \ref{pro:dom-supp} (iii).
\end{proof}

For a seed  $({\bf x}_t, B_t)$ of $\mathcal A$, we denote by   $G_t=G_t^{t_0}$ the $G$-matrix of the vertex $t$ with respect to the rooted vertex $t_0$ and by
\[ \mathbb R_{\geq 0}G_t \coloneqq   \mathbb R_{\geq 0}{\bf g}_{1;t}+\cdots+\mathbb R_{\geq 0}{\bf g}_{n;t}\subseteq \mathbb R^n\]
 the cone in $\mathbb R^n$ spanned by the columns of $G_t=({\bf g}_{1;t},\ldots,{\bf g}_{n;t})$. We also denote by 
\begin{align}
    \mathbb Q_{>0}G_t& \coloneqq  \mathbb Q_{>0}{\bf g}_{1;t}+\cdots+\mathbb Q_{>0}{\bf g}_{n;t}\subseteq \mathbb Q^n,\nonumber\\
    \mathbb Z_{\geq 1}G_t& \coloneqq  \mathbb Z_{\geq 1}{\bf g}_{1;t}+\cdots+\mathbb Z_{\geq 1}{\bf g}_{n;t}\subseteq \mathbb Z^n.\nonumber
\end{align}

 \begin{lemma}\label{lem:v-v}
     Let $({\bf x}_{t'}, B_{t'})=\mu_k({\bf x}_t,  B_t)$ be a green mutation in $\mathcal A$. 
      Then there exist two vectors ${\bf v},{\bf v}'\in\mathbb Z_{\geq 1}^n$ such that $G_t{\bf v}-G_{t'}{\bf v}'\in\mathbb Z_{\geq 0}^n$.
 \end{lemma}
\begin{proof}
Without loss of  generality, we can assume that $k=1$. We write $G_t=({\bf g}_1,{\bf g}_2,\ldots,{\bf g}_n)$, where ${\bf g}_i\in\mathbb Z^n$. Since $({\bf x}_{t'}, B_{t'})=\mu_k({\bf x}_t, B_t)=\mu_1({\bf x}_t, B_t)$, the $G$-matrix $G_{t'}$ has the form $G_{t'}=({\bf g}_1',{\bf g}_2,\ldots,{\bf g}_n)$.

By Theorem \ref{thm:CG} (i), we know that $\mathbb R_{\geq 0}G_t$ and $\mathbb R_{\geq 0}G_{t'}$ are two cones of dimension $n$ in $\mathbb R^n$ and they have a common facet $\mathcal C$ (cone of dimension $n-1$) given by
$\mathcal C=\mathbb R_{\geq 0}{\bf g}_2+\ldots+\mathbb R_{\geq 0}{\bf g}_n$.

 Let $S=diag(s_1,\ldots,s_n)$ be the fixed skew-symmetrizer for the exchange matrices of $\mathcal A$ and ${\bf c}_{1;t}$ the first column of the $C$-matrix $C_t=C_t^{t_0}$.  By Theorem \ref{thm:CG} (iii), we have $SC_tS^{-1}(G_t)^{\rm T}=I_n$, equivalently, $(G_t)^{\rm T}(SC_tS^{-1})=I_n$. So  the first column vector $s_1^{-1}(S{\bf c}_{1;t})$ of $SC_tS^{-1}$ is an inner normal vector of the facet $\mathcal C$ of the cone $\mathbb R_{\geq 0}G_t$. So when we go along the direction $s_1^{-1}(S{\bf c}_{1;t})$, we can cross the common facet $\mathcal C$ from the cone $\mathbb R_{\geq 0}G_{t'}$ to the cone $\mathbb R_{\geq 0}G_{t}$.

 Now we take a vector ${\bf h}'\in\mathbb Z_{\geq 1}G_{t'}\subseteq \mathbb R_{\geq 0}G_{t'}$. Then there exists $\lambda\in\mathbb Q_{>0}$  such that $${\bf h} \coloneqq   {\bf h}'+\lambda(s_1^{-1}(S{\bf c}_{1;t}))\in\mathbb Q_{>0}G_t\subseteq \mathbb R_{\geq 0}G_t.$$ We can take a positive integer $p$ large enough such that $p{\bf h}\in \mathbb Z_{\geq 1}G_t$.
 Since $p{\bf h}\in \mathbb Z_{\geq 1}G_t\subseteq \mathbb Z^n$ 
 and $p{\bf h}'\in\mathbb Z_{\geq 1}G_{t'}\subseteq \mathbb Z^n$, we have $p\lambda s_1^{-1}(S{\bf c}_{1;t})=p{\bf h}-p{\bf h}'\in\mathbb Z^n$. Since $({\bf x}_{t'}, B_{t'})=\mu_k({\bf x}_t, B_t)$ is a green mutation, we have ${\bf c}_{1;t}\in\mathbb Z_{\geq 0}^n$. So we have $p{\bf h}-p{\bf h}'=p\lambda s_1^{-1}(S{\bf c}_{1;t})\in\mathbb Z_{\geq 0}^n$. 

 Since $p{\bf h}\in\mathbb Z_{\geq 1}G_t$ and $p{\bf h}'\in\mathbb Z_{\geq 1}G_{t'}$, there exist ${\bf v},{\bf v}'\in\mathbb Z_{\geq 1}^n$ such that 
 \[
 p{\bf h}=G_t{\bf v}\;\;\;\text{and}\;\;\; p{\bf h}'=G_{t'}{\bf v}'.
 \]
 Thus $G_t{\bf v}-G_{t'}{\bf v}'=p{\bf h}-p{\bf h}'\in\mathbb Z_{\geq 0}^n$. 
\end{proof}

\begin{lemma}\label{lem:r-r}
 Let $F$ be a polynomial in $\mathbb Z[y_1,\ldots,y_n]$ with constant term $1$.   If  $({\bf r}-{\bf r}')\in\mathbb Z_{\geq 0}^n$ for some ${\bf r},{\bf r}'\in\mathbb Z^n$ and $F[{\bf r}]=0$, then $F[{\bf r}']=0$.
\end{lemma}
    \begin{proof}
Let ${\bf y}^{\bf u}$ be a monomial appearing in $F$. We know that ${\bf u}\in\mathbb Z_{\geq 0}^n$. Since $({\bf r}-{\bf r}')\in\mathbb Z_{\geq 0}^n$, we have ${\bf u}^T({\bf r}'-{\bf r})=-{\bf u}^T({\bf r}-{\bf r}')\leq 0$. Since $F[{\bf r}]=0$, we have ${\bf u}^T{\bf r}\leq 0$. Thus \[{\bf u}^T{\bf r}'={\bf u}^T({\bf r}'-{\bf r}+{\bf r})={\bf u}^T({\bf r}'-{\bf r})+{\bf u}^T{\bf r}\leq 0.\]
Since $F$ has constant term $1$ and by the definition of $F[{\bf r}']$, we obtain $F[{\bf r}']=0$.
  \end{proof}

\begin{theorem}\label{thm:graph}
    If $({\bf x}_{t'}, B_{t'})=\mu_k({\bf x}_t, B_t)$ is a green mutation, then $\dom^{t_0}[t]\subsetneq \dom^{t_0}[t']$. In particular, the oriented exchange graph $\overrightarrow{\mathcal H}^{t_0}=\overrightarrow{\mathcal H}^{t_0}(\mathcal A)$ of $\mathcal A$ is acyclic.
\end{theorem}
\begin{proof}
Since $({\bf x}_{t'},B_{t'})=\mu_k({\bf x}_t, B_t)$ is a green mutation and by Lemma \ref{lem:v-v}, there exist 
$${\bf v}=(v_1,\ldots,v_n)^T,\; {\bf v}'=(v_1',\ldots,v_n')^T\in\mathbb Z_{\geq 1}^n$$ such that $G_t{\bf v}-G_{t'}{\bf v}'\in\mathbb Z_{\geq 0}^n$. Set $u=\prod_{i=1}^nx_{i;t}^{v_i}$ and $u'=\prod_{i=1}^nx_{i;t'}^{v_i'}$. We know that
$${\bf g}_u^{t_0}=G_t{\bf v},\;\;{\bf g}_{u'}^{t_0}=G_{t'}{\bf v}'.$$ 
Since $\supp(u)=[{\bf x}_t],\;\supp(u')=[{\bf x}_{t'}]$ and by Corollary \ref{cor-supp}, we have that
\begin{align}
  \dom^{t_0}[t]&=\dom^{t_0}(u)=\{z\in\mathcal X\mid 
  F_z^{t_0}[S{\bf g}_u^{t_0}]=0\}=\{z\in\mathcal X\mid 
  F_z^{t_0}[SG_t{\bf v}]=0\},\nonumber\\
   \dom^{t_0}[t']&=\dom^{t_0}(u')=\{z\in\mathcal X\mid 
  F_z^{t_0}[S{\bf g}_{u'}^{t_0}]=0\}=\{z\in\mathcal X\mid 
  F_z^{t_0}[SG_{t'}{\bf v}']=0\}.\nonumber
\end{align}
Let $z\in\dom^{t_0}[t]$, that is, we have
$F_z^{t_0}[SG_t{\bf v}]=0$. Since
 $$SG_t{\bf v}-SG_{t'}{\bf v}'=S(G_t{\bf v}-G_{t'}{\bf v}')\in\mathbb Z_{\geq 0}^n$$
 and  by Lemma \ref{lem:r-r}, we get $F_z^{t_0}[SG_{t'}{\bf v}']=0$. This implies that $z\in\dom^{t_0}[t']$. So we have $\dom^{t_0}[t]\subseteq \dom^{t_0}[t']$. 
 
In the rest part of the proof, we will show  $x_{k;t'}\in\dom^{t_0}[t']$ but $x_{k;t'}\notin\dom^{t_0}[t]$. Thus the inclusion $\dom^{t_0}[t]\subseteq \dom^{t_0}[t']$ is strict.

By Proposition \ref{pro:dom-supp} (ii),  we have $[{\bf x}_t]\subseteq \dom^{t_0}[t]$ and $[{\bf x}_{t'}]\subseteq \dom^{t_0}[t']$. In particular, $$x_{k;t}\in \dom^{t_0}[t]\;\;\; \text{and}\;\;\; x_{k;t'}\in\dom^{t_0}[t'].$$ Since $\dom^{t_0}[t]\subseteq \dom^{t_0}[t']$, we get $x_{k;t}\in\dom^{t_0}[t']$. Then by Proposition \ref{pro:dom-supp} (iii), we have
$x_{k;t}\in\dom^{t_0}[t']\subseteq \dom^{t_0}(x_{k;t'})$. This implies
$$F_{x_{k;t}}^{t_0}[S{\bf g}_{x_{k;t'}}^{t_0}]=0.$$
By using the vertex $t_0$ to calculate the $F$-invariant $(x_{k;t}\mid\mid x_{k;t'})_F$, we have $$(x_{k;t}\mid\mid x_{k;t'})_F=F_{x_{k;t}}^{t_0}[S{\bf g}_{x_{k;t}'}^{t_0}]+F_{x_{k;t'}}^{t_0}[S{\bf g}_{x_{k;t}}^{t_0}]=F_{x_{k;t'}}^{t_0}[S{\bf g}_{x_{k;t}}^{t_0}].$$
On the other hand, if we use the vertex $t$ to calculate  $(x_{k;t}\mid\mid x_{k;t'})_F$, we see that
$$(x_{k;t}\mid\mid x_{k;t'})_F=0+F_{x_{k;t'}}^{t}[S{\bf g}_{x_{k;t}}^{t}]=(1+y_k)[S{\bf e}_k]=s_k,$$
where ${\bf e}_k$ is the $k$th column of $I_n$ and $s_k$ is the $(k,k)$-entry of $S=diag(s_1,\ldots,s_n)$. Thus $$F_{x_{k;t'}}^{t_0}[S{\bf g}_{x_{k;t}}^{t_0}]=(x_{k;t}\mid\mid x_{k;t'})_F=s_k>0,$$ which means $x_{k;t'}\notin \dom^{t_0}(x_{k;t})$. By Proposition \ref{pro:dom-supp} (iii), we know that $\dom^{t_0}[t]\subseteq \dom^{t_0}(x_{k;t})$. Thus $x_{k;t'}\notin \dom^{t_0}(x_{k;t})$ implies that $x_{k;t'}\notin \dom^{t_0}[t]$.

Since $x_{k;t'}\in\dom^{t_0}[t']$ but $x_{k;t'}\notin \dom^{t_0}[t]$, we know that the inclusion $\dom^{t_0}[t]\subseteq \dom^{t_0}[t']$ is strict,  \ie we have  $\dom^{t_0}[t]\subsetneq \dom^{t_0}[t']$.
\end{proof}

\begin{example}\label{ex:A2-dom}
    Take $B=\begin{bmatrix}
    0&1\\-1&0
\end{bmatrix}$, ${\bf x}=(x_1,x_2)$ and $S=I_2$. 
The oriented exchange graph $\overrightarrow{\mathcal H}^{t_0}$ of $\mathcal A=\mathcal A({\bf x}, B)$ is as follows:
\[
\xymatrix{&t_0=\{x_1,x_2\}\ar[ld]_{\mu_{x_1}}\ar[rd]^{\mu_{x_2}}&\\
t_1=\{{x_3},x_2\}\ar[d]_{\mu_{x_2}}&&t_4=\{x_5,x_1\}\ar[d]^{\mu_{x_1}}\\
t_2=\{x_3,{x_4}\}\ar[rr]^{\mu_{x_3}}&&t_3=\{{x_5},x_4\}}
\]
where $x_3=x_1^{-1}x_2\cdot(1+\widehat y_1),  x_4=x_1^{-1}\cdot (1+\widehat y_1+\widehat y_1\widehat y_2),
    x_5=x_2^{-1}\cdot (1+\widehat y_2)$.  It is easy to check that
    \begin{gather}
      \dom^{t_0}(x_1)=\{x_1,x_2,x_5\},\;\;\;
\dom^{t_0}(x_2)=\{x_1,x_2,x_3\},\;\;\;
\dom^{t_0}(x_3)=\{x_1,x_2,x_3,x_4\},\nonumber\\
\dom^{t_0}(x_4)=\{x_1,x_2,x_3,x_4,x_5\}=
\dom^{t_0}(x_5).\nonumber
    \end{gather}
Then by Proposition \ref{pro:dom-supp} (iii), we can get that
    \begin{gather}
        \dom^{t_0}[t_0]=\{x_1, x_2\}, \;\;\dom^{t_0}[t_1]=\{x_1, x_2, x_3\},\;\;\dom^{t_0}[t_2]=\{x_1, x_2, x_3, x_4\},\nonumber\\
        \dom^{t_0}[t_3]=\{x_1, x_2, x_3, x_4, x_5\},\;\;\dom^{t_0}[t_4]=\{x_1, x_2, x_5\}.\nonumber
    \end{gather}
 We can see that the initial cluster variables $x_1,x_2$ are contained in any dominant set with respect to the vertex $t_0$, this is because the $F$-polynomials of initial cluster variables are $1$.
\end{example}

 Let $\mathcal H_0$ be the set of seeds (up to seed equivalence) of $\mathcal A$, equivalently,  $\mathcal H_0$ is the vertex set of the exchange graph of $\mathcal A$. For the rooted vertex $t_0\in\mathbb T_n$, we define an relation $\leq_{t_0}$ on $\mathcal H_0$ as follows: $[{\bf x}_w,B_w]\leq_{t_0} [{\bf x}_t, B_t]$ if and only if there  exists a sequence $\overleftarrow{\mu}$ of green mutations such that $$[\overleftarrow{\mu}({\bf x}_w, B_w)]=[{\bf x}_t, B_t],$$  \ie there exists a  path from the vertex $[{\bf x}_w, B_w]$ to the vertex $[{\bf x}_t, B_t]$ in the oriented exchange graph $\overrightarrow {\mathcal H}^{t_0}=\overrightarrow {\mathcal H}^{t_0}(\mathcal A)$ of $\mathcal A$. 
 
 If $[{\bf x}_w, B_w]\leq_{t_0} [{\bf x}_t, B_t]$ and $[{\bf x}_t, B_t]\neq  [{\bf x}_w, B_w]$, we write $[{\bf x}_w, B_w]<_{t_0} [{\bf x}_t,  B_t]$.

\begin{corollary}
The relation $\leq_{t_0}$ gives a partial order on $\mathcal H_0$,  \ie  $(\mathcal H_0,\leq_{t_0})$ is a poset.
\end{corollary}

\begin{proof}
By Theorem \ref{thm:graph}, the oriented exchange graph of $\mathcal A$ is acyclic. The result follows.
\end{proof}

The {\em Hasse quiver}  of the poset set $(\mathcal H_0,\leq_{t_0})$ is a quiver $\overrightarrow {\mathcal H}'$ defined as follows:

\begin{itemize}
    \item the vertex set of $\overrightarrow {\mathcal H}'$ is $\mathcal H_0$;
    \item $ [{\bf x}_w, B_w]\rightarrow [{\bf x}_t, B_t]$ is an arrow in $\overrightarrow {\mathcal H}'$ if and only if  $ [{\bf x}_w, B_w]<_{t_0}[{\bf x}_t, B_t]$ and there exists no seed $({\bf x}_s, B_s)$ such that $ [{\bf x}_w, B_w]<_{t_0}[{\bf x}_s, B_s]<_{t_0}[{\bf x}_t, B_t]$. In other words, the arrows correspond to the cover relations under the partial order $\leq_{t_0}$.
\end{itemize}
Notice that the Hasse quiver $\overrightarrow{\mathcal H}'$ and the oriented exchange graph $\overrightarrow{\mathcal H}^{t_0}$ of $\mathcal A$ have the same vertex set $\mathcal H_0$. It is easy to see that the arrow set of  $\overrightarrow{\mathcal H}'$ is a subset of the arrow set of $\overrightarrow{\mathcal H}^{t_0}$ of $\mathcal A$. It is natural to ask whether the two quivers coincide. The next subsection aims to answer this question.

\subsection{Oriented exchange graphs are Hasse quivers}\label{sec52}
 In this subsection, we show that the oriented exchange graph 
$\overrightarrow{\mathcal H}^{t_0}=\overrightarrow{\mathcal H}^{t_0}(\mathcal A)$ of $\mathcal A$ coincides with the Hasse quiver $\overrightarrow{\mathcal H}'$ of the poset $(\mathcal H_0,\leq_{t_0})$.

\begin{proposition}\label{pro:p-dom}
Let $({\bf x}_t, B_t)$ be a seed of $\mathcal A$. Denote by 
\begin{align}
    \mathcal P(\dom^{t_0}[t])&=\{x\in\dom^{t_0}[t]\mid F_z^{t_0}[S{\bf g}_x^{t_0}]=0, \;\forall z\in\dom^{t_0}[t]\}\nonumber\\
    &=\{x\in\dom^{t_0}[t]\mid \dom^{t_0}[t]\subseteq\dom^{t_0}(x)\}.\nonumber
\end{align}
Then 
$\mathcal P(\dom^{t_0}[t])=[{\bf x}_t]$ and  in particular, $[{\bf x}_t]$ is uniquely determined by the dominant set $\dom^{t_0}[t]$.
\end{proposition}
\begin{proof}
Fix $k\in[1,n]$. By Proposition \ref{pro:dom-supp} (ii), (iii), we know that $$x_{k;t}\in\dom^{t_0}[t]\;\;\; \text{and}\;\;\; \dom^{t_0}[t]\subseteq \dom^{t_0}(x_{k;t}).$$  So we have $x_{k;t}\in\mathcal P(\dom^{t_0}[t])$. Since $k$ is an arbitrary integer $[1,n]$, we get $[{\bf x}_t] \subseteq \mathcal P(\dom^{t_0}[t])$.

Now let us take $x\in \mathcal P(\dom^{t_0}[t])$. We know that $F_z^{t_0}[S{\bf g}_x^{t_0}]=0$ for any $z\in\dom^{t_0}[t]$. Since $x_{k;t}$ is in $\dom^{t_0}[t]$ for $k\in[1,n]$, we have $F_{x_{k;t}}^{t_0}[S{\bf g}_{x}^{t_0}]=0$  for $k\in[1,n]$. 

Since $x\in \mathcal P(\dom^{t_0}[t])\subseteq \dom^{t_0}[t] \subseteq \dom^{t_0}(x_{k;t})$, we have $F_x^{t_0}[S{\bf g}_{x_{k;t}}^{t_0}]=0$.
Thus $$(x_{k;t}\mid\mid x)_F=F_{x_{k;t}}^{t_0}[S{\bf g}_{x}^{t_0}]+F_x^{t_0}[S{\bf g}_{x_{k;t}}^{t_0}]=0.$$
Thus for any two cluster variables $z_1,z_2$
in $[{\bf x}_t]\cup \{x\}$, we have $(z_1\mid\mid z_2)_F=0$. Then by Lemma \ref{lem:cluster},  we know that $[{\bf x}_t]\cup \{x\}$ is a subset of some cluster of $\mathcal A$.  We must have $x=x_{i;t}$ for some $i\in[1,n]$. Thus we get $$\mathcal P(\dom^{t_0}[t])\subseteq [{\bf x}_t].$$
Hence, $\mathcal P(\dom^{t_0}[t])= [{\bf x}_t]$. 
\end{proof}

The following result is analogous to \cite[Proposition 2.26]{air_2014} in  $\tau$-tilting theory.
\begin{lemma}\label{lem:leq}
Let $({\bf x}_t,B_t)$, $({\bf x}_s, B_s)$ and $({\bf x}_w, B_w)$ be three seeds of $\mathcal A$. If $$\dom^{t_0}[w]\subseteq \dom^{t_0}[s]\subseteq\dom^{t_0}[t],$$
then $[{\bf x}_t]\cap [{\bf x}_w]\subseteq [{\bf x}_s]$. 
\end{lemma}
\begin{proof}
By the definition of $\mathcal P(\dom^{t_0}[s])$, we can see 
$\mathcal P(\dom^{t_0}[t])\cap \dom^{t_0}[s]\subseteq \mathcal P(\dom^{t_0}[s])$.  By Proposition \ref{pro:p-dom}, we know that  $[{\bf x}_t]=\mathcal P(\dom^{t_0}[t])$ and  $[{\bf x}_s]=\mathcal P(\dom^{t_0}[s])$. 
Thus we have
$[{\bf x}_t]\cap \dom^{t_0}[s]\subseteq [{\bf x}_s]$.

By Proposition \ref{pro:dom-supp} (ii), we know that 
$[{\bf x}_w]\subseteq\dom^{t_0}[w]$. Then by $\dom^{t_0}[w]\subseteq  \dom^{t_0}[s]$, we have $[{\bf x}_w]\subseteq \dom^{t_0}[s]$.
So we know that
$[{\bf x}_t]\cap[{\bf x}_w]\subseteq [{\bf x}_t]\cap  \dom^{t_0}[s]\subseteq [{\bf x}_s]$. 
\end{proof}

\begin{theorem}\label{thm:hasse}
The following statements are equivalent:
\begin{itemize}
    \item [(a)]  There exists an arrow $[{\bf x}_{w}, B_w]\rightarrow [{\bf x}_t, B_t]$ in the oriented exchange graph 
$\overrightarrow{\mathcal H}^{t_0}$ of $\mathcal A$.
\item [(b)] $[{\bf x}_w, B_w]<_{t_0}[{\bf x}_t, B_t]$ and  there exists no seed $({\bf x}_s, B_s)$ such that $$[{\bf x}_w, B_w]<_{t_0}[{\bf x}_s, B_s]<_{t_0}[{\bf x}_t, B_t].$$
\end{itemize}

\end{theorem}
\begin{proof}
$``(a) \Rightarrow (b)"$: Suppose that $[{\bf x}_{w}, B_w]\rightarrow [{\bf x}_t, B_t]$ is an arrow in the oriented exchange graph 
$\overrightarrow{\mathcal H}^{t_0}$ of $\mathcal A$. Then we know that $[{\bf x}_w, B_w]<_{t_0}[{\bf x}_t, B_t]$.

Since the two seeds $({\bf x}_t, B_t)$ and $({\bf x}_w, B_w)$ differ by once mutation, for simplicity, we can assume $({\bf x}_t, B_t)=\mu_1({\bf x}_w, B_w)$. If we write ${\bf x}_w=(z_1,z_2,\ldots,z_n)$, then ${\bf x}_t$ has the form ${\bf x}_t=(z_1',z_2,\ldots,z_n)$.

If there exists a seed  $({\bf x}_s, B_s)$ such that $[{\bf x}_w, B_w]<_{t_0}[{\bf x}_s, B_s]\leq_{t_0}[{\bf x}_t, B_t]$. 
Then we have $\dom^{t_0}[w]\subsetneq \dom^{t_0}[s]\subseteq\dom^{t_0}[t]$. By Lemma \ref{lem:leq}, we have $\{z_2,\ldots,z_m\}=[{\bf x}_t]\cap[{\bf x}_w]\subseteq[{\bf x}_s]$. 
Since $[{\bf x}_w, B_w]<_{t_0}[{\bf x}_s, B_s]$, we know that $[{\bf x}_s]\neq [{\bf x}_w]$. Thus the two clusters ${\bf x}_s$ and ${\bf x}_w$ have exactly $n-1$ common cluster variables $z_2,\ldots,z_n$.  Then by \cite[Corollary 3]{cao-li-2020}, we get $[{\bf x}_s, B_s]=[\mu_1({\bf x}_w, B_w)]=[{\bf x}_t, B_t]$. So there exists no seed $({\bf x}_s, B_s)$ such that $[{\bf x}_w, B_w]<_{t_0}[{\bf x}_s, B_s]<_{t_0}[{\bf x}_t, B_t]$.

$``(b) \Rightarrow (a)"$: Since
$[{\bf x}_w, B_w]<_{t_0}[{\bf x}_t, B_t]$, there exists a non-empty sequence $\overleftarrow{\mu}$ of green mutations from $[{\bf x}_w, B_w]$ to $[{\bf x}_t, B_t]$. Since  there exists no seed $({\bf x}_s, B_s)$ such that $$[{\bf x}_w, B_w]<_{t_0}[{\bf x}_s, B_s]<_{t_0}[{\bf x}_t, B_t],$$
the length of $\overleftarrow{\mu}$ must be $1$,  \ie $[{\bf x}_t, B_t]$ is obtained from $[{\bf x}_w, B_w]$ by once green mutation. So there exists an arrow $[({\bf x}_{w}, B_w)]\rightarrow [({\bf x}_t, B_t)]$ in the oriented exchange graph 
$\overrightarrow{\mathcal H}^{t_0}$ of $\mathcal A$.
\end{proof}
The above theorem rules out the existence of the subquiver $\widetilde A_{1,d}$ ($d\geq 2$) of the form
\[\xymatrix{\bullet\ar[r]^{\beta_1}\ar[d]_{\alpha}&\bullet\ar@{.}[d]\\
\bullet&\bullet\ar[l]_{\beta_d}
}
\]
in the oriented exchange graph $\overrightarrow{\mathcal H}^{t_0}$ of $\mathcal A$.

The following result is a direct consequence of the above theorem and it  is analogous to the corresponding result \cite[Theorem 0.6]{air_2014} in the $\tau$-tilting theory.
\begin{corollary}\label{cor:hasse}
      The oriented exchange graph $\overrightarrow{\mathcal H}^{t_0}$ of $\mathcal A$ coincides with the Hasse quiver $\overrightarrow{\mathcal H}'$ of the poset $(\mathcal H_0,\leq_{t_0})$.
\end{corollary}

\subsection{Dominant sets and  torsion classes} \label{sec53}

In this subsection we briefly explain why the dominant sets can be viewed as a replacement of torsion classes in $\tau$-tilting theory.

 We fix a finite dimensional basic algebra $A$ over a field $\mathbf{k}$.  Denote by $\mod A$ the category of finitely generated left $A$-modules, and by $\tau$ the Auslander-Reiten translation in $\mod A$.  The  isomorphism classes of indecomposable projective modules in $\mod A$ are denoted by $P_1,\ldots,P_n$.

Given a module $M\in\mod A$, we denote by
\begin{itemize}
\item $\add M$  the subcategory of $\mod A$ consisting of modules which are direct summands of $M^{\oplus p}$ for some $p>0$.
\item $\Fac M$ the factor modules of the modules in $\add M$.
\item $\prescript{\bot}{}{M} \coloneqq   \{X\in \mod A  \mid \Hom_{A}(X,M)=0\}$.
\item $M^{\bot} \coloneqq   \{Y\in \mod A  \mid \Hom_{A}(M,Y)=0\}$.
\item $|M|$ the number of non-isomorphic indecomposable direct summands of $M$.
\end{itemize}

\begin{definition}Let $M$ be a module in $\mod A$ and $P$ a projective module in $\mod A$.

\begin{itemize}
    \item [(i)] $M$ is called {\em $\tau$-rigid} if $\Hom_A(M,\tau M)=0$.
\item[(ii)] The pair $(M,P)$ is called \emph{$\tau$-rigid} if $M$ is $\tau$-rigid and $\Hom_A(P,M)=0$. 
\item[(iii)] The pair $(M,P)$ is called \emph{$\tau$-tilting} if $(M,P)$ is a $\tau$-rigid pair and  $|M|+|P|=|A|=n$.
\end{itemize}
\end{definition}
 We will always consider modules and $\tau$-rigid pairs up to isomorphisms. A $\tau$-rigid pair $(M,P)$ is \emph{indecomposable} if $M\oplus P$ is indecomposable in $\mod A$ and it is \emph{basic} if both $M$ and $P$ are basic in $\mod A$. In a (basic) $\tau$-tilting pair $(M,P)$, it is known \cite[Proposition 2.3]{air_2014} that $P$ is uniquely determined by $M$.

A subcategory $\mathcal T$ of $\mod A$ is called a {\em torsion class}, if $\mathcal T$ is closed under quotients and extensions. A torsion class $\mathcal T$ is said to be {\em functorially finite}, if there exists a module $M\in\mod A$ such that $\mathcal T=\Fac M$.

Let $\mathcal C$ be a subcategory of $\mod A$. A module $U\in\mathcal C$ is said to be {\em Ext-projective} in $\mathcal C$, if $\Ext_A^1(U,\mathcal C)=0$. We denote by $\mathcal P(\mathcal C)$ the direct sum of one copy of each of the indecomposable Ext-projective objects in $\mathcal C$ up to isomorphisms.

\begin{theorem}[{\cite[Prop. 1.2 (b), Thm. 2.7]{air_2014}}] \label{thm:air-torsion}
The following statements hold.
\begin{itemize}
    \item [(i)] There is a well-defined map $\Psi$ 
from $\tau$-rigid pairs to functorially finite torsion classes in $\mod A$
given by $(M,P)\mapsto \Fac M$.

\item[(ii)] The above map $\Psi$ is a bijection if we restrict it to basic $\tau$-tilting pairs, which induces a natural order on the set of basic $\tau$-tilting pairs 
$$(M,P)\leq(M^\prime,P^\prime)\overset{{\rm def.}}{\Longleftrightarrow}\Fac M\subseteq\Fac M^\prime.$$

\item[(iii)] Let $\mathcal T$ be a functorially finite torsion class and denote by $(M,P)$ the basic $\tau$-tilting pair such that $\Fac M=\mathcal T$. Then  $M=\mathcal P(\mathcal T)$.
\end{itemize}
\end{theorem}

One can refer to \cite[Theorem 2.19]{air_2014} for the definition of mutations of $\tau$-tilting pairs. Adachi, Iyama and Reiten \cite[Theorem 0.6]{air_2014} proved that mutations  of $\tau$-tilting pairs correspond to the cover relations in the poset of $\tau$-tilting pairs (or functorially finite torsion classes).  Some correspondences between cluster algebras and $\tau$-tilting theory are given in Table \ref{table}.
\begin{table}[ht]
\begin{equation*}
\begin{array}{|c|c|}
\hline
 &
\\[-3mm]
\text{Cluster algebras}& \hspace{2mm} \tau\text{-tilting theory}
\\[2mm]
\hline
 &
\\[-3mm]
\text{Seeds}& \hspace{2mm} \tau\text{-tilting pairs}
\\[2mm]
\hline
 &
\\[-2mm]
\text{Mutations of seeds}& \hspace{2mm} \text{Mutations of } \tau\text{-tilting pairs}
\\[2mm]
\hline
 &
\\[-2mm]
\text{Initial cluster variables } & \hspace{2mm} (0,P_1),\ldots,(0,P_n)
\\[2mm]
\hline
 &
\\[-3mm]
\text{Non-initial  cluster variables}& \hspace{2mm} \text{Indecomposable }\tau\text{-rigid modules}
\\[2mm]
\hline
 &
\\[-3mm]
\text{Cluster variables}& \hspace{2mm} \text{Indecomposable }\tau\text{-rigid pairs}
\\[2mm]
\hline
\end{array}
\end{equation*}
\caption{Cluster algebras vs. $\tau$-tilting theory \label{table}}
\end{table}

Consider $M\in\mod A$ and let 
\[\bigoplus_{i=1}^nP_i^{b_i}\rightarrow
\bigoplus_{i=1}^nP_i^{a_i}\rightarrow M\rightarrow 0
\]
be the minimal projective presentation of $M$ in $\mod A$. The
vector $$\delta_M \coloneqq   (a_1-b_1,\ldots,a_n-b_n)^T\in\mathbb Z^n$$ is called the {\em $\delta$-vector} of $M$ and the vector ${\bf g}_M \coloneqq   -\delta_M$ is called the {\em $g$-vector} of $M$.

For a $\tau$-rigid pair $(M,P)$, we define its  $\delta$-vector and $g$-vector as follows:
\[ 
    \delta_{(M,P)} \coloneqq   \delta_M-\delta_P,\;\; {\bf g}_{(M,P)} \coloneqq   -\delta_{(M,P)}.
\]
With this definition, we can see that the $g$-vector ${\bf g}_{(0,P_k)}$ of $(0,P_k)$ is the $k$th column of $I_n$.
\begin{remark}
    Notice that the $\delta$-vectors defined here coincide with the $g$-vectors used in \cite{air_2014}. For the considerations on the cluster algebra side, $g$-vectors defined here are the negative of the $\delta$-vectors. 
\end{remark}

\begin{definition}The $F$-polynomial $F_M$ of $M\in\mod A$ is defined to be
\[ F_M=\sum_{{\bf v}\in\mathbb N^n}\chi(\Gr_{\bf v}(M)){\bf y}^{\bf v}\in\mathbb Z[y_1,\ldots,y_n],
\]
where $\Gr_{\bf v}(M)$ is the quotient module Grassmannian of $M$ with dimension vector ${\bf v}$ and $\chi$ is the Euler–Poincar\'{e} characteristic.
\end{definition}

Since the zero module is a quotient module of $M$, the polynomial $F_M$ has constant term $1$. So we have 
\[F_M[{\bf r}]=\max\{ {\bf v}^T{\bf r}\mid \chi(\Gr_{\bf v}(M))\neq 0\}
\geq 0,\quad \forall {\bf r}\in\mathbb Z^n.
\] 
For a $\tau$-rigid pair $(M,P)$, we define its  $F$-polynomial as follows:
\[ F_{(M,P)} \coloneqq   F_M\in\mathbb Z[y_1,\ldots,y_n].\]
In particular, we have $F_{(0,P_k)}=1$ for $k=1,\ldots,n$.
\begin{remark}\label{rmk:Q-B}
Notice that in Section \ref{sec:E-invariant} we define $F$-polynomial using submodules (see Definition  \ref{def:g-F-E}), while in this section we define  $F$-polynomial using quotient modules. The reason is that in this section we want to work with torsion classes (corresponding to $\tau$-tilting theory) rather than torsion-free classes (corresponding to $\tau^{-1}$-tilting theory). So the algebras/quivers in Section \ref{sec:E-invariant}
correspond to the opposite algebras/quivers in this section. For example, in Section \ref{sec:E-invariant} we use the representations of the quiver $1\leftarrow 2$ to model the cluster algebra $\mathcal A$ with initial mutation matrix $B=\begin{bmatrix}
    0&1\\ -1&0
\end{bmatrix}$, while in this section we would use the representations of the quiver $1\rightarrow 2$ to model $\mathcal A$.
\end{remark}

Denote by $\mathcal X_+$ the set of indecomposable $\tau$-rigid modules in $\mod A$ and $\mathcal X$ the  set of indecomposable $\tau$-rigid pairs in $\mod A$. We have  
\[\mathcal X=\mathcal X_+\sqcup\{(0,P_1),\ldots,(0,P_n)\}.\]
When we write $U\in\mathcal X$, we mean that $U$ is an indecomposable $\tau$-rigid pair, that is, either $U\in\mathcal X_+$ or $U=(0,P_k)$ for some $k\in[1,n]$.

Following the definition of the dominant sets in cluster algebras, we define the dominant sets in $\tau$-tilting theory.
\begin{definition}
    Let $(M,P)$ be a $\tau$-tilting pair in $\mod A$. The {\em dominant set} $\dom[M,P]$ and the {\em reduced dominant set} $\dom_+[M,P]$  of $(M,P)$ are defined as follows:
    \begin{align}
         \dom[M,P]& \coloneqq   \{U\in\mathcal X\mid F_{U}[{\bf g}_{(M,P)}]=0\},\nonumber\\
          \dom_+ [M,P]& \coloneqq   \{U\in\mathcal X_+\mid F_{U}[{\bf g}_{(M,P)}]=0\}.\nonumber
    \end{align}
\end{definition}

Since $F_{(0,P_k)}=1$  for $k=1,\ldots,n$,  the set $\{(0,P_1),\ldots,(0,P_n)\}$  is contained in the dominant set $\dom[M,P]$ for any $\tau$-tilting pair $(M,P)$. Thus we have
\begin{eqnarray}\label{eqn:dom}
    \dom[M,P]=\dom_+[M,P]\sqcup\{(0,P_1),\ldots,(0,P_n)\}.
\end{eqnarray}

For a module $X\in \mod A$ and a vector $\delta\in\mathbb R^n$, we write $\delta(X)=\langle\delta, \underline \dim X\rangle\in\mathbb R$ for the inner product of $\delta$ and the dimension vector of $X$.

\begin{lemma}[{\cite{BKT-2014}*{Prop. 3.1}, \cite{asai-2021}*{Prop. 3.11}}] 
\label{lem:delta}
Let $\delta\in\mathbb R^n$ and set $$\overline{\mathcal T}_\delta:=\{ N\in\mod A\mid \text{for any quotient module X of }N,\;\; \delta(X)\geq 0
     \}.$$ 
 The following statements hold.
     \begin{itemize}
         \item [(i)]  The subcategory $\overline{\mathcal T}_\delta$ is a torsion class in $\mod A$.
         \item[(ii)]  Let $(M,P)$ be a basic $\tau$-tilting pair in $\mod A$ and $\delta=\delta_{(M,P)}$  the $\delta$-vector of $(M,P)$. Then $\Fac M=\overline{\mathcal T}_\delta$. 
     \end{itemize}
\end{lemma}

Torsion classes of the form $\overline{\mathcal T}_\delta$ are called {\em semistable torsion classes}. By Lemma \ref{lem:delta} (ii),  all functorially finite torsion classes are semistable torsion classes.

\begin{lemma}[{\cite[Theorem 1.4]{fei_2019a}}] \label{lem:polytope}
    Let $U\in \mod A$ and let $$F_U=\sum_{{\bf v}\in\mathbb N^n}\chi(\Gr_{\bf v}(U)){\bf y}^{\bf v}\in\mathbb Z[y_1,\ldots,y_n]$$ be the $F$-polynomial of $U$. Then the convex hull $\mathcal N(F_U)$ 
    of $\{{\bf v}\in\mathbb N^n\mid \chi(\Gr_{\bf v}(U))\neq 0\}$ and the convex hull $\mathcal N(U)$ of the dimension vectors of the quotient modules of $U$ are the same. 
\end{lemma}
\begin{proof}
  The corresponding result for submodules and $F$-polynomials defined using submodules is proved by Fei in \cite[Theorem 1.4]{fei_2019a}. The required result follows from  Fei's result by considering the modules over the opposite algebras.
\end{proof}

\begin{proposition}\label{pro:dom-tau}
  Let $(M,P)$ be a $\tau$-tilting pair in $\mod A$. Then we have

  \begin{itemize}
      \item [(i)] $\dom_+[M,P]=\mathcal X_+\cap \Fac M=\{\text{indecomposable }\tau\text{-rigid modules in }\Fac M\}$. \vspace{1.5mm}
      \item[(ii)] $ \dom[M,P]=\{\text{indecomposable }\tau\text{-rigid modules in }\Fac M\}\sqcup\{(0,P_1),\ldots,(0,P_n)\}$.
  \end{itemize}
\end{proposition}
\begin{proof}
(i) Let $\delta=\delta_{(M,P)}$ be the $\delta$-vector of $(M,P)$ and ${\bf g}={\bf g}_{(M,P)}=-\delta$ the $g$-vector of $(M,P)$. Let $U$ be an indecomposable $\tau$-rigid module in $\mod A$ and   $$F_U=\sum_{{\bf v}\in\mathbb N^n}\chi(\Gr_{\bf v}(U)){\bf y}^{\bf v}\in\mathbb Z[y_1,\ldots,y_n]$$  the $F$-polynomial of $U$. We know that $U\in\dom_+[M,P]$ if and only if  
\[
F_U[{\bf g}]=\max\{ {\bf v}^T{\bf g}\mid \chi(\Gr_{\bf v}(U))\neq 0\}=0.
\]
Thanks to Lemma \ref{lem:polytope}, this is equivalent to
\[\max\{ \langle{\bf g}, \underline \dim X\rangle
\mid X \text{ is a quotient module of }U\}=0,
\]
equivalently, $\langle{\bf g}, \underline \dim X\rangle\leq  0$ for any quotient module $X$ of $U$. This is also equivalent to $\delta(X)=-\langle{\bf g}, \underline \dim X\rangle\geq 0$ for any quotient module $X$ of $U$. Then by Lemma \ref{lem:delta}, this is equivalent to $U\in \Fac M$.
Hence, we have 
 $$\dom_+[M,P]=\mathcal X_+\cap \Fac M=\{\text{indecomposable $\tau$-rigid modules in $\Fac M$}\}.$$

 (ii) This follows from (i) and the equality \eqref{eqn:dom}.
\end{proof}

Since $\Fac M$ is the smallest torsion class containing all the indecomposable $\tau$-rigid modules in $\Fac M$, we know that $\Fac M$ is uniquely determined by the dominant set $\dom[M,P]$. Thus we can view the dominant sets  as a kind of replacement of torsion classes in $\tau$-tilting theory.

\begin{example}
    Let $A$ be the path algebra of the quiver $1\rightarrow 2$. We have the exact sequence 
    \[0\rightarrow P_2\rightarrow P_1\rightarrow S_1\rightarrow 0\] in $\mod A=\add (P_1\oplus P_2\oplus S_1)$.
The Hasse quiver of the poset of functorially torsion classes in $\mod A$ is as follows:
\[
\xymatrix{
&\mathcal T_0=\{0\}\ar[ld]\ar[rd]&\\
\mathcal T_1=\add (S_1) \ar[d]&&\mathcal T_4=\add (P_2) \ar[d]\\
\mathcal T_2=\add(S_1\oplus P_1)\ar[rr]&&\mathcal T_3=\mod A}
\]
The  oriented exchange graph of  $\tau$-tilting pairs is as follows:
\[
\xymatrix{
&(0,P_1\oplus P_2)\ar[ld]_{\mu_{(0,P_1)}}\ar[rd]^{\mu_{(0,P_2)}}&\\
(S_1,P_2) \ar[d]_{\mu_{(0,P_2)}}&&(P_2,P_1) \ar[d]^{\mu_{(0,P_1)}}\\
(S_1\oplus P_1,0)\ar[rr]^{\mu_{(S_1,0)}}&&\mathcal (P_1\oplus P_2,0)}
\]
The corresponding dominant sets are as follows:
\begin{align}
   &\dom[0,P_1\oplus P_2]=\{(0,P_1),\;(0,P_2)\},\nonumber\\
   &\dom[S_1, P_2]=\{(0,P_1),\;(0,P_2), \;S_1\},\nonumber\\
   &\dom[S_1\oplus P_1,0]=\{(0,P_1),\;(0,P_2),\;S_1,\;P_1\},\nonumber\\
   &\dom[P_1\oplus P_2,0]=\{(0,P_1),\;(0,P_2),\;S_1,\;P_1,\;P_2\},\nonumber\\
   &\dom[P_2,P_1]=\{(0,P_1),\;(0,P_2),\;P_2\}.\nonumber
\end{align}
It is easy to check that 
\begin{gather}
F_{S_1}=1+y_1,\;\;F_{P_1}=1+y_1+y_1y_2,\;\;F_{P_2}=1+y_2,\nonumber\\
   {\bf g}_{S_1}=\begin{bmatrix}
       -1\\ 1
   \end{bmatrix},\;\;
   {\bf g}_{P_1}=\begin{bmatrix}
       -1\\ 0
   \end{bmatrix},\;\;
   {\bf g}_{P_2}=\begin{bmatrix}
       0\\ -1
   \end{bmatrix}.\nonumber
\end{gather}
From the viewpoint of categorification of cluster algebras, the indecomposable $\tau$-rigid modules $S_1, P_1, P_2$ correspond to the non-initial cluster variables  $$x_3={x_1^{-1}x_2\cdot(1+\widehat y_1),}\;\;\;  x_4={x_1^{-1}\cdot (1+\widehat y_1+\widehat y_1\widehat y_2),}\;\;\;
    x_5={x_2^{-1}\cdot (1+\widehat y_2)}$$ 
    in Example \ref{ex:A2-dom}.
\end{example}

\section*{Acknowledgements}
 I would like to express my gratitude to Antoine De Saint Germain, Jiang-Hua Lu, and Jie Pan for their helpful discussions during my postdoctoral stay at The University of Hong Kong. I am also thankful to Bernhard Keller for his email exchanges and to Jiarui Fei for his valuable feedback on the first version of this manuscript. I am particularly indebted to Ryo Fujita for his helpful discussions, comments and remarks during my writing the part on monoidal cluster categorification.
 
 This project, which began in 2023,  has been partially supported  by grants from the National Key R\&D Program of China (2024YFA1013801), the Research Grants Council of the Hong Kong SAR, China (GRF 17307718, GRF 17306621, GRF 17303420),
 the National Natural Science Foundation of China (Grant No.\ 12071422), and the Guangdong Basic and Applied Basic Research Foundation (Grant No.\ 2021A1515012035).

\bibliographystyle{alpha}
\bibliography{myref}

\end{document}